\documentclass[11pt]{article}

\usepackage{amsmath, amsfonts, amssymb, amsthm}
\usepackage{mathtools}
\usepackage{bbm}

\usepackage{enumerate}
\usepackage{graphicx}
\usepackage{color}
\usepackage{bm}
\usepackage{appendix}
\usepackage{tikz}
\usepackage{chngcntr}
\DeclareMathOperator{\prox}{prox}
\DeclareMathOperator{\env}{env}
\DeclareMathOperator{\argmin}{argmin}
\DeclareMathOperator{\sign}{sgn}
\DeclareMathOperator{\ri}{ri}
\DeclareMathOperator{\dom}{dom}
\DeclareMathOperator{\bd}{bd}
\DeclareMathOperator{\interior}{int}

\newtheorem{theorem}{Theorem}
\newtheorem{corollary}{Corollary}
\newtheorem{lemma}{Lemma}

\newtheorem{definition}{Definition}
\newtheorem{proposition}{Proposition}

\theoremstyle{remark}
\newtheorem{remark}{Remark}

\usepackage[margin=1in]{geometry}

\newcommand*\samethanks[1][\value{footnote}]{\footnotemark[#1]}

\newcommand{\dbar}{d\hspace*{-0.08em}\bar{}\hspace*{0.1em}}

\newcommand{\vertiii}[1]{{\left\vert\kern-0.25ex\left\vert\kern-0.25ex\left\vert #1
    \right\vert\kern-0.25ex\right\vert\kern-0.25ex\right\vert}}

\newcommand\T{\rule{0pt}{2.6ex}}       % Top strut
\newcommand\B{\rule[-2ex]{0pt}{0pt}}

\providecommand{\keywords}[1]
{
  \textbf{\textit{Keywords---}} #1
}

\begin{document}

\title{Structured Sparsity Promoting Functions}

\author{Lixin Shen\thanks{Department of Mathematics, Syracuse University, Syracuse, NY 13244, USA. Email: \texttt{lshen03@syr.edu} and \texttt{eetripp@syr.edu}} \and {Bruce W. Suter}\thanks{Air Force Research Laboratory, Rome, NY. Email: \texttt{bruce.suter@us.af.mil.}}
\and Erin E. Tripp\samethanks[1]}

\maketitle
\begin{abstract}
Motivated by the minimax concave penalty based variable selection in high-dimensional linear regression, we introduce a simple scheme to construct structured semiconvex sparsity promoting functions from convex sparsity promoting functions and their Moreau envelopes. Properties of these functions are developed by leveraging their structure. In particular, we provide sparsity guarantees for the general family of functions. We further study the behavior of the proximity operators of several special functions including indicator functions of closed convex sets, piecewise quadratic functions, and the linear combinations of them. To demonstrate these properties, several concrete examples are presented and existing instances are featured as special cases.
\end{abstract}

\keywords{Moreau envelope, proximity operator, variable selection, sparsity, thresholding operator}

%%%%introduction%%%%

\section{Introduction}
Natural signals and data streams are often inherently sparse in certain bases or dictionaries where they can be approximately represented by only a few significant components carrying the most relevant information \cite{Candes-Tao:IEEE-TIT:06,mallat:99,Suter:97}. Regularization methods are a powerful tool for sparse modeling and have been widely used to analyze these data sets. A particular method depends on the choice of penalty used to enforce constraints on the objective. The natural penalty function to promote sparsity is the so-called $\ell_0$-norm, which counts the nonzero components of a vector. However minimizing the $\ell_0$-norm is a combinatorial optimization problem which is known to be NP-hard.

To overcome these computational difficulties, regularization methods with the $\ell_1$-norm as its penalty function like LASSO \cite{Tibshirani:jrss:96} and Dantzig selectors \cite{Candes-Tao:IEEE-TIT:06} have been proposed. The convexity of the $\ell_1$-norm makes the implementation of the corresponding methods numerically tractable. However, despite its appealing properties, convex regularization methods can suffer from the bias issue that is inherited from the convexity of the penalty function. To address this, non-convex penalties including the $\ell_q$-penalty with $0<q<1$ \cite{Frank-Friedman:Technometrics:1993}, the smoothly clipped absolute deviation penalty (SCAD) \cite{Fan-Li:JASA:01} and the minimax concave penalty (MCP) \cite{Zhang:AS:2010} have been proposed to replace the $\ell_1$-norm penalty.
%need to explain bias more
%add transition

In this paper, we introduce a family of semiconvex sparsity promoting functions of which each is the difference of a convex sparsity promoting function with its Moreau envelope. Roughly speaking, a sparsity promoting function is one that admits its global minimum at the origin but is nondifferentiable there; a function is semiconvex if it becomes convex after adding a convex quadratic function to it. Semiconvex functions possess useful structure and obey generalizations of many classical results from convex analysis (see, e.g. \cite{Bolte-Daniilidis-Ley-Mazet:TAMS:2010}).

We show that as long as a convex function is a sparsity promoting function, so is the difference between it and its Moreau envelope. This result makes the construction of nonconvex sparsity promoting functions effortless.  Some interesting properties of such functions are: (i) they are always nonnegative and semiconvex and (ii) they are a special type of difference of convex (DC) functions with one having a Lipschitz continuous gradient. Due to these properties, we will refer to these functions as structured semiconvex sparsity promoting functions. These properties enable us to make use of the fruitful results, for example, in DC programming \cite{Tao-Le:AMV:1997}, to develop efficient algorithms for the associated regularized optimization problems. What's more, these functions provide a bridge between convex and nonconvex sparsity promoting penalties. As a specific example, we recover the MCP from the difference of the $\ell_1$-norm and its envelope. It has been shown (e.g. in \cite{Soubies-Blanc-Feraud-Aubert:SIAMOPT:2017}) that this closely approximates the $\ell_0$-norm while preserving the continuity and subdifferentiability of $\ell_1$.

The proximity operator, which was first introduced by Moreau in \cite{moreau:BSMF:65} as a generalization of the notion
of projection onto a convex set, has been used extensively in nonlinear optimization (see, e.g., \cite{Attouch-Bolte-Svaiter:MP:13,Bauschke-Combettes:11,Combettes-Wajs:MMS:05}).  The desired features of the aforementioned regularization methods can be explained in terms of the proximity operators of the corresponding penalties. Therefore to determine the effectiveness of our proposed functions, we must examine the behavior of their proximity operators. The proximity operator of the $\ell_0$-norm is the hard thresholding operator, which annihilates all entries below a certain threshold and keeps all entries above the threshold. In fact, we see that hard thresholding rules are characteristic of penalties which are concave near the origin and constant elsewhere. More generally, we provide sparsity guarantees in terms of thresholding behavior for the entire family of structured semiconvex sparsity promoting functions, with further details for certain special functions.

The rest of the paper is organized as follows.  Section~\ref{sec:Motivation} provides motivation for the suggested scheme. Section~\ref{sec:CSPF} recalls some necessary background in optimization and introduces the concept of the sparsity promoting function. In Section~\ref{sec:NSPF}, we construct a family of semiconvex sparsity promoting functions which are the difference of convex sparsity promoting functions and their Moreau envelopes. Many interesting properties of this family of functions are presented. In Section~\ref{sec:Special}, several special sparsity promoting functions are presented and discussed thoroughly. Some examples of practical interest are provided in Section~\ref{sec:Examples}. We conclude by discussing applications and plans for future work in Section~\ref{sec:concluions}.

%%%%%%%%%%%%%%%%%%%%%%%%%%%%%%%%%%%%%%%%%%%%%%%%%%%%%%%%%%
\section{Motivation} \label{sec:Motivation}
%%%%%%%%%%%%%%%%%%%%%%%%%%%%%%%%%%%%%%%%%%%%%%%%%%%%%%%%%%
Our work on semiconvex sparsity promoting functions was motivated mainly by the minimax concave penalty (MCP) based variable selection in high-dimensional linear regression \cite{Zhang:AS:2010}. Variable selection is fundamental in statistical analysis
of high-dimensional data.  It is also easily interpretable in terms of sparse signal recovery. We consider a linear regression model with $n$-dimensional response vector $y$, $n \times p$ model matrix $X$, $p$-dimensional regression vector $\gamma$, and $n$-dimensional error vector $\epsilon$:
$$
y = X \gamma +\epsilon.
$$
The goal of variable selection is to recover the true underlying sparse model of the pattern $\{j: \gamma_j \neq 0\}$ and to estimate the non-zero regression coefficients $\gamma_j$, where $\gamma_j$ is the $j$-th component of $\gamma$. For small $p$, subset selection methods
can be used to find a good guess of the pattern (see, e.g., \cite{Schwarz:AS:1978}). However,  subset selection becomes computationally infeasible for large $p$.

To overcome the computational difficulties of subset selection method, the method of penalized least squares is widely used in variable selection to produce meaningful interpretable models:
\begin{equation}\label{model:regression}
\min_{\gamma \in \mathbb{R}^p} \frac{1}{2n}\|y-X \gamma\|^2 + \sum_{j=1}^p \rho(|\gamma_j|,\lambda),
\end{equation}
where $\rho(\cdot,\lambda)$ is a penalty function indexed by $\lambda \ge 0$.  The penalty function $\rho(t,\lambda)$, defined on $[0, \infty)$, is assumed to be nondecreasing in $t$ with $\rho(0, \lambda)=0$ and continuously differentiable for $t \in (0, \infty)$. The formulation in \eqref{model:regression} includes many  popular variable selection methods. For example, the best subset selection amounts to using the $\ell_0$ penalty $\rho(|t|, \lambda)=\frac{\lambda^2}{2} \mathbbm{1}_{\{|t| \neq 0\}}$ while LASSO \cite{Tibshirani:jrss:96} and basis pursuit \cite{Chen-Donoho-Saunders:SIAMSC:98} use  the $\ell_1$ penalty $\rho(|t|, \lambda)=\lambda |t|$. Here $\mathbbm{1}_{\{u \in E\}}$ denotes the characteristic function and $\mathbbm{1}_{\{u \in E\}}$ equals $1$ if $u \in E$ and $0$ otherwise. The estimator (i.e, the hard thresholding operator) with the $\ell_0$ penalty suffers from instability in model prediction while the estimator (the soft thresholding operator) with the $\ell_1$ penalty suffers from the bias issue, severely interfering with variable selection for large $p$ \cite{Fan-Li:JASA:01}. To remedy this issue, the SCAD penalty was introduced in \cite{Fan-Li:JASA:01}. The estimator with the SCAD penalty is continuous and leaves large components not excessively penalized. In \cite{Zhang:AS:2010}, the MCP penalty was introduced and is defined as follows
\begin{equation}\label{def:MCP}
\rho(|t|, \lambda) = \lambda \int_0^{|t|} \max\left\{0,1-\frac{x}{a\lambda}\right\} \; dx,
\end{equation}
where the parameter $a>0$. This penalty function (see \cite{Zhang:AS:2010}) minimizes the maximum concavity
$$
\kappa(\rho, \lambda) := \sup_{0<t_1<t_2<\infty}-\frac{\rho(t_2, \lambda)-\rho(t_1, \lambda)}{t_2-t_1}
$$
subject to the unbiasedness $\frac{\partial}{\partial t} \rho(t, \lambda)=0$ for all $t>a\lambda$ and selection features $\frac{\partial}{\partial t} \rho(0+, \lambda)=\lambda$. The number $\kappa(\rho, \lambda)$ is related to the computational complexity of regularization method for solving \eqref{model:regression}. The simulations in \cite{Fan-Li:JASA:01,Zhang:AS:2010} gave a strong statistical evidence that the estimators from the non-convex penalty functions SCAD and MCP are useful in variable selection. Recently, an application of MCP to signal processing was reported in \cite{Selesnick:IEEESPL:2017}.

Due to its success in applications, we take a closer look at MCP. The MCP function in \eqref{def:MCP} can be rewritten as
$$
\rho(|t|, \lambda)  = \lambda (|t|-\env_{a \lambda |\cdot|}(t)),
$$
where $\env_{a \lambda |\cdot|}$ is the Moreau envelope of $|\cdot|$ with index $a\lambda$ (see next section). Clearly, the MCP penalty can be considered as a variation of the $\ell_1$ penalty function, that is, the absolution function $|\cdot|$ is replaced by $|\cdot|-\env_{a \lambda |\cdot|}$. From this simple observation, we are drawn to consider a family of penalty functions defined by
$$
f-\env_{\alpha f}
$$
with $f$ satisfying some proper properties and $\alpha>0$.

The goal of this paper is to have a comprehensive study on mathematical properties of this family of functions, particularly their proximity operators that are closely related to selection features when adopted in \eqref{model:regression}.

%and is defined as $p(0, \lambda)=0$ and
%$$
%\frac{\partial}{\partial t} \rho(t, \lambda)= \lambda \mathbbm{1}_{\{t \neq \lambda\}} + \frac{(a\lambda-t)_{+}}{(a-1)\lambda}\mathbbm{1}_{\{t>\lambda\}}
%$$
%for $t>0$ some $a>2$.

%%%%%%%%%%%%%%%%%%%%%%%%%%%%%%%%%%%%%%%%%%%%%%%%%%%%%%%%%%
\section{Sparsity Promoting Functions: Definition} \label{sec:CSPF}
%%%%%%%%%%%%%%%%%%%%%%%%%%%%%%%%%%%%%%%%%%%%%%%%%%%%%%%%%%
In this section, we provide a formal definition of sparsity promoting and characterize convex sparsity promoting functions. We begin by collecting the necessary definitions and facts from convex analysis.

All functions in this work are defined on $\mathbb{R}^n$, Euclidean space equipped with the standard inner product $\langle \cdot, \cdot \rangle$ and the induced Euclidean norm $\|\cdot\|$. We use $\Gamma(\mathbb{R}^n)$ (respectively $\Gamma_0(\mathbb{R}^n)$) to represent the set of proper lower semicontinuous (respectively convex) functions on  $\mathbb{R}^n$. The domain of an operator $A$ (respectively a function $g$) is denoted $\mathrm{dom} (A)$ (respectively $\mathrm{dom} (g)$). The boundary of a set $S$ denoted by $\bd(S)$ is the set of points in the closure $\bar{S}$ which are not in the interior $\interior(S)$. The relative interior of a set $S$ denoted by $\ri(S)$ is the interior of $S$ when it is viewed as a subset of the affine space it spans. For any $x \in \mathbb{R}^n$ and any $\delta > 0$, we use $B_\delta(x)$ to denote the open ball centered at $x$ with radius $\delta$. In particular, we are interested in $B_{\|x\|}(x) = \{ u : \|u -x\| < \|x\|\}.$  For a real number $a$,, the signum function $\sign(a)$ is defined as 
$$
\sign(a)=\left\{
           \begin{array}{ll}
             -1, & \hbox{if $a<0$;} \\
             0, & \hbox{if $a=0$;} \\
             1, & \hbox{if $a>0$.}
           \end{array}
         \right.
$$

%when it is viewed as a subset of the affine space it spans.

For any $g \in \Gamma(\mathbb{R}^n)$, the  Fr\'{e}chet subdifferential  of $g$ at $x \in \mathrm{dom} (g)$ is the set
\begin{equation*}
\partial g(x) \coloneqq \left\{d \in \mathbb{R}^n : \liminf_{u \to x} \frac{g(u) - g(x) - \langle d, u - x \rangle}{\|u-x\|} \geq 0 \right\}.
\end{equation*}
For any $x \notin \mathrm{dom} (g)$, $\partial g(x) = \emptyset$. If $\partial g(x)$ is single-valued, $\partial g(x) = \{\nabla g(x)\}$. We leave off the brackets when there is no risk of confusion. If $g \in \Gamma_0(\mathbb{R}^n)$, the above subdifferential reduces to the usual one
$$
\partial g(x)=\{d \in \mathbb{R}^n: g(y)\ge g(x) + \langle d, y-x\rangle, \forall y \in \mathbb{R}^n\}.
$$
If $g \in \Gamma_0(\mathbb{R}^n)$, then $\partial g$ is a monotone operator; that is, for any $x, y \in \dom(g)$, $d \in \partial g(x)$, and $\dbar \in \partial g(y)$,  $\langle \dbar - d, y - x \rangle \geq 0.$

For a function $g$ in $\Gamma(\mathbb{R}^n)$, the Moreau envelope of $f$ with parameter $\alpha$, denoted by $\mathrm{env}_{\alpha g}$, is
$$
\mathrm{env}_{\alpha g}(x)=\inf_{u \in \mathbb{R}^n} \left\{g(u)+\frac{1}{2\alpha}\|u-x\|^2\right\}.
$$
The associated proximity operator of $g$ with parameter $\alpha$ at $x$ is the set of all points at which the above infimum is attained, denoted by $\prox_{\alpha g}(x)$:
$$
\prox_{\alpha g}(x)=\argmin_{u \in \mathbb{R}^n} \left\{g(u)+\frac{1}{2\alpha}\|u-x\|^2\right\}.
$$
When $\prox_{\alpha g} (x) \neq \emptyset$, $\env_{\alpha g}(x) = g(p) + \frac{1}{2\alpha}\|p - x\|^2$ for all $p \in \prox_{\alpha g}(x)$.

Recall that for a proper function $g$ on $\mathbb{R}^n$, the Fenchel conjugate $g^*$ is defined as
$$
g^*(x)=\sup_{u \in \mathbb{R}^n}\{\langle u, x \rangle - g(u)\}.
$$
The Fenchel conjugate is closely related to the Moreau envelope. Indeed, it is shown in \cite{Bauschke-Combettes:11} that for any $x \in \mathbb{R}^n$ and $\alpha>0$,
\begin{equation}\label{eq:Fenchel-Envelope}
\left(g+\frac{1}{2\alpha}\|\cdot\|^2\right)^*(\alpha^{-1} x)= \left(-\env_{\alpha g }+ \frac{1}{2\alpha}\|\cdot\|^2\right)(x).
\end{equation}

We now rigorously define what is meant by sparsity promoting and discuss how this captures the behavior described in the previous section.

\begin{definition}\label{def:SPF}
Let $f \in \Gamma(\mathbb{R}^n)$. Then $f$ is said to be a sparsity promoting function provided that (i) $f(0)=0$ and $f$ achieves its global minimum at the origin; and (ii) the set $\partial f(0)$ contains at least one nonzero element.
\end{definition}

From the above definition, if $f \in \Gamma(\mathbb{R}^n)$ is a sparsity promoting function, then by Fermat's rule $0 \in \partial f(0)$ and $f$ must be nondifferentiable at the origin. As pointed out in \cite{Fan-Li:JASA:01},  the non-differentiability of $f$ at the origin is necessary for $f$ to be a suitable penalty in \eqref{model:regression} for variable selection.

One typical sparsity promoting function is the absolute value function on $\mathbb{R}$. The global minimum is $|0| = 0$, and $\partial | \cdot | (0) = [-1, 1]$. We will return to this example throughout to illustrate various properties and connect them to our motivating example MCP. Another example of a sparsity promoting function is the indicator function $\iota_C$ that is defined by
$$
\iota_C(x) \coloneqq \begin{cases} 0, & \text{if } x \in C;\\
+ \infty, & \text{ otherwise,} \end{cases}
$$
where $C$ is a closed, convex set such that $0 \in \bd C$ and $\{0\} \subsetneq C$. For further discussion of this example, we refer to Section \ref{sec:Special}. Beyond these examples, the following result shows that every norm $\vertiii{\cdot}$ on $\mathbb{R}^n$ is a sparsity promoting function.
\begin{proposition}
Let $\vertiii{\cdot}$  be a norm on $\mathbb{R}^n$. Then the norm $\vertiii{\cdot}$ is a sparsity promoting function.
\end{proposition}
\begin{proof}\ \ It is obvious that the norm $\vertiii{\cdot}$ is convex and $0=\vertiii{0}=\min_{x \in \mathbb{R}^n}\vertiii{x}$. We further know that
$$
\partial \vertiii{\cdot}(0) = \{s \in \mathbb{R}^n: \max_{\vertiii{u} \le 1} \langle s, u \rangle \le 1\},
$$
which is the unit ball associated with the dual norm of $\vertiii{\cdot}$ (see, e.g. \cite{Hiriart-Lemar:BOOK:93}). The result of this proposition follows.
\end{proof}

It is well known that the relationship between the subdifferential and proximity operator of a function $f \in \Gamma_0(\mathbb{R}^n)$ is characterized as follows (see, e.g., \cite{Bauschke-Combettes:11,Micchelli-Shen-Xu:IP-11}): for any $\alpha >0$
\begin{equation}\label{eq:diff-prox}
x \in \alpha \partial f (y) \Leftrightarrow y = \prox_{\alpha f} (x+y).
\end{equation}
From this relationship, we get the following characterization of convex sparsity promoting functions.

\begin{lemma}\label{lemma:convex-prox-sparsity}
Let $f\in \Gamma_0(\mathbb{R}^n)$ be a sparsity promoting function and let $\alpha > 0$. Then the following statements hold.
\begin{enumerate}[{\normalfont (i)}]
\item If $x \in \alpha \partial f(0)$, then $\prox_{\alpha f}(x)=0$.
\item For all $x \in \dom(f)$, $\|\prox_{\alpha f}(x)\| \le \|x\|$.
\end{enumerate}
\end{lemma}
\begin{proof}\ \
(i): This is a direct consequence of equation~\eqref{eq:diff-prox}.

(ii): Note that $\prox_{\alpha f}(0)=0$ due to $0 \in \alpha \partial f(0)$ and Item (i). Since $\prox_{\alpha f}$ is a nonexpansive operator, then for all $x \in \dom(f)$,
$
\|\prox_{\alpha f}(x)\|=\|\prox_{\alpha f}(x)-\prox_{\alpha f}(0)\| \le \|x-0\|.
$  \end{proof}

It follows from Lemma~\ref{lemma:convex-prox-sparsity} that the proximity operator of a convex sparsity promoting function shrinks all input towards the origin, and all input below a certain threshold are sent to zero. As an example, the proximity operator of $| \cdot |$ is
$\prox_{\alpha |\cdot|}(x)=\mathrm{sign}(x) \max\{|x|-\alpha, 0\}$,
which is the well-known soft thresholding operator in wavelet literature \cite{donoho:biometrika:94}. This very behavior for the $\ell_1$-penalty is described by Tibshirani in the name LASSO: least absolute shrinkage and selection operator \cite{Tibshirani:jrss:96}.
%%%%%%%%%%%%%%%%Main%%%%%%%%%%%%%%%%%%%%

\section{Semiconvex Sparsity Promoting Functions}\label{sec:NSPF}

In this section, we introduce the titular family of semiconvex sparsity promoting functions. For any $f\in \Gamma_0(\mathbb{R}^n)$ and any positive number $\alpha>0$, we define
\begin{equation}\label{def:falpha}\tag{$\mathcal{F}_\alpha$}
f_\alpha(x):= f(x)-\env_{\alpha f}(x).
\end{equation}
Clearly $f_\alpha \in \Gamma(\mathbb{R}^n)$ is the difference of two convex functions. Returning to the example $f(x) = |x|$, we see that $f_\alpha(x) = \min\{|x| - \frac{1}{2\alpha}x^2, \frac{\alpha}{2}\}$. As discussed in the previous section, this is the scaled the minimax concave penalty (MCP) given in \cite{Zhang:AS:2010}.

Sparsity promotion depends entirely on the behavior of a function and its subdifferential at the origin. Since the Moreau envelope of any function $f$ in $\Gamma_0(\mathbb{R}^n)$ is differentiable (see, e.g. \cite{Bauschke-Combettes:11}), the subdifferentials of $f_\alpha$ and $f$ are related as follows (see \cite{Clarke:83}):
\begin{equation}\label{eq:sub-f-falpha}
\partial f_\alpha (x) = \partial f (x) - \nabla \env_{\alpha f}(x).
\end{equation}
Due to this inherent relationship between $\partial f_\alpha$ and $\partial f$, we see immediately that $f_\alpha$ must be sparsity promoting if $f$ is.

\begin{theorem}\label{thm:sparsity}
Let $f \in \Gamma_0(\mathbb{R}^n)$ be a sparsity promoting function. Then the following statements hold:
\begin{itemize}
\item[(i)] For any $\alpha>0$, the function $f_\alpha$ defined by \eqref{def:falpha} is  a sparsity promoting function. Moreover, $\partial f_\alpha(0)=\partial f(0)$;
\item[(ii)] Let $g: x \longmapsto f(-x)$. Then both $g$ and $g_\alpha$ are sparsity promoting. Moreover, $g_\alpha = f_\alpha (-\cdot)$ and $\partial g_\alpha(0)=-\partial f(0)$.
\end{itemize}
\end{theorem}

\begin{proof}\ \
(i): As a direct consequence of the definition of the Moreau envelope, $\env_{\alpha f}(x) \le f(x)$ for all $x \in \mathbb{R}^n$, hence $f_\alpha(x) \ge 0$ for all $x \in \dom(f)$. Since $f$ is a sparsity promoting function, we have $f_\alpha(0)=f(0)-\env_{\alpha f}(0)=0$. Therefore, $\min_{x \in \mathbb{R}^n} f_\alpha(x)=f_\alpha(0)=0$. On the other hand, from \eqref{eq:sub-f-falpha} and the relation $\nabla \env_{\alpha f}(x)=\frac{1}{\alpha} (x-\prox_{\alpha f}(x))$, we have $\partial f_\alpha(0)=\partial f(0)$ which contains at least one nonzero element by assumption. Hence, $f_\alpha$ is sparsity promoting.

(ii): Since $g(0)=f(0)=\min_{x\in \mathbb{R}^n} f(x) = \min_{x\in \mathbb{R}^n} g(x)$ and $\partial g(0)=-\partial f(0)$, so $g$ is sparsity promoting. Hence, by Item (i), $g_\alpha$ is sparsity promoting and $\partial g_\alpha(0)=-\partial f(0)$. By the definition of the Moreau envelope, $\env_{\alpha g}(x)=\env_{\alpha f}(-x)$ which leads to $g_\alpha = f_\alpha (-\cdot)$.
\end{proof}

With Theorem~\ref{thm:sparsity}, we say $f_\alpha$ is a structured sparsity promoting function if $f$ is a convex sparsity promoting function. We now prove   that $f_\alpha$ is semiconvex and show how its semiconvexity depends on the convexity of $f$. We remind the reader of the definition.  A function $g\in \Gamma_0(\mathbb{R}^n)$ is $\sigma$-strongly convex if and only if there exists a constant $\sigma>0$ such that the function $g-\frac{\sigma}{2}\|\cdot\|^2$ is convex.  A function $g \in \Gamma(\mathbb{R}^n)$ is $\rho$-semiconvex if $g+\frac{\rho}{2}\|\cdot\|^2$ is convex.

\begin{proposition}\label{prop:semiconvex}
Let $f$ be a function in $\Gamma_0(\mathbb{R}^n)$. Then $f_\alpha$, defined by \eqref{def:falpha}, is $\frac{1}{\alpha}$-semiconvex. If, in addition, $f$ is $\mu$-strongly convex, then  $f_\alpha$ is $(\mu - \frac{1}{\alpha})$-strongly convex if $\mu > \frac{1}{\alpha}$, convex if $\mu = \frac{1}{\alpha}$, and  $(\frac{1}{\alpha} - \mu)$-semiconvex if $\mu< \frac{1}{\alpha}$.
\end{proposition}

\begin{proof}\ \ Write
$$
f_\alpha = f-\env_{\alpha f}=f+\left(-\env_{\alpha f}+\frac{1}{2\alpha}\|\cdot\|^2\right)-\frac{1}{2\alpha}\|\cdot\|^2.
$$
By \eqref{eq:Fenchel-Envelope}, for all $x \in \mathbb{R}^n$ we have that
\begin{equation}\label{tmp1:semiconvex}
f_\alpha(x) = f(x) +  \left(f+\frac{1}{2\alpha}\|\cdot\|^2\right)^*(\alpha^{-1} x) - \frac{1}{2\alpha}\|x\|^2,
\end{equation}
which implies that $f_\alpha$ is $\frac{1}{\alpha}$-semiconvex.

In addition, if $f$ is $\mu$-strongly convex, then there exists a convex function $g$ such that $f=g+\frac{\mu}{2}\|\cdot\|^2$. Replacing $f(x)$ in \eqref{tmp1:semiconvex} by $g(x)+\frac{\mu}{2}\|x\|^2$, we have
\begin{equation*}\label{tmp2:semiconvex}
f_\alpha(x) = g(x) +  \left(f+\frac{1}{2\alpha}\|\cdot\|^2\right)^*(\alpha^{-1} x) +\frac{1}{2}(\mu-\frac{1}{\alpha})\|x\|^2.
\end{equation*}
It follows from the above equation that $f_\alpha$ is $(\mu - \frac{1}{\alpha})$-strongly convex if $\mu > \frac{1}{\alpha}$, convex if $\mu = \frac{1}{\alpha}$, and  $(\frac{1}{\alpha} - \mu)$-semiconvex if $\mu< \frac{1}{\alpha}$.
\end{proof}

The following result is a direct consequence of Proposition~\ref{prop:semiconvex}.
\begin{corollary}\label{cor:convexity}
Let $f$ be a function in $\Gamma_0(\mathbb{R}^n)$ and let $f_\alpha$ be defined by \eqref{def:falpha}. For any given $x \in \mathbb{R}^n$ and positive parameters $\alpha$ and $\beta$, we define
\begin{equation}\label{temp:obj}
F(u)= f_\alpha(u)+\frac{1}{2\beta}\|u-x\|^2,
\end{equation}
where  $u \in \mathbb{R}^n$. Then, $F$ is $\left({\beta}^{-1}-{\alpha}^{-1}\right)$-strongly convex if $\beta<\alpha$, convex if $\beta=\alpha$, and $\left({\alpha}^{-1}-{\beta}^{-1}\right)$-semiconvex if $\beta>\alpha$. If, in addition, $f$ is $\mu$-strongly convex, then $F$ is $\left(\mu-{\alpha}^{-1}+{\beta}^{-1}\right)$-strongly convex if $\mu>{\alpha}^{-1}-{\beta}^{-1}$, convex if $\mu={\alpha}^{-1}-{\beta}^{-1}$, and $\left({\alpha}^{-1}-{\beta}^{-1}-\mu\right)$-semiconvex if $\mu<{\alpha}^{-1}-{\beta}^{-1}$.
\end{corollary}

As for convex sparsity promoting functions, we can further characterize the sparsity promotion of $f_\alpha$ by examining its proximity operator. Roughly speaking, we show that  $\prox_{\beta f_\alpha}(x)=0$ for all $x \in \min\{\alpha, \beta\} \cdot \partial f (0)$. Towards this end, we present two technical lemmas. The first is a generalization of Lemma \ref{lemma:convex-prox-sparsity}.

%By the definition of the proximity operator, $\prox_{\beta f_\alpha}(x) = \argmin_{u \in \mathbb{R}^n} F(u)$. We see that computing $\prox_{\beta f_\alpha}$ is a convex optimization problem when $\beta \leq \alpha$ and only becomes nonconvex when $\beta > \alpha$.

\begin{lemma}\label{lem:sparsity_technical2} Let $f \in \Gamma_0(\mathbb{R}^n)$ be sparsity promoting and $f_\alpha$ as defined in \eqref{def:falpha}.
\begin{enumerate}[{\normalfont (i)}] \item For any $x \in \dom(f)$,  $\prox_{\beta f_\alpha}(x) \subseteq \overline{B_{\|x\|}}(x)$.
\item If $x \in \min\{\alpha, \beta\}\cdot \partial f(0)$, then $0 \in \prox_{\beta f_\alpha}(x)$.
\end{enumerate}
\end{lemma}

\begin{proof}\ \ For a fixed $x \in \mathbb{R}^n$, define $F$ as in \eqref{temp:obj}, so that $\prox_{\beta f_\alpha}(x) = \argmin_{u \in \mathbb{R}^n} F(u).$

(i):  Since $F(0) = \frac{1}{2\beta} \|x\|^2$ and $0 \in \overline{B_{\|x\|}}(x)$, to show $\prox_{\beta f_\alpha}(x) \subseteq \overline{B_{\|x\|}}(x)$, we only need to show that for all $u \in \mathbb{R}^n \backslash \overline{B_{\|x\|}}(x)$, $F(u) > F(0)$. Actually, if $u \in \mathbb{R}^n \backslash \overline{B_{\|x\|}}(x)$, then $\|u-x\|^2 > \|x\|^2$. Since $f_\alpha$ is non-negative, it follows from \eqref{temp:obj} that
$F(u) > \frac{1}{2\beta}\|x\|^2 = F(0)$. Thus the conclusion of Item (i) holds.

(ii): To prove Item (ii), from Item (i) and $F(0) = \frac{1}{2\beta}\|x\|^2$, it suffices to show $F(u) \geq \frac{1}{2\beta}\|x\|^2$ for all $u \in \overline{B_{\|x\|}}(x)$. From the assumption of $x \in \min\{\alpha, \beta\} \cdot \partial f(0)$, we have that  for all $u \in \mathbb{R}^n$,
$
f(u) \geq \frac{1}{\min\{\alpha, \beta\}}\langle x, u \rangle.
$
Since $f(0) = 0$, we have $\env_{\alpha f}(u) \leq \frac{1}{2\alpha}\|u\|^2$ for all $u \in \mathbb{R}^n$.  Hence
$$
f_\alpha(u) \geq \frac{1}{\min\{\alpha, \beta\}}\langle x, u \rangle - \frac{1}{2\alpha}\|u\|^2.
$$
Therefore
\begin{align*}
F(u) &\geq\frac{1}{\min\{\alpha, \beta\}}\langle x, u \rangle - \frac{1}{2\alpha}\|u\|^2 + \frac{1}{2\beta}\|u-x\|^2 \\
&= \begin{cases} \big(\frac{1}{2\beta} - \frac{1}{2\alpha}\big)\|u\|^2 + \frac{1}{2\beta}\|x\|^2 , & \text{if } \beta \leq \alpha;\\
\big(\frac{1}{2\alpha} - \frac{1}{2\beta}\big)(\|x\|^2 - \|u-x\|^2) + \frac{1}{2\beta}\|x\|^2 , & \text{if } \alpha < \beta.
\end{cases}
\end{align*}
So, $F(u) \geq \frac{1}{2\beta}\|x\|^2 = F(0)$ holds for all $u \in \overline{B_{\|x\|}}(x)$. This completes the proof of the lemma.
\end{proof}

\begin{remark} From item (i) of Lemma \ref{lem:sparsity_technical2} we see for $x \in \mathbb{R}$, $\sign(x) = \sign(p)$ if $p \in \prox_{\beta f_\alpha}(x)$ and both $x$ and $p$ are simultaneously nonzero. We note that this is also true for $\prox_{\alpha f}(x)$. \end{remark}

The following technical lemma will greatly simplify the proof of Theorem~\ref{thm:major}, our main result.  While the lemma may seem strange at first glance, the conditions therein arise naturally from the computation of the proximity operator.
\begin{lemma}\label{lem:sparsity_technical1}
Let $f \in \Gamma_0(\mathbb{R}^n)$ be a sparsity promoting function and $w \in \dom (\partial f)$. If $w \in \partial f(0)$ and there exists a nonzero $\xi \in \ri (\partial f(0)) \cap \partial f(w)$, then $w = 0$.
\end{lemma}
\begin{proof}\ \ Assume that $w \neq 0$. First, since $w \in \partial f(0)$ and $f(0) = 0$, we have $f(w) \geq \|w\|^2 > 0.$

Second, since $\xi \in \partial f(0)) \cap \partial f(w)$, then $\xi \in \partial f(0)$ implies $f(0)+f^*(\xi)=\langle 0, \xi \rangle$ while $\xi \in \partial f(w)$ implies $f(w)+f^*(\xi)=\langle \xi, w \rangle$.  Hence,
\begin{equation}\label{temp:f2}
f(w) = \langle \xi, w\rangle.
\end{equation}
By the monotonicity of $\partial f$, for any $\eta \in \partial f(0)$, $\langle \xi - \eta, w\rangle \geq 0.$ Together with \eqref{temp:f2} we get
\begin{equation}\label{temp:f3}
f(w) \geq \langle \eta, w \rangle.
\end{equation}

Finally, since $\xi \in \ri(\partial f(0))$ and $\partial f(0)$ is convex, there exists $\lambda > 1$ such that $\lambda \xi \in \partial f(0)$. By \eqref{temp:f2} and \eqref{temp:f3}, we get
$$f(w) \geq \langle \lambda \xi, w \rangle = \lambda f(w)$$
which implies $f(w) \leq 0$. This is a contradiction, so $w = 0$.

%Assume that $w \neq 0$. First, since $w \in \partial f(0)$ and $f(0) = 0$, we have $f(w) \geq \|w\|^2 > 0.$ Second, from Item (i), since $\xi \in \partial f(0) \cap \partial f(w)$,
\end{proof}

Now for our main result which characterizes the sparsity promoting structure of $f_\alpha$ in terms of the sparsity threshold of its proximity operator.

\begin{theorem}\label{thm:major}
Let $f \in \Gamma_0(\mathbb{R}^n)$ be a sparsity promoting function. For any $x \in \dom(f)$, the following statements hold:
\begin{enumerate}[{\normalfont (i)}]
  \item If $\beta < \alpha$, then $\prox_{\beta f_\alpha}(x)=0$ for $x\in \beta \partial f(0)$;
  \item If $\beta = \alpha$, then $\prox_{\beta f_\alpha}(x)=0$ for $x\in \mathrm{ri}(\alpha \partial f(0))$;
  \item If $\beta > \alpha$, then $\prox_{\beta f_\alpha}(x)=0$ for $x\in \alpha \partial f(0)$.
\end{enumerate}
\end{theorem}

\begin{proof} \ \
Given $x \in \mathbb{R}^n$, define $F(u) = f_\alpha(u) + \frac{1}{2\beta}\|u-x\|^2$.

(i) We first consider the situation $\beta<\alpha$. From Corollary~\ref{cor:convexity}, we know that $F$ is $\left(\frac{1}{\beta}-\frac{1}{\alpha}\right)$-strongly convex and therefore has a unique minimizer. By Lemma \ref{lem:sparsity_technical2}, $x \in \beta \partial f(0)$ implies that $0 = \argmin_{u \in \mathbb{R}^n} F(u)$. Together these imply that $\prox_{\beta f_\alpha}(x) = 0$.\\

%Hence, for any $u,v \in \mathbb{R}^n$ and any $p \in \partial F(u), q \in \partial F(v)$, $$\langle p-q, u-v \rangle \ge \left(\frac{1}{\beta}-\frac{1}{\alpha}\right)\|u-v\|^2.$$ By the assumption that $x\in \beta \partial f(0)$ and the fact of $\partial f(0)=\partial f_\alpha(0)$, it follows that $0 \in \partial F(0)$. Now because $0 \in \partial F(w^*)$, we see that $\langle 0-0, w^*-0 \rangle \geq \left(\frac{1}{\beta}-\frac{1}{\alpha}\right)\|w^*-0\|^2$. Hence $w^*=0$ for all $x \in \beta \partial f(0)$.

(ii) Next we consider $\alpha=\beta$. From Corollary \ref{cor:convexity}, $F(u)$ is convex but not strongly, and the minimizer may no longer be unique. By Lemma \ref{lem:sparsity_technical2}, $0 \in \prox_{\beta f_\alpha}(x)$ for $x \in \alpha \partial f(0)$.

Now suppose $x \in \ri(\alpha \partial f(0))$  and let $w^*$ be an element of $\prox_{\beta f_\alpha}(x)$. To show that $w^* = 0$, by identifying $\alpha f $, $x$, and $w^*$, respectively,  as $f$, $\xi$, and $w$ in Lemma \ref{lem:sparsity_technical1}, it suffices to show that $x \in \partial (\alpha f)(w^*)$ and $w^* \in \partial (\alpha f)(0)$. By Fermat's rule, $w^* \in \prox_{\beta f_\alpha}(x)$ implies that $0 \in \partial f_{\alpha}(w^*) + \frac{1}{\beta}(w^* - x)$. As we saw earlier that $\partial f_\alpha (w^*) = \partial f(w^*) - \nabla \env_{\alpha f}(w^*)$ and $\nabla \env_{\alpha f}(w^*) = \frac{1}{\alpha}(w^* - \prox_{\alpha f}(w^*))$, this can be rewritten as
\begin{equation}\label{eq:fermat-2}
\frac{1}{\beta}x +\left(\frac{1}{\alpha}-\frac{1}{\beta}\right)w^* -\frac{1}{\alpha}\prox_{\alpha f}(w^*) \in \partial f(w^*).
\end{equation}
From~\eqref{eq:fermat-2}, we get $x - \prox_{\alpha f}(w^*) \in \partial (\alpha f)(w^*)$. Therefore the conditions $x \in \partial (\alpha f)(w^*)$ and $w^* \in \partial (\alpha f)(0)$ hold if and only if $\prox_{\alpha f}(w^*) = 0$.

Since $x \in \partial (\alpha f)(0)$, by the monotonicity of $\partial f$ we have
\begin{equation*}\label{eq:wprox}
\langle x -  \prox_{\alpha f}(w^*) - x, w^* \rangle  \geq 0.
\end{equation*}
That is, $\langle \prox_{\alpha f}(w^*), w^* \rangle \leq 0$.  But due to the nonexpansiveness of $\prox_{\alpha f}$ and the fact that $\prox_{\alpha f}(0)=0$, $$\langle \prox_{\alpha f}(w^*), w^* \rangle \geq \|\prox_{\alpha f}(w^*) \|^2.$$ This implies that $\prox_{\alpha f}(w^*)=0$. Thus by Lemma \ref{lem:sparsity_technical1}, $w^* = 0$.
%Further, the fact $\prox_{\alpha f}(w^*)=0$ yields
%\begin{equation}\label{eq:proxw=0}
%\env_{\alpha f} (w^*)=f(\prox_{\alpha f}(w^*))+\frac{1}{2\alpha}\|\prox_{\alpha f}(w^*)-w^*\|^2=\frac{1}{2\alpha}\|w^*\|^2.
%\end{equation}
%In light of Lemma \ref{lem:sparsity_technical1} and \eqref{eq:proxw=0}, we have
%\begin{align*}
%f_\alpha(w^*)+\frac{1}{2\alpha}\|w^* - x\|^2 &=f(w^*)-\env_{\alpha f} (w^*)+\frac{1}{2\alpha}\|w^*\|^2- \frac{1}{\alpha}\langle x, w^* \rangle+ \frac{1}{2\alpha}\|x\|^2\\
%&= (f(w^*) - \frac{1}{\alpha} \langle x, w^* \rangle) - (\env_{\alpha f}(w^*) - \frac{1}{2\alpha}\|w^*\|^2)
%=\frac{1}{2\alpha}\|x\|^2.
%\end{align*}
%Since $f_\alpha(0)+\frac{1}{2\alpha}\|0 - x\|^2=\frac{1}{2\alpha}\|x\|^2$,  it means that $0 \in \prox_{\beta f_\alpha}(x)$.

(iii) Finally, we consider the situation of $\beta>\alpha$. In this case, we assume that $0 \neq x \in \alpha\partial f(0)$. From Lemma \ref{lem:sparsity_technical2}, we know that $0 \in \prox_{\beta f_\alpha}(x)$. We further show that the point $0$ is the only element in $\prox_{\beta f_\alpha}(x)$.

Recall from the proof of Lemma \ref{lem:sparsity_technical2} that when $\beta > \alpha$,
$$
F(u) \geq \left(\frac{1}{2\alpha} - \frac{1}{2\beta}\right)(\|x\|^2 - \|u-x\|^2) + \frac{1}{2\beta}\|x\|^2 \geq \frac{1}{2\beta}\|x\|^2.
$$
Actually, if $w^* \in \prox_{\beta f_\alpha}(x)$, then $w^*$ must be on the boundary of $\overline{B_{\|x\|}}(x)$ and $F(w^*)=f_\alpha(w^*)+\frac{1}{2\beta} \|w^*-x\|^2=\frac{1}{2\beta} \|x\|^2$. Thus, $f_\alpha(w^*)=0$, that is, $f(w^*)=\env_{\alpha f}(w^*)$. We also know that $f(w^*)\ge \frac{1}{\alpha}\langle x, w^*\rangle$ and $\env_{\alpha f}(w^*) \le \frac{1}{2\alpha}\|w^*\|^2$. Therefore, because $2\langle x, w^*\rangle=\|w^*\|^2$, we get
$$
\env_{\alpha f}(w^*)=\frac{1}{2\alpha}\|w^*\|^2,
$$
which implies that $0=\prox_{\alpha f}(w^*)$. On the other hand, the identity $f(w^*)=\env_{\alpha f}(w^*)$ indicates $w^*=\prox_{\alpha f}(w^*)$. Therefore, $w^*=0$. This completes the proof.
\end{proof}

\begin{remark}
Item (iii) of the theorem is not tight. In fact in every example, when $\beta > \alpha$, $\prox_{\beta f_\alpha}(x) = 0$ for all $x$ in a set strictly larger than $\alpha \partial f(0)$. However, the exact form of this set depends entirely on the function in question.
\end{remark}

%%%%%%%%%%%%%at most quadratic%%%%%%%%%%%%%%%%%%
\section{Some Special Functions}\label{sec:Special}
The last section dealt primarily with behavior around the origin for general semiconvex sparsity promoting functions. In this section, we describe the structure of $f_\alpha$ on the entire domain for special classes of sparsity promoting functions, namely indicator functions, piecewise quadratic functions, and their linear combinations. The study of these particular functions is motivated by the thresholding behavior of their proximity operators.
%more here?

\subsection{Indicator Functions}
Indicator functions are commonly used to include constraints in the objective of an optimization problem. We show in this section that they are not only fixed by the mapping $f \mapsto f_\alpha$ but they are the only functions that are fixed.

Throughout, we assume $C$ is a closed convex set in $\mathbb{R}^n$ with boundary $\bd(C)$. Recall that the indicator function of $C$ is
\begin{equation}\label{def:ind}\tag{$\mathcal{I}$}
\iota_C(x) = \begin{cases} 0, & \text{if } x \in C;\\
+ \infty, & \text{ otherwise.} \end{cases} \end{equation}
We first determine when this is a sparsity promoting function.

\begin{lemma}\label{def:sparseind}The indicator function $\iota_C$ is sparsity promoting if and only if $0 \in \bd(C)$ and $\{0\} \\subsetneq C$. \end{lemma}
\begin{proof} As long as $0 \in C$, $\iota_C(0) = 0$, but to be sparsity promoting, there must also be a nonzero element in $\partial \iota_C(0)$. Recall that for any $x$, $\partial \iota_C(x)$ is the normal cone to $C$ at $x$. That is,
$$\partial \iota_C(x) =  N_C(x) \coloneqq \begin{cases} \{ u : \sup \langle C-x, u \rangle \leq 0 \}, & \text{if } x \in C\\
\emptyset, & \text{ otherwise.} \end{cases}$$
Note that for $x \in C$, the normal cone is nonempty because $\{ 0 \} \subseteq N_C(x)$.  We further recall the following result from \cite{Bauschke-Combettes:11}:
$$x \in \interior(C) \iff N_C(x) = \{ 0\}.$$
If $0 \in \bd(C)$, it follows that $N_C(x)$ is nonempty and contains a nonzero element. Conversely, if we assume $N_C(0)$ is nonempty, we must have $0 \in C$. If we further assume that $N_C(0)$ contains a nonzero element, then $0 \not\in \interior(C)$. So we see that $0 \in \bd(C)$ is equivalent to the sparsity promoting definition given in Section \ref{sec:CSPF}.
\end{proof}

It is well known (see, e.g. \cite{Bauschke-Combettes:11}) that $\prox_{\alpha \iota_C}(x) = P_C(x)$ and that $p = P_C(x)$ if and only if $x-p \in N_C(p)$. Here $P_C (x)$ is the unique operator such that $\|x-P_C(x)\|$ is the distance from $x$ to $C$. In terms of the proximity operator, this becomes $0 = \prox_{\alpha \iota_C}(x)$ if and only if $x \in N_C(0)$. Moreover $\env_{\alpha \iota_C}(x) = \frac{1}{2\alpha}\|P_C(x) - x\|^2$ and
\begin{equation}\tag{$\mathcal{I}_\alpha$}
(\iota_C)_\alpha(x) \coloneqq \iota_C(x) - \env_{\alpha \iota_C}(x) = \iota_C(x).
\end{equation} This immediately implies that $\prox_{\beta (\iota_C)_\alpha}(x) = P_C(x)$ as well. The converse of the above is also true.

\begin{proposition}\label{prop:indicator}
Let $f \in \Gamma_0(\mathbb{R}^n)$ be sparsity promoting. If $f = f_\alpha$ as defined by \eqref{def:falpha}, then $f = \iota_{\dom (f)}$.
\end{proposition}

\begin{proof}
Notice that $\dom(\env_{\alpha f}) = \mathbb{R}^n$ so $\dom(f_\alpha) = \dom(f)$. Hence $f = f_\alpha$ implies that $\env_{\alpha f}(x) = 0$ for all $x \in \dom(f)$. Because $f$ is sparsity promoting, $f(x) \geq 0$ for all $x$. Hence, $0=\env_{\alpha f}(x) = \min_{u \in \mathbb{R}^n}\{f(u) + \frac{1}{2\alpha}\|u-x\|^2\}$ for all $x \in \dom(f)$ implies that $f(x) = 0$ for all $x \in \dom(f)$.
\end{proof}

\begin{remark} The proposition is true more generally if $f \in \Gamma_0(\mathbb{R}^n)$ is simply nonnegative. \end{remark}

%%%%%%%%%%%%%%%%%%%%%%%%%%%%%%%%%%%%%%%%%%%%%%%%%%%%%%%%%%%%%
\subsection{Piecewise Quadratic Functions} \label{subsec:quadratic}
%%%%%%%%%%%%%%%%%%%%%%%%%%%%%%%%%%%%%%%%%%%%%%%%%%%%%%%%%%%%%

Piecewise quadratic functions include a variety of important examples: absolute value, ReLU (rectified linear unit), and elastic net. We generalize the proximity-related properties of these functions and provide a framework for generating customized penalty functions.

The piecewise quadratic functions we consider here have the following form
\begin{equation}\label{def:quad}\tag{$\mathcal{Q}$}
f(x) = \begin{cases} \frac{1}{2}a_1x^2 + b_1x, & \text{if } x \leq 0;\\
\frac{1}{2}a_2x^2 + b_2x, & \text{if } x \ge 0,
\end{cases}
\end{equation}
where the coefficients $a_1$, $a_2$, $b_1$, and $b_2$ are real numbers.  The characterization of sparsity promoting functions having a form given \eqref{def:quad} is established in the following lemma.
\begin{lemma}\label{lemma:NS-conditions-quadratic}
Let $f$ be a piecewise quadratic function defined by \eqref{def:quad}. Then $f$ is sparsity promoting if and only if
\begin{equation}\label{eq:requirements}
a_1 \ge 0, \quad a_2\ge 0, \quad b_1 \le 0\le b_2, \quad \mbox{and} \quad b_2-b_1>0.
\end{equation}
\end{lemma}
\begin{proof}\ \
``$\Rightarrow$'': Since $f$ is sparsity promoting, then the assumption that $f$ attains its minimum at $0$ implies that $a_1 \ge0$, $a_2\ge 0$, $b_1\le 0$, and $b_2\ge 0$. One can directly verify that $\partial f(0)=[b_1,b_2]$. This must contain at least one nonzero element, hence, $b_2-b_1>0$.

``$\Leftarrow$'': One can see that $f$ is nonincreasing on $(-\infty, 0]$ from $a_1 \ge 0$ and $b_1\le 0$ and that  $f$ is nondecreasing on $[0,\infty)$ from $a_2 \ge 0$ and $b_2\ge 0$. So $f$ achieves its global minimum at $0$. The condition $b_2-b_1>0$ implies that the set $\partial f(0)=[b_1,b_2]$ has nonzero elements. Therefore, $f$ is a sparsity promoting function.
\end{proof}

\begin{remark} As a by-product of the above lemma, if $f$ given by \eqref{def:quad} is a sparsity promoting function, then $f$ must be convex, hence $f \in \Gamma_0(\mathbb{R})$. \end{remark}

In the rest of this section, we always assume that the coefficients in  \eqref{def:quad} satisfy the conditions listed in \eqref{eq:requirements}. The proximity operator and Moreau envelope of $f$ with index $\alpha$ at $x \in \mathbb{R}$ are
$$\prox_{\alpha f}(x) = \begin{cases} \min\big\{ 0, \frac{1}{\alpha a_1 + 1}(x - \alpha b_1)\big\}, & \text { if } x \leq 0;\\
\max\big\{0, \frac{1}{\alpha a_2 + 1}(x- \alpha b_2) \big\}, & \text{if } x \ge 0; \end{cases}$$
and
$$\env_{\alpha f}(x) = \begin{cases} \frac{1}{\alpha a_1 + 1}(f(x) - \frac{\alpha b_1^2}{2}), & \text{if } x \leq \alpha b_1;\\
\frac{1}{2\alpha}x^2, & \text{if } \alpha b_1 \leq x \leq \alpha b_2;\\
\frac{1}{\alpha a_2 + 1}(f(x) - \frac{\alpha b_2^2}{2}), & \text{if } x \ge \alpha b_2.
\end{cases}$$
respectively.  From the above two equations, we get
\begin{equation}\label{def:falpha_quad}\tag{$\mathcal{Q}_\alpha$}
f_\alpha(x) = \begin{cases}
\frac{\alpha a_1}{\alpha a_1 + 1}f(x) + \frac{ \alpha b_1^2}{2(\alpha a_1+ 1)}, & \text{if } x \leq \alpha b_1;\\
f(x) - \frac{1}{2\alpha}x^2, & \text{if } \alpha b_1 \leq x \leq \alpha b_2;\\
\frac{\alpha a_2}{\alpha a_2 + 1}f(x) + \frac{ \alpha b_2^2}{2(\alpha a_2 + 1)}, & \text{if } x \ge \alpha b_2,
\end{cases}
\end{equation}
which is a piecewise quadratic polynomial with possible breakpoints at $\alpha b_1$, $0$, and $\alpha b_2$. We know this $f_\alpha$ is sparsity promoting by Theorem~\ref{thm:sparsity}. Some other properties of this function which follow immediately from \eqref{def:falpha_quad} are collected in the following lemma.

\begin{lemma}\label{lem:sparsequadratic} Let $f \in \Gamma_0(\mathbb{R})$ be a sparsity promoting function defined by \eqref{def:quad}. Then the following hold:
\begin{itemize}
\item[(i)] $f_\alpha$ is nonincreasing on $(-\infty, 0]$ and is nondecreasing on $[0, \infty)$;
\item[(ii)] $f_\alpha$ on $(-\infty, \alpha b_1]$ is convex and is a degree $2$ polynomial if $a_1>0$ or constant if $a_1=0$;
\item[(iii)] $f_\alpha$ on $[\alpha b_2, \infty)$ is convex and is a degree $2$ polynomial if $a_2>0$ or a constant if $a_2=0$;
\item[(iv)] $f_\alpha$ on $[\alpha b_1, \alpha b_2]$ is convex if $\min\{a_1,a_2\} \ge \frac{1}{\alpha}$.
\end{itemize}
\end{lemma}

Just as the sparsity promoting property corresponds to certain behavior in the proximity operator near the origin, this result in Lemma~\ref{lem:sparsequadratic}  guarantees special properties of the proximity operator away from the origin. To illustrate, we return to $f(x) = |x|$. This satisfies \eqref{def:quad} with $a_1=a_2=0$, $b_1 = -1$, and $b_2 = 1$. We saw in Section \ref{sec:NSPF} that $f_\alpha(x) = \min\{|x| - \frac{1}{2\alpha}x^2, \frac{\alpha}{2}\}$. Because this function is constant away from the origin, $\prox_{\beta f_\alpha}(x)$ must be the identity for large values of $x$.  For example, if $\beta > \alpha$, $\prox_{\beta f_\alpha}(x)=x$ when $|x| \sqrt{\alpha \beta}$. Some other detail can be found in Example 1 of Section \ref{sec:Examples}.

In the rest of this subsection, we will give a general discussion on  the proximity operator $\prox_{\beta f_\alpha}$ for $f$ defined by  \eqref{def:falpha_quad}. We assume that $x \geq 0$ for a moment. By Lemma~\ref{lem:sparsity_technical2}, we know that $\prox_{\beta f_\alpha}(x) \subseteq  [0, \infty)$, therefore by the definition of the proximity operator,
$$
\prox_{\beta f_\alpha}(x) = \argmin_{u \in [0, \infty)} E(x,u):= f_\alpha(u) + \frac{1}{2\beta}(u-x)^2.
$$
In view of \eqref{def:falpha_quad}, the objective function $E(x,u)$ with $(x,u) \in [0, \infty)\times [0, \infty)$ is
\begin{equation}\label{eq:E}
E(x,u)=\left\{
       \begin{array}{ll}
         E_1(x,u), & \hbox{if $u \in [0,\alpha b_2]$;} \\
         E_2(x,u), & \hbox{if $u \in [\alpha b_2, \infty)$,}
       \end{array}
     \right.
\end{equation}
where
\begin{eqnarray}
E_1(x,u) &=& \frac{1}{2}\left(a_2 - \frac{1}{\alpha} + \frac{1}{\beta}\right)u^2 + \left(b_2 - \frac{1}{\beta}x\right)u + \frac{1}{2\beta}x^2, \label{eq:E_1} \\
E_2(x,u) &=& \frac{1}{2} \left(\frac{\alpha a_2^2}{\alpha a_2 + 1} + \frac{1}{\beta}\right)u^2 + \left(\frac{\alpha a_2 b_2}{\alpha a_2 + 1} - \frac{1}{\beta} x \right) u + \frac{\alpha b_2^2}{2(\alpha a_2 + 1)} + \frac{1}{2\beta} x^2. \label{eq:E_2}
\end{eqnarray}
These two functions match at the line $u=\alpha b_2$, that is,  for all $x \ge 0$,
\begin{equation}\label{eq:e1-E2}
E_1(x,\alpha b_2)=E_2(x, \alpha b_2),
\end{equation}
which will facilitate the proofs of technical lemmas given later.

Define
$$
s_1(x)=\argmin_{u \in [0, \alpha b_2]} E_1(x,u) \quad \mbox{and} \quad s_2(x)=\argmin_{u \in [\alpha b_2, \infty)} E_2(x,u).
$$
Obviously,
\begin{equation}\label{eq:prox-s1cups2}
\prox_{\beta f_\alpha}(x) \subset s_1(x) \cup  s_2(x).
\end{equation}
Therefore, to figure out the expression of $\prox_{\beta f_\alpha}(x)$, there is a need to know the structures of the sets $s_1(x)$ and $s_2(x)$.

Since the quadratic polynomial $E_2(x,\cdot)$ is strictly convex, then we have for each $x \ge 0$, $s_2(x)$ is a singleton set as follows:
\begin{eqnarray}
s_2(x)&=&\max\left\{\alpha b_2, \frac{\alpha a_2 + 1}{\alpha a_2(a_2\beta + 1)+1}\left(x-\frac{\alpha a_2 \beta b_2}{\alpha a_2 + 1}\right)\right\}   \nonumber  \\
&=& \begin{cases}
\alpha b_2, & \text { if } 0 \le x \leq \alpha b_2(a_2 \beta + 1);\\
 \frac{\alpha a_2 + 1}{\alpha a_2(a_2\beta + 1)+1}\big(x-\frac{\alpha a_2 \beta b_2}{\alpha a_2 + 1}\big), & \text{if } x \geq \alpha b_2(a_2 \beta + 1),
\end{cases}  \label{eq:s2}
\end{eqnarray}
which clearly is a piecewise linear function of $x$.

%Particularly, if $b_2=0$, we have $\prox_{\beta f_\alpha}(x) = s_2(x)$ which has one and only one element, hence, $\prox_{\beta f_\alpha}(x)$ is single-valued and is a linear function of $x$.

\begin{lemma}\label{lemma:b2=0}
Let $f$ be a piecewise quadratic sparsity promoting function as defined by \eqref{def:quad}. If $b_2=0$, then $\prox_{\beta f_\alpha}(x) = s_2(x)$ for all $x\ge 0$, where $s_2$ is given by \eqref{eq:s2}.
\end{lemma}
\begin{proof}\ \ This follows from \eqref{eq:E} and \eqref{eq:E_2} that $E(x,u)=E_2(x,u)$ for $(x,u) \in [0, \infty)\times [0, \infty)$.
\end{proof}

Next, we assume that $b_2>0$ by Lemma~\ref{lemma:NS-conditions-quadratic}. In view of the form of $E_1(x,\cdot)$ in \eqref{eq:E_1}, we consider three cases: $a_2 - \frac{1}{\alpha} + \frac{1}{\beta} > 0$, $a_2 - \frac{1}{\alpha} + \frac{1}{\beta} = 0$, and $a_2 - \frac{1}{\alpha}+\frac{1}{\beta} < 0$ which are equivalently to  (i) $\alpha b_2 (a_2 \beta + 1) > \beta b_2$, (ii) $\alpha b_2 (a_2 \beta + 1) = \beta b_2$, and  (iii) $\alpha b_2 (a_2 \beta + 1) < \beta b_2$, respectively. Accordingly, $E_1(x,\cdot)$ is strongly convex, convex, or concave on $[0, \alpha b_2]$. The result for case (i) is stated in the following lemma.

\begin{lemma}\label{lemma:b2>0:case1}
Let $f$ be a piecewise quadratic sparsity promoting function as defined by \eqref{def:quad}. If $b_2>0$ and  $\alpha b_2 (a_2 \beta + 1) > \beta b_2$, then
\begin{equation}\label{eq:general-prox-case1}
\prox_{\beta f_\alpha}(x) =\left\{
         \begin{array}{ll}
           0, & \hbox{if $0\le x < \beta b_2$;} \\
          \frac{\alpha}{(a_2\beta+1)\alpha-\beta}(x-\beta b_2), & \hbox{if $\beta b_2 \le x \le \alpha b_2 (a_2\beta+1)$;} \\
           \frac{\alpha a_2 + 1}{\alpha a_2(a_2\beta + 1)+1}\left(x-\frac{\alpha a_2 \beta b_2}{\alpha a_2 + 1}\right), & \hbox{if $x >\alpha b_2 (a_2\beta+1)$.}
         \end{array}
       \right.
\end{equation}
\end{lemma}
\begin{proof}\ \ From \eqref{eq:prox-s1cups2}, we first find the set $s_1(x)$ since the set $s_2(x)$ is already given in \eqref{eq:s2}.   By the assumption of this lemma, for each $x \ge 0$, $s_1(x)$ contains only one element and is given as follows:
$$
s_1(x)=\left\{
         \begin{array}{ll}
           0, & \hbox{if $0\le x < \beta b_2$;} \\
          \frac{\alpha}{(a_2\beta+1)\alpha-\beta}(x-\beta b_2), & \hbox{if $\beta b_2 \le x \le \alpha b_2 (a_2\beta+1)$;} \\
           \alpha b_2, & \hbox{if $x >\alpha b_2 (a_2\beta+1)$.}
         \end{array}
       \right.
$$
To determine the expression of $\prox_{\beta f_\alpha}(x)$ from the sets $s_1(x)$ and $s_2(x)$, we look at the behaviours of the functions $E_1$ and $E_2$ in the first quadrant of the $(x,u)$-plane.

We use Figure \ref{figure:b2>0:case1} to visualize the minimizers of $E_1$ and $E_2$. Three vertical lines $x=0$, $x=\beta b_2$, and $x=\alpha b_2 (a_2\beta+1)$, and two horizontal lines $u=0$ and $u=\alpha b_2$ partition the first quadrant into six rectangular regions (I to VI).  The solid red line is the graph of $s_1(x)$ while the dashed blue line is the graph of $s_2(x)$.

We know $E_1(x, 0) \leq E_1(x, u)$ in region I and $E_2(x, \alpha b_2) \leq E_2(x, u)$ in region II, so $E_1(x,0) < E_2(x, \alpha b_2)$ by Equation \eqref{eq:e1-E2} for $0\le x\le \beta b_2$. We observe  $E_1(x, s_1(x)) \leq E_1(x, u)$ in region III and $E_2(x, \alpha b_2) \leq E_2(x, u)$ in region IV, so $E_1(x,s_1(x)) < E_2(x, \alpha b_2)$ by Equation \eqref{eq:e1-E2} for $\beta b_2 \le x \le \alpha b_2 (a_2\beta+1)$; Finally, we know $E_1(x, \alpha b_2) \leq E_1(x, u)$ in region V and $E_2(x, s_2(x)) \leq E_2(x, u)$ in region VI, so $E_2(x,s_2(x)) < E_1(x, \alpha b_2)$ by Equation \eqref{eq:e1-E2} for $x >\alpha b_2 (a_2\beta+1)$. Thus $\prox_{\beta f_\alpha}$ is given by \eqref{eq:general-prox-case1}.
\begin{figure}[h]
\centering
\begin{tabular}{ccc}
\includegraphics[scale=0.45]{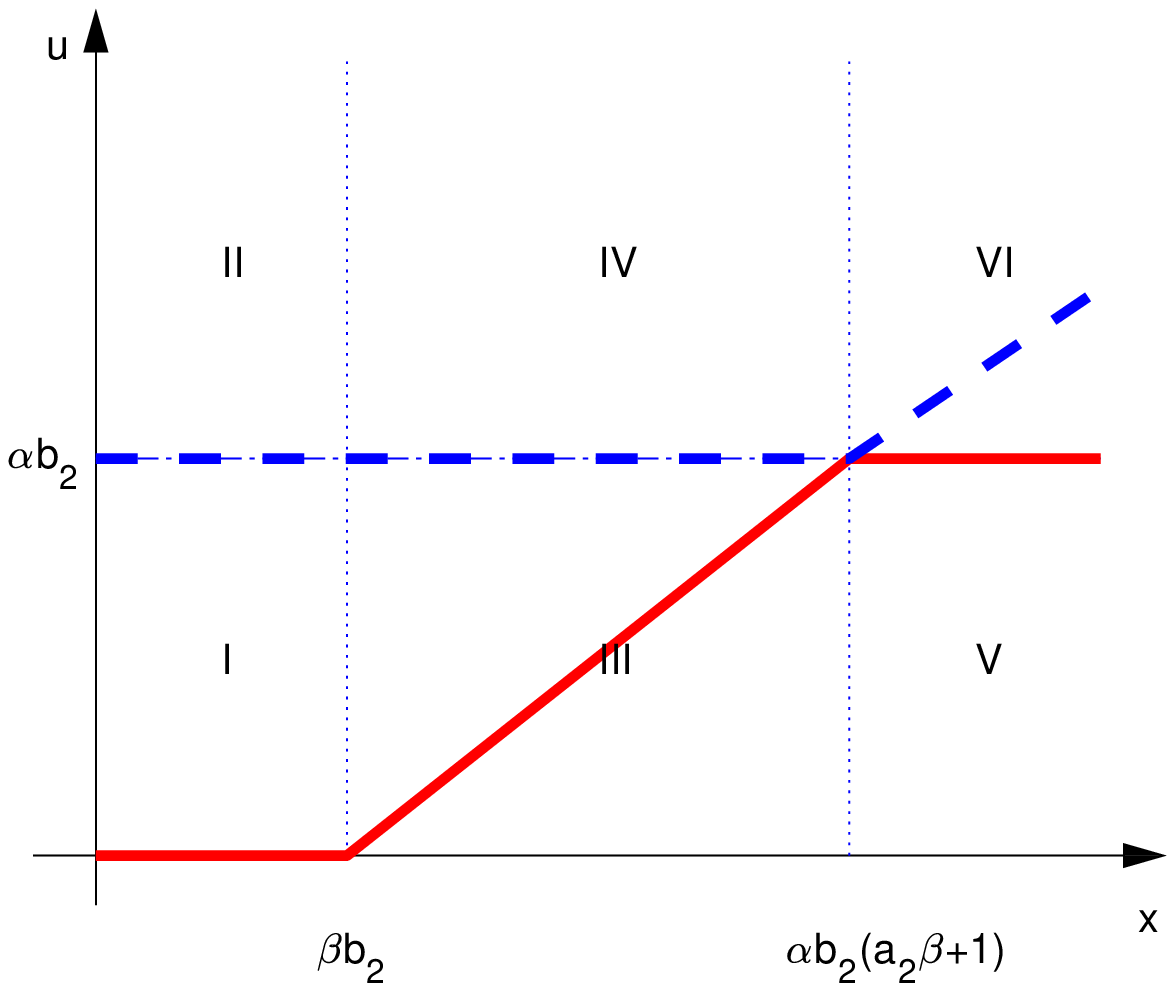} & & \includegraphics[scale=0.45]{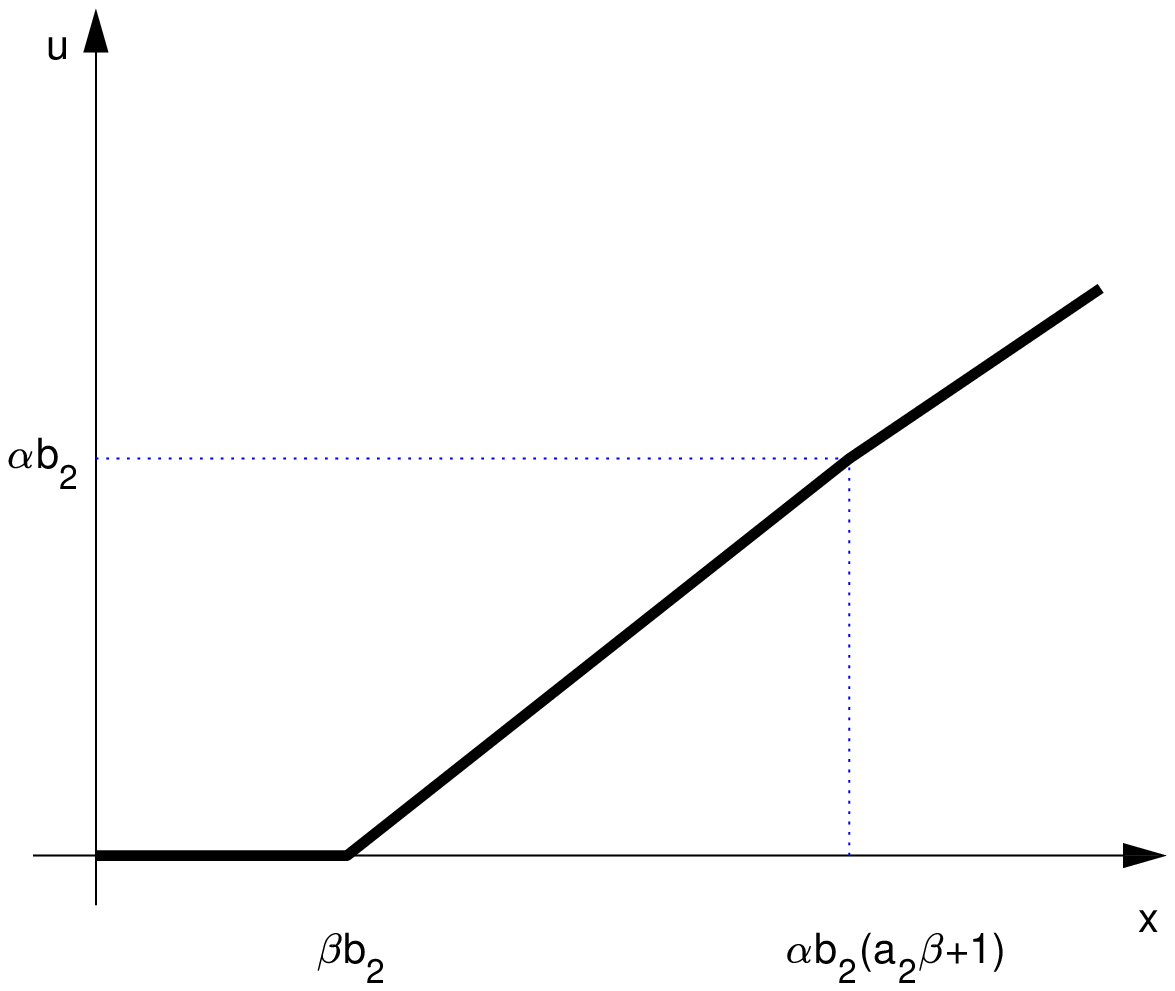}\\
(a) & & (b)
\end{tabular}
\caption{An illustration of case (i): $b_2>0$ and $\alpha b_2 (a_2 \beta + 1) > \beta b_2$. The graphs of (a) $s_1(x)$ (solid) and $s_2(x)$ (dashed) and (b) the resulting proximity operator $\prox_{\beta f_\alpha}(x)$.} \label{figure:b2>0:case1}
\end{figure}
\end{proof}

Next result is for case (ii).

\begin{lemma}\label{lemma:b2>0:case2}
Let $f$ be a piecewise quadratic sparsity promoting function as defined by \eqref{def:quad}. If $b_2>0$ and  $\alpha b_2 (a_2 \beta + 1) = \beta b_2$, then
\begin{equation}\label{eq:general-prox-case2}
\prox_{\beta f_\alpha}(x) =\left\{
         \begin{array}{ll}
           0, & \hbox{if $0\le x < \beta b_2$;} \\
          {[0, \alpha b_2]}, & \hbox{if $x=\beta b_2$;} \\
           \frac{\alpha a_2 + 1}{\alpha a_2(a_2\beta + 1)+1}\left(x-\frac{\alpha a_2 \beta b_2}{\alpha a_2 + 1}\right), & \hbox{if $x >\beta b_2$.}
         \end{array}
       \right.
\end{equation}
\end{lemma}
\begin{proof}\ \ Similar to the proof of Lemma~\ref{lemma:b2>0:case1}, we first give the explicit form of the set $s_1(x)$:
$$
s_1(x)=\left\{
         \begin{array}{ll}
           0, & \hbox{if $0\le x < \beta b_2$;} \\
          {[0, \alpha b_2]}, & \hbox{if $x=\beta b_2$;} \\
           \alpha b_2, & \hbox{if $x >\beta b_2$.}
         \end{array}
       \right.
$$
We note that $\prox_{\beta f_\alpha}$ can be set-valued only at $\beta b_2$.

In Figure~\ref{figure:b2>0:case2}, two vertical lines $x=0$ and $x=\beta b_2$,  and two horizontal lines $u=0$ and $u=\alpha b_2$ partition the first quadrant into four rectangular regions (I to IV).  The solid red line is the graph of $s_1(x)$ while the dashed blue line is the graph of $s_2(x)$. It is identical to Figure \ref{figure:b2>0:case1} with the middle regions collapsed to a line. Following the same reasoning as in Lemma \ref{lemma:b2>0:case1}, we see that \eqref{eq:general-prox-case2} holds.
\begin{figure}[h]
\centering
\begin{tabular}{ccc}
\includegraphics[scale=0.45]{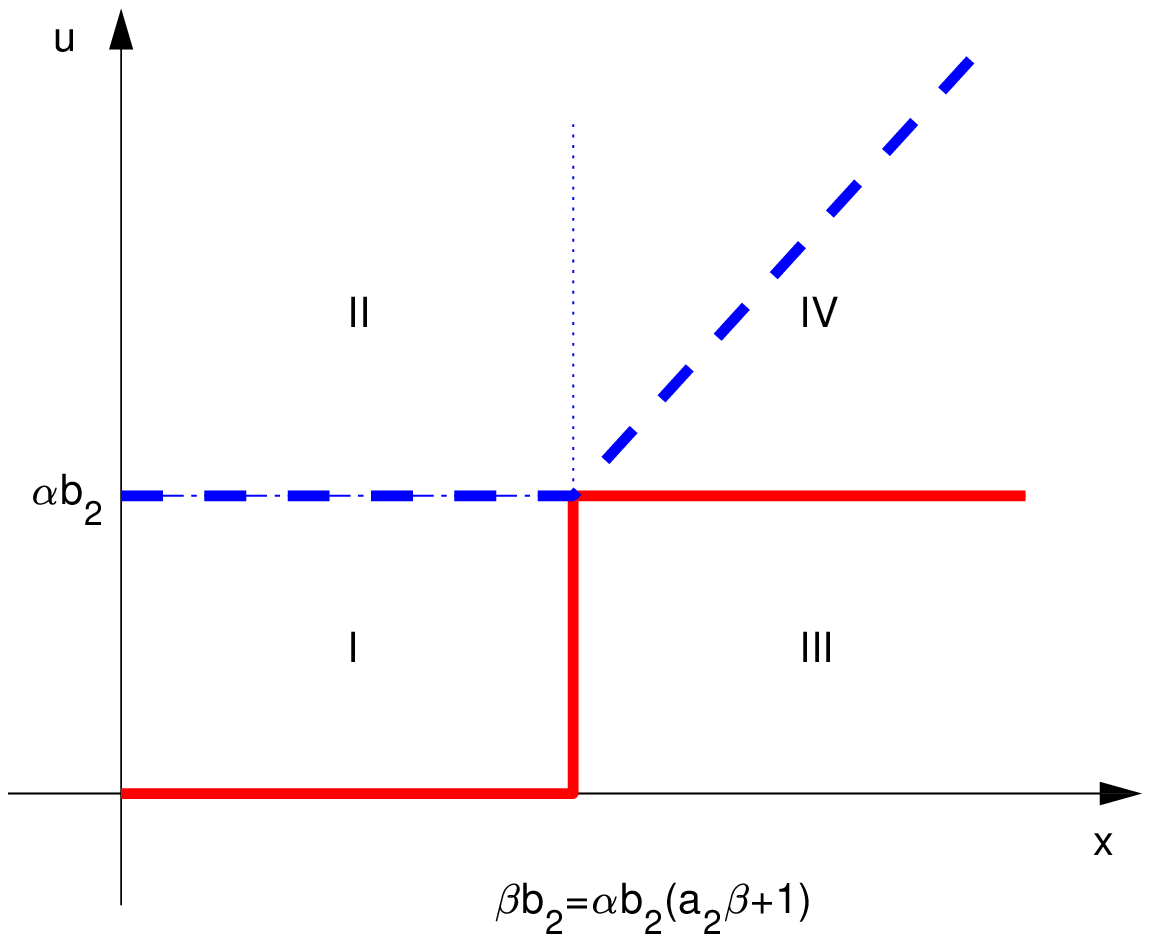} & & \includegraphics[scale=0.45]{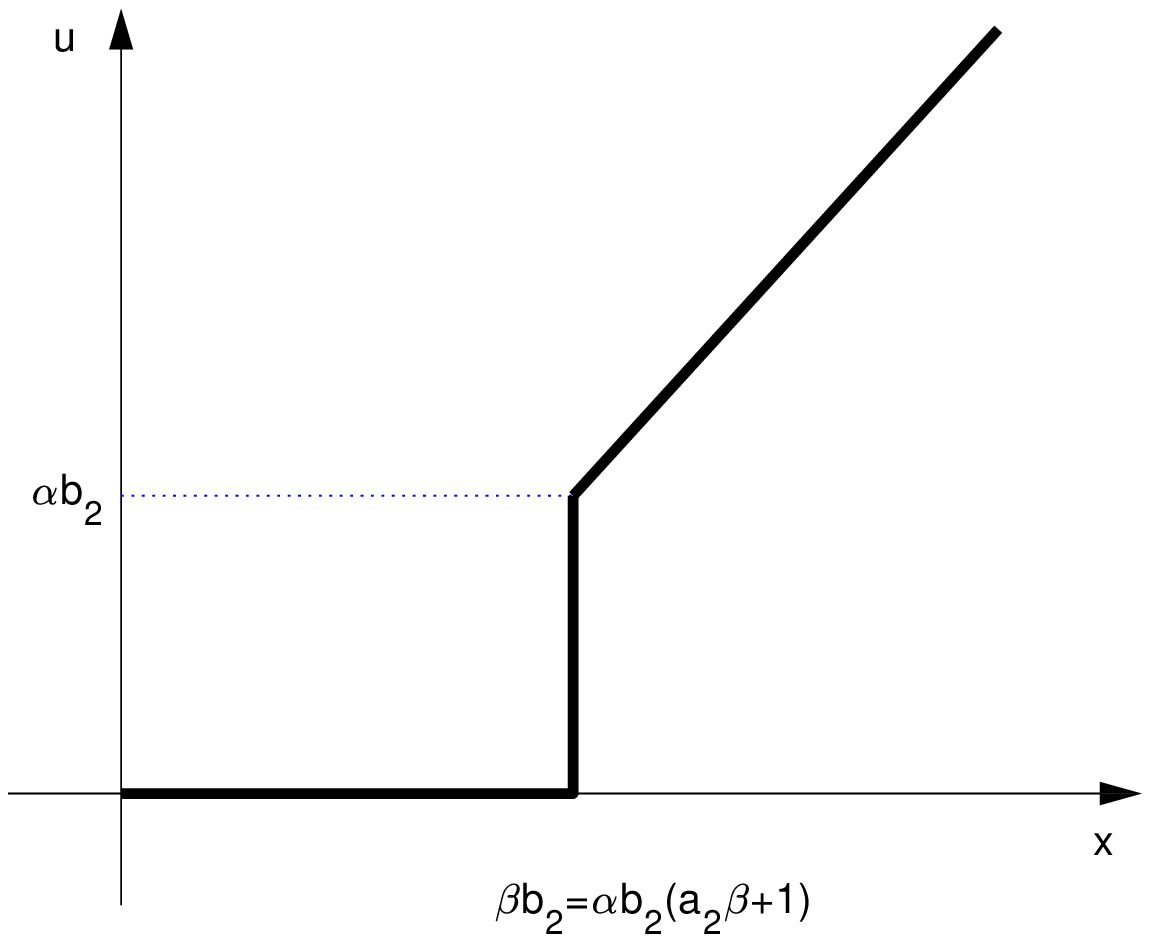} \\
(a) & & (b)
\end{tabular}
\caption{An illustration of case  (ii): $b_2 > 0$ and $\alpha b_2 (a_2 \beta + 1) = \beta b_2$. The graphs of (a) $s_1(x)$ (solid) and $s_2(x)$ (dashed) and (b) the resulting proximity operator $\prox_{\beta f_\alpha}(x)$.}\label{figure:b2>0:case2}
\end{figure}
\end{proof}

Finally, we consider case (iii). Because $\beta b_2$ and $\alpha b_2 (a_2 \beta + 1)$ have now switched positions, we see that we must take care when dealing with the intermediate $x$ values.

\begin{lemma}\label{lemma:b2>0:case3}
Let $f$ be a piecewise quadratic sparsity promoting function as defined by \eqref{def:quad}. Define
$$
\tau^+ = \frac{\alpha a_2 \beta b_2}{\alpha a_2 + 1} + \frac{\sqrt{\alpha \beta (\alpha a_2^2 \beta + \alpha a_2 + 1)}b_2}{\alpha a_2 + 1}.
$$
If $b_2>0$ and  $\alpha b_2 (a_2 \beta + 1) < \beta b_2$,
\begin{equation}\label{eq:general-prox-case3}
\prox_{\beta f_\alpha}(x) =\left\{
         \begin{array}{ll}
           0, & \hbox{if $0\le x < \tau^+$;} \\
          \left\{0, \frac{\alpha a_2 + 1}{\alpha a_2(a_2\beta + 1)+1}\left(\tau^+-\frac{\alpha a_2 \beta b_2}{\alpha a_2 + 1}\right)\right\}, & \hbox{if $x=\tau^+$;} \\
           \frac{\alpha a_2 + 1}{\alpha a_2(a_2\beta + 1)+1}\left(x-\frac{\alpha a_2 \beta b_2}{\alpha a_2 + 1}\right), & \hbox{if $x >\tau^+$.}
         \end{array}
       \right.
\end{equation}
\end{lemma}
\begin{proof}\ \ Again, we first give the explicit form of the set $s_1(x)$. Note that $E_1(x, \cdot)$ is concave in this case, so the minimum occurs at the endpoints according to the position of the vertex. Thus,
$$
s_1(x)=\left\{
         \begin{array}{ll}
           0, & \hbox{if $0\le x < \frac{1}{2}(\alpha b_2(a_2\beta+1)+\beta b_2)$;} \\
          \{0, \alpha b_2\}, & \hbox{if $x=\frac{1}{2}(\alpha b_2(a_2\beta+1)+\beta b_2)$;} \\
           \alpha b_2, & \hbox{if $x>\frac{1}{2}(\alpha b_2(a_2\beta+1)+\beta b_2)$.}
         \end{array}
       \right.
$$
This is set-valued at $\frac{1}{2}(\alpha b_2(a_2\beta+1)+\beta b_2)$.

As before, we plot $s_1(x)$ and $s_2(x)$ in Figure \ref{figure:b2>0:case3}. Three vertical lines $x=0$, $x=\alpha b_2(a_2\beta+1)$, and $x=\frac{1}{2}(\alpha b_2(a_2\beta+1)+\beta b_2)$,  and two horizontal lines $u=0$ and $u=\alpha b_2$ partition the first quadrant into six rectangular regions as shown in Figure~\ref{figure:b2>0:case3}(a).The solid red line is the graph of $s_1(x)$ while the dashed blue line is the graph of $s_2(x)$. From this figure and \eqref{eq:e1-E2}, it is easy to see that regions I, II, V, and VI behave as in the previous cases. That is, $\prox_{\beta f_\alpha}(x) = s_1(x)$ for $0 \leq x \leq \alpha b_2(a_2 \beta + 1)$ and $\prox_{\beta f_\alpha}(x) = s_2(x)$ for $x \ge \frac{1}{2}(\alpha b_2(a_2 \beta + 1) + \beta b_2)$.

To find the expression of $\prox_{\beta f_\alpha}(x)$ for $\alpha b_2(a_2\beta+1)<x< \frac{1}{2}(\alpha b_2(a_2\beta+1)+\beta b_2)$, from the solid red line and the dashed blue in regions III and IV, we need to compare the value of $E_1(x,0)$ with $E_2(x,s_2(x))$. Using \eqref{eq:s2}, a direct computation gives
$$
E_2(x,s_2(x)) - E_1(x,0) = - \frac{\alpha a_2 + 1}{2\beta(\alpha a_2(a_2 \beta + 1)+1)}\left(x - \frac{\alpha a_2 \beta b_2}{\alpha a_2 + 1}\right)^2 + \frac{\alpha b_2^2}{2(\alpha a_2 + 1)}.
$$
Notice that $E_2(x,s_2(x)) - E_1(x,0)>0$ at $x=\alpha b_2(a_2\beta+1)$ and $E_2(x,s_2(x)) - E_1(x,0)<0$ at $x=\frac{1}{2}(\alpha b_2(a_2\beta+1)+\beta b_2)$. Hence, the quadratic polynomial $E_2(x,s_2(x)) - E_1(x,0)$ has only one root at $\tau^+$ that is between $\alpha b_2 (a_2 \beta + 1)$ and $ \frac{1}{2}(\alpha b_2(a_2\beta+1)+\beta b_2)$. So, the result of this lemma holds and is illustrated in Figure~\ref{figure:b2>0:case3}(c).
\begin{figure}[h]
\centering
\begin{tabular}{ccc}
\includegraphics[scale=0.45]{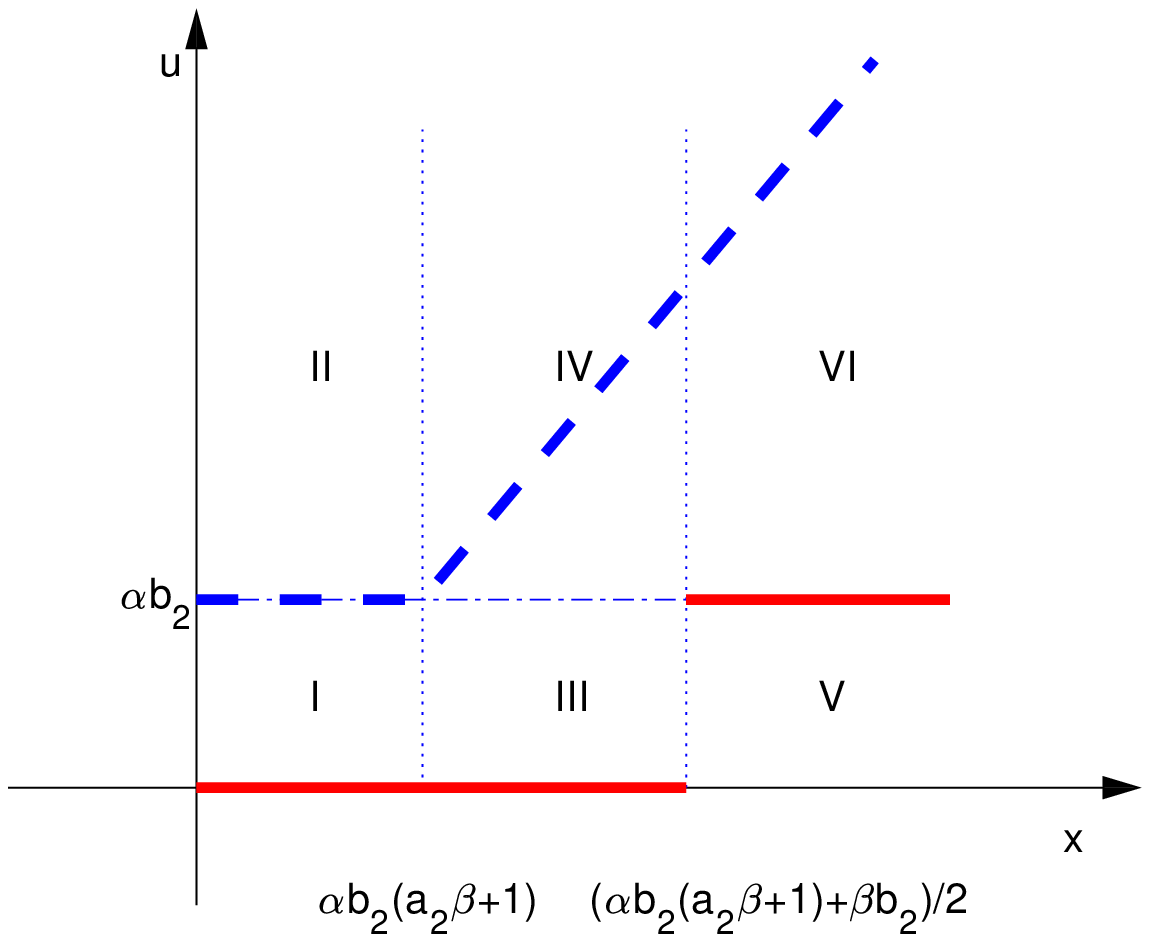} &  \includegraphics[scale=0.45]{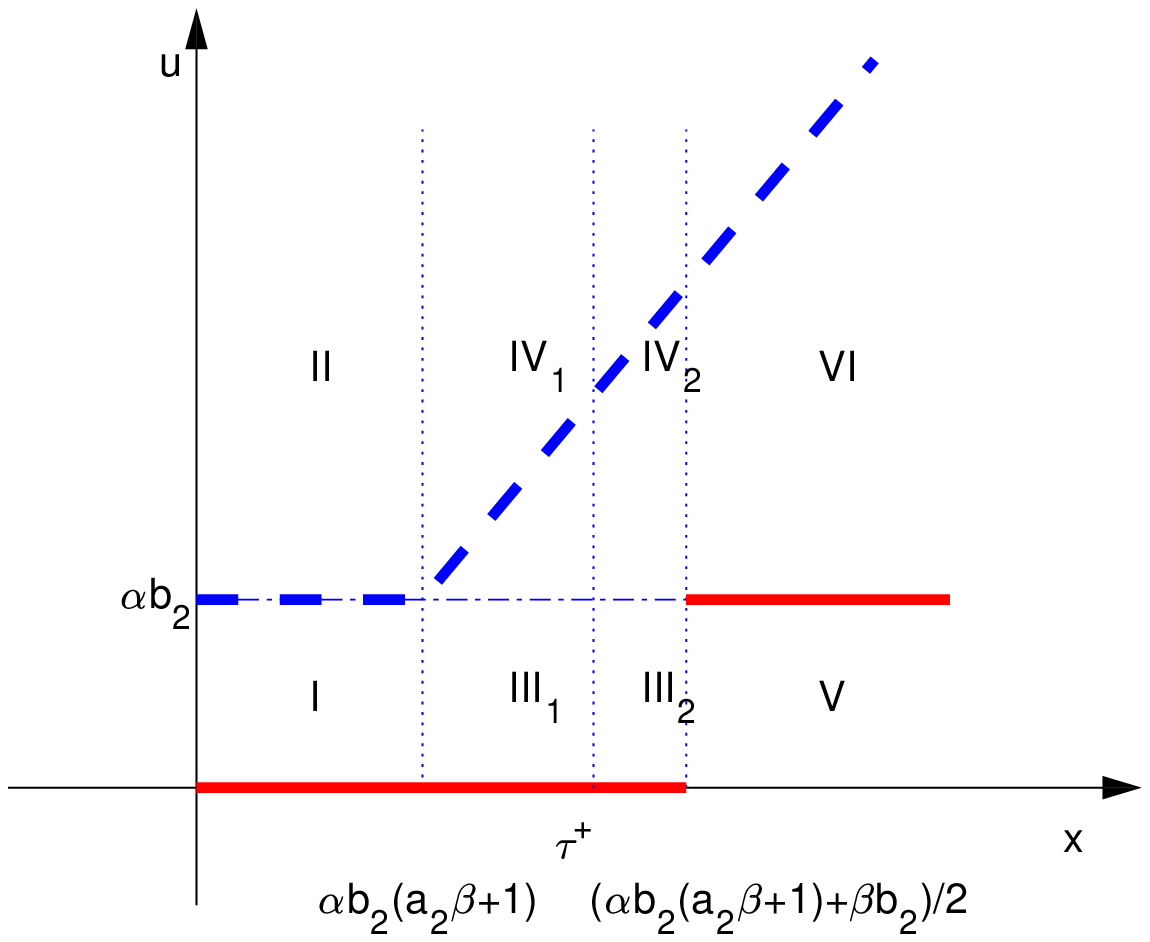} & \includegraphics[scale=0.45]{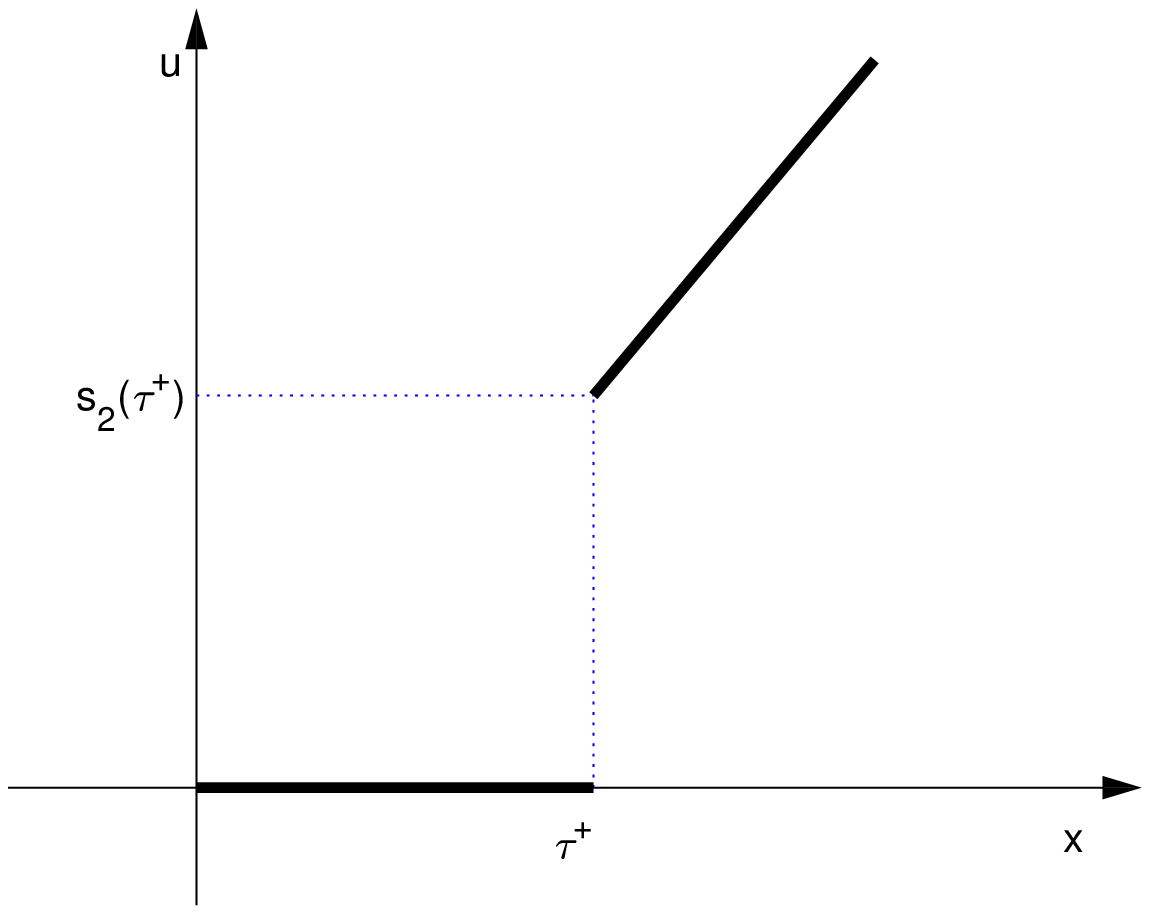}\\
(a) &  (b) & (c)
\end{tabular}
\caption{An illustration of case (iii): $b_2 > 0$ and $\alpha b_2 (a_2 \beta + 1) < \beta b_2$. The graphs of (a), (b) $s_1(x)$ (solid) and $s_2(x)$ (dashed) and (c) the resulting proximity operator $\prox_{\beta f_\alpha}(x)$.}\label{figure:b2>0:case3}
\end{figure}
\end{proof}

With the above results, we know $\prox_{\beta f_\alpha}(x)$ for $x \geq 0$. The following lemma extends these results to $x \leq 0$.
\begin{lemma}\label{lemma:x<=0}
Let $f$ be a piecewise quadratic sparsity promoting function as defined by \eqref{def:quad}. Define $g: x \mapsto f(-x)$. Then for $x\le 0$ and any positive numbers $\alpha$ and $\beta$, we have $\prox_{\beta f_\alpha}(x)= -\prox_{\beta g_\alpha}(-x)$ where $\prox_{\beta g_\alpha}(-x)$ can be evaluated using the results in Lemmas~\ref{lemma:b2=0}-\ref{lemma:b2>0:case3}.
\end{lemma}
\begin{proof} Since $f$ is sparsity promoting, so is $g$ by Theorem~\ref{thm:sparsity}. Moreover, $f_\alpha = g_\alpha(-\cdot)$ which leads to $\prox_{\beta f_\alpha}(x)= -\prox_{\beta g_\alpha}(-x)$ for all $x$. Note that
\begin{equation*}
g(x) = \begin{cases} \frac{1}{2}a_2x^2 - b_2x, & \text{if } x \leq 0;\\
\frac{1}{2}a_1x^2 - b_1x, & \text{if } x \ge 0,
\end{cases}
\end{equation*}
which is a piecewise quadratic sparsity promoting function. All results developed in Lemmas~\ref{lemma:b2=0}-\ref{lemma:b2>0:case3} can be applied for $g$. Therefore, the results of this lemma follow immediately.
\end{proof}

In summary, we have the following result.
\begin{theorem}\label{thm:sparsequadratic}
If $f \in \Gamma_0(\mathbb{R})$ is a quadratic sparsity promoting function as defined by \eqref{def:quad}, then the following statements hold.
\begin{itemize}
\item[(i)] $\prox_{\beta f_\alpha}$ is set-valued at at most one point at each side of the origin. Moreover, $\prox_{\beta f_\alpha}$ is piecewise linear on any interval not containing these possible set-valued points.
\item[(ii)] For any $p \in \prox_{\beta f_\alpha}(x)$, $|p| \leq |x|$. Furthermore, $\sign(p) = \sign(x)$ if both $p$ and $x$ are nonzero.
\end{itemize}
\end{theorem}
\begin{proof} \ \ All results follows directly from the expressions of $\prox_{\beta f_\alpha}(x)$ given in Lemma~\ref{lemma:b2=0}-Lemma~\ref{lemma:x<=0}.
\end{proof}

\begin{remark}
Theorem~\ref{thm:sparsequadratic} guarantees that $\prox_{\beta f_\alpha}$ will be a thresholding operator for any $f_\alpha$ given by \eqref{def:falpha_quad}. Furthermore, Lemmas~\ref{lemma:b2=0}-\ref{lemma:b2>0:case3} provide detailed and easily customizable forms which can be tailored to applications.
\end{remark}

%%%%%%%%%%%%%%%%%%%%%%%%%%%%%%%%%%%%%%%%%%%%%%%%%%%%%%%%%%%%%
\subsection{Piecewise Quadratic on Intervals}
%%%%%%%%%%%%%%%%%%%%%%%%%%%%%%%%%%%%%%%%%%%%%%%%%%%%%%%%%%%%%
Let $C$ be a closed interval containing the origin and $f$ a piecewise quadratic function defined by \eqref{def:quad}. We consider a function $\widetilde{f}$ that is the restriction of $f$ on the interval $C$ as follows:
\begin{equation}\label{def:ftidle}\tag{$\widetilde{\mathcal{Q}}$}
\widetilde{f} = f+\iota_C.
\end{equation}
\begin{lemma}\label{lemma:f+iota}
Let $f$ be a piecewise quadratic sparsity promoting function defined by \eqref{def:quad} and let $C$ be a closed interval on $\mathbb{R}$ such that $\{0\} \subsetneq \partial f(0) \cap C$. Then  $\widetilde{f}$ defined in \eqref{def:ftidle} is a sparsity promoting function. Moreover,
\begin{equation}\tag{$\widetilde{\mathcal{Q}}_\alpha$}
\widetilde{f}_\alpha = f_\alpha + \iota_C.
\end{equation}
\end{lemma}
\begin{proof}\ \  Since $f$ is sparsity promoting, $\min_{x\in \mathbb{R}} f(x) = f(0) = 0$. Because  $\{0\} \subsetneq \partial f(0) \cap C$, we know that $\widetilde{f}(0)=\min_{x\in C} f(x)=\min_{x\in \mathbb{R}} f(x)=0$. That is, $\widetilde{f}$ achieves its minimum at the origin. We further know that $\partial f(0) = \partial \widetilde{f}(0)$, hence $\{0\} \subsetneq \partial f(0) \cap C = \partial \widetilde{f}(0) \cap C$. Therefore, $\widetilde{f}$ is sparsity promoting.

By Lemma~\ref{lemma:convex-prox-sparsity} and Lemma~\ref{lem:sparsity_technical2}, $\prox_{\alpha f} (x) \in C$ if $x \in C$. This indicates that for $x \in C$
$$
\env_{\alpha f}(x)=\min_{u \in \mathbb{R}} \left\{f(u)+\frac{1}{2\alpha}(u-x)^2\right\} = \min_{u \in C} \left\{f(u)+\frac{1}{2\alpha}(u-x)^2\right\}=\env_{\alpha \widetilde{f}}(x).
$$
The above identities yield $\widetilde{f}_\alpha = f_\alpha + \iota_C$. This completes the proof of the result.
\end{proof}

By the above lemma, for $\widetilde{f}$ defined in \eqref{def:ftidle} we always assume that the coefficients in $f$ satisfy \eqref{eq:requirements} and that $C=[\lambda_1, \lambda_2]$ with $\lambda_1\le 0 \le \lambda_2$ and $\lambda_2-\lambda_1>0$.

\begin{theorem}\label{thm:f+iota}
Let $\widetilde{f}$ be defined in \eqref{def:ftidle}, let $x \in \mathbb{R}$, an let $\alpha$ and $\beta$ be two positive numbers. Then the following statements hold.
\begin{itemize}
\item[\normalfont(i)] If the set $\prox_{\beta {f}_\alpha}(x) \cap C$ is not empty, then $\prox_{\beta {f}_\alpha}(x)\cap C \subseteq \prox_{\beta \widetilde{f}_\alpha}(x)$;

\item[\normalfont(ii)] If $\lambda_2 \in  \prox_{\beta \widetilde{f}_\alpha}(x)$, then $\lambda_2 \in  \prox_{\beta \widetilde{f}_\alpha}(y)$ for all $y>x$;

\item[\normalfont(iii)] If $\lambda_1 \in  \prox_{\beta \widetilde{f}_\alpha}(x)$, then $\lambda_1 \in  \prox_{\beta \widetilde{f}_\alpha}(y)$ for all $y<x$;

\end{itemize}
\end{theorem}
\begin{proof}\ \
(i): Assume $p$ is an element in $\prox_{\beta {f}_\alpha}(x) \cap C$. We have
\begin{eqnarray*}
f_\alpha(p)+\frac{1}{2\beta}(p-x)^2 &=& \min_{u \in \mathbb{R}}\left\{f_\alpha(u)+\frac{1}{2\beta}(u-x)^2\right\} \\
&=& \min_{u \in C}\left\{f_\alpha(u)+\frac{1}{2\beta}(u-x)^2\right\} \\
&=& \min_{u \in \mathbb{R}}\left\{\widetilde{f}_\alpha(u)+\frac{1}{2\beta}(u-x)^2\right\},
\end{eqnarray*}
where the first equation is due to $p \in \prox_{\beta {f}_\alpha}(x)$, the second equation is due to $p\in C$, the last one is due to Lemma~\ref{lemma:f+iota}, hence, $p\in \prox_{\beta \widetilde{f}_\alpha}(x)$.

(ii): Since $\lambda_2 \ge 0$, the inclusion $\lambda_2 \in  \prox_{\beta \widetilde{f}_\alpha}(x)$ together with Lemma~\ref{lem:sparsity_technical2} implies that $x\ge 0$  and for all $u \in [\lambda_1, \lambda_2]$,
$$
\widetilde{f}_\alpha(u)+\frac{1}{2\beta}(u-x)^2 \ge \widetilde{f}_\alpha(\lambda_2)+\frac{1}{2\beta}(\lambda_2-x)^2.
$$
With the above inequality, when $y>x$, we have that
\begin{eqnarray*}
\widetilde{f}_\alpha(\lambda_2)+\frac{1}{2\beta}(\lambda_2-y)^2&=&\widetilde{f}_\alpha(\lambda_2)+\frac{1}{2\beta}(\lambda_2-x)^2+\frac{1}{2\beta}(y-x)(y+x-2\lambda_2) \\
&\le&\widetilde{f}_\alpha(u)+\frac{1}{2\beta}(u-x)^2+\frac{1}{2\beta}(y-x)(y+x-2u) \\
&=&\widetilde{f}_\alpha(u)+\frac{1}{2\beta}(u-y)^2
\end{eqnarray*}
hold for all $u \in [\lambda_1, \lambda_2]$.  This yields $\lambda_2 \in  \prox_{\beta \widetilde{f}_\alpha}(y)$.

(iii): The proof is similar to (ii).
\end{proof}

%
%We end this section by extending the previous result to separable functions in higher dimensions.
%
%\begin{corollary}\label{cor:separable}
%If $f \in \Gamma_0(\mathbb{R}^n)$ can be decomposed as $f(x) = \sum_{i = 1}^n f_i(x_i)$ such that each $f_i$ is sparsity promoting and at most quadratic, then $f_\alpha$ is prox-thresholding in each coordinate.
%\end{corollary}

Theorem~\ref{thm:f+iota} tells us that the for $\widetilde{f}$ as in \eqref{def:ftidle}, $\prox_{\beta \tilde{f}_\alpha}$ will resemble the proximity operator of $f_\alpha$ around the origin and the proximity operator of $\iota_C$ elsewhere. Due to the number of parameters, there are a huge number of possible combinations. Rather than list all of the combinations here, we provide the details for a specific function in Example 4 of Section \ref{sec:Examples}.

We have shown that sparsity promoting quadratic and indicator functions have thresholding proximity operators. The results essentially rely on the fact that $\env_{\alpha f}$ is quadratic for these functions. In fact, quadratic and indicator functions are the only ones with this property \cite{Planiden-Wang:JOTA:2018}, so our discussion is a comprehensive method for obtaining thresholding rules.

%%%%%%%%%%%%%%%examples%%%%%%%%%%%%%%%%%%%

\section{Examples}\label{sec:Examples}
In this section, we illustrate our theory by presenting several examples that are of practical interest.

For the first example, we collect and expand upon the previous discussion of $f(x)=\|x\|_1=\sum_{i=1}^n|x_i|$ for $x \in \mathbb{R}^n$.  The $\ell_1$-norm has been extensively used in myriad applications for promoting sparsity.

The second example is the ReLU (Rectified Linear Unit) function. It is the most commonly used activation function in convolutional neural networks or deep learning. The ReLU function on $\mathbb{R}^n$ is defined as follows:
$f(x)=\sum_{i=1}^n \max\{0, x_i\},$ where $x \in \mathbb{R}^n$.

The third example is the elastic net penalty function which is widely used in statistics (see \cite{Zou:Hastie:JRSS:2006}). The general form of the elastic net is the linear combination of the $\ell_1$ and $\ell_2$ norms as follows: $f(x)=\frac{\lambda_1}{2}\|x\|^2+\lambda_2\|x\|_1,$  where $\lambda_1$ and $\lambda_2$ are two nonnegative parameters. In our discussion, we will simply choose $\lambda_1=\lambda_2=1$. This is known as the naive elastic net.

The last example is similar to the first one, but restricted to a cube centered at the origin. The function $f$ is given as follows:
$f(x)=\|x\|_1 + \iota_C(x),$  where $C=[-\lambda, \lambda]^n$. Generally speaking, this function promotes the sparsity on $C$.

We notice that the function $f$ in the above four examples can be written as
$$
f(x)=\sum_{i=1}^n g(x_i)
$$
for $x\in \mathbb{R}^n$ and some specific function $g$. For example, $g$ is $|\cdot|$, $\max\{0, \cdot\}$, $\frac{1}{2}|\cdot|^2 + |\cdot|$, or $|\cdot|+\iota_{[-\lambda, \lambda]}$, in examples 1, 2, 3, or 4, an analogue of $f$ when $\mathbb{R}^n$ reduces to $\mathbb{R}$. We further have that $\prox_{\alpha f}(x)=\prox_{\alpha g}(x_1) \times \prox_{\alpha g}(x_2) \times \cdots \times \prox_{\alpha g}(x_n)$, $\env_{\alpha f}(x)=\sum_{i=1}^n \env_{\alpha g}(x_i)$, $\prox_{\beta f_\alpha}(x)=\prox_{\beta g_\alpha}(x_1) \times \prox_{\beta g_\alpha}(x_2) \times \cdots \times \prox_{\beta g_\alpha}(x_n)$, and $\env_{\beta f_\alpha}(x)=\sum_{i=1}^n \env_{\beta g_\alpha}(x_i)$. Therefore, in the following discussion we will restrict ourself on $n=1$.

\subsection{Example 1: The absolute value function}

The first example is the absolute value function $f: \mathbb{R} \rightarrow \mathbb{R}: x \longmapsto |x|$, which is a special case of the piecewise quadratic function in \eqref{def:quad} with $a_1=a_2=0$, $b_1=-1$, and $b_2=1$. This function is nondifferentiable at the origin with $\mathrm{argmin}_{x \in \mathbb{R}} f(x) =\{0\}$ and $\partial f (0) =\partial |\cdot|(0)=[-1,1]$.

\begin{figure}[h]\centering
\begin{tabular}{ccc}
\includegraphics[scale=0.35]{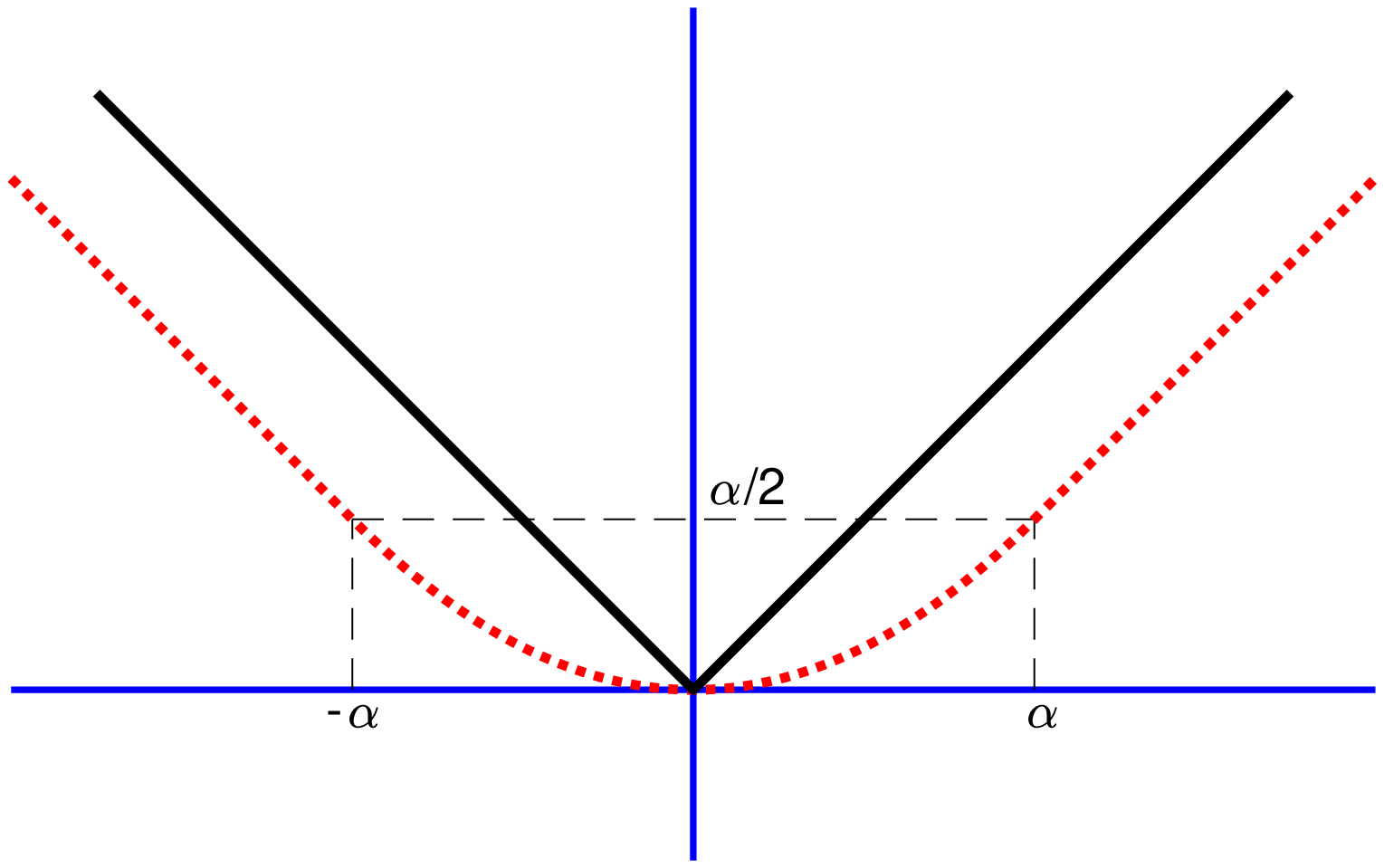} & &\includegraphics[scale=0.35]{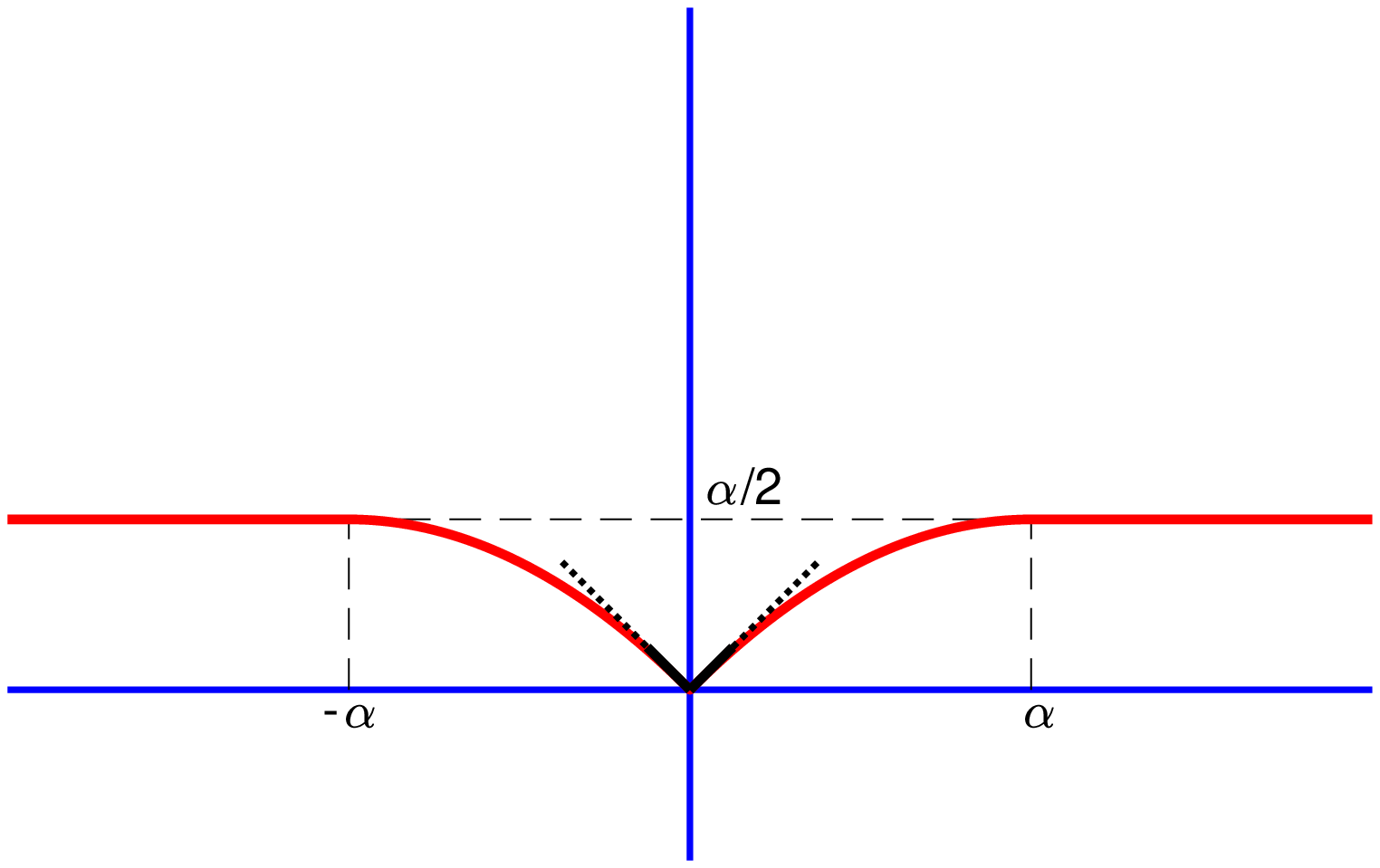}\\
(a) & & (b)
\end{tabular}
\caption{Example 1. (a) The graphs of $f$ (solid), $\env_{\alpha f}$ (dotted),  and (b) the graph of $f_\alpha = f(x) - \env_{\alpha f(x)}$. Near the origin $f_\alpha$ retains the structure of $f$, which is emphasized in black (solid-dotted).} \label{figure:ex1-f-envf}

\end{figure}

The proximity operator and the Moreau envelope of $f$ with parameter $\alpha>0$ are
$$
\mathrm{prox}_{\alpha |\cdot|} (x)=\mathrm{sgn}(x) \max\{0, |x|-\alpha\} \quad \mbox{and} \quad
\mathrm{env}_{\alpha |\cdot|} (x) = \left\{
                                \begin{array}{ll}
                                  \frac{1}{2\alpha}x^2, & \hbox{if $|x|\le \alpha$;} \\
                                  |x|-\frac{1}{2}\alpha, & \hbox{otherwise,}
                                \end{array}
                              \right.
$$
respectively. It is well know that $\mathrm{prox}_{\alpha |\cdot|}$ is called the soft thresholding in literature of wavelet \cite{Donoho:ieeeIT:06} and $\mathrm{env}_{\alpha |\cdot|}$ is Huber's function in robust statistics \cite{Huber:09}. Figure~\ref{figure:ex1-proxf} shows the typical shape of the proximity operator of $f$.

\begin{figure}[h]
\centering
\includegraphics[scale=0.35]{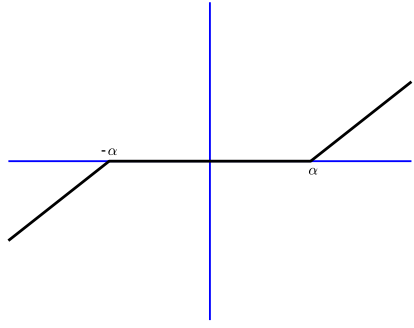}
\caption{Example 1. The typical shape of $\prox_{\alpha f}$.}\label{figure:ex1-proxf}
\end{figure}

As defined in \eqref{def:falpha}, for the absolute value function $f$,
$$
f_\alpha(x):= |x|-\mathrm{evn}_{\alpha |\cdot|} (x)=\left\{
                                \begin{array}{ll}
                                  |x|-\frac{1}{2\alpha}x^2, & \hbox{if $|x|\le \alpha$;} \\
                                  \frac{1}{2}\alpha, & \hbox{otherwise.}
                                \end{array}
                              \right.
$$
This function $f_\alpha$ (see Figure~\ref{figure:ex1-f-envf}(b)) is identical to the minimax convex penalty (MCP) function given in \cite{Zhang:AS:2010}, but motivated from statistic perspective.

%%%%%%%%%%%%%%%%%%%
%\includegraphics[scale=0.35]{figure_absprox.eps}

%The subdifferential of $f_\alpha$ at $x$ is
%$$
%\partial f_\alpha(x) = \left\{
%                         \begin{array}{ll}
%                           {[-1,1]}, & \hbox{if $x=0$.}\\
%                           \{\mathrm{sign}(x)-\frac{1}{\alpha} x\}, & \hbox{if $0<|x|<\alpha$;} \\
%                           \{0\}, & \hbox{if $|x| \ge \alpha$;}
%                         \end{array}
%                       \right.
%$$
%Hence, $f_\alpha$ is a sparsity promoting function.

The expression of $\mathrm{prox}_{\beta f_{\alpha}}$ depends on the relative values of $\alpha$ and $\beta$. If $\beta<\alpha$, Lemma~\ref{lemma:b2>0:case1} gives
\begin{equation}\label{eq:beta<alpha}
\mathrm{prox}_{\beta f_{\alpha}} (x) =    \left\{
                           \begin{array}{ll}
                           0, & \hbox{if $|x|\le \beta$;}\\
                           \frac{\alpha}{\alpha-\beta}(|x|-\beta)\mathrm{sgn}(x), & \hbox{if $\beta<|x|\le\alpha$;} \\
                           x, & \hbox{if $|x| \ge \alpha$.}
                         \end{array}
                        \right.
\end{equation}
This is the firm thresholding operator \cite{Gao-Bruce:SS:97}. If $\beta=\alpha$, Lemma~\ref{lemma:b2>0:case2} gives
\begin{equation}\label{eq:beta=alpha}
\mathrm{prox}_{\beta f_{\alpha}} (x) =\left\{
                           \begin{array}{ll}
                           0, & \hbox{if $|x|< \alpha$;}\\
                           {[0, \alpha]}, & \hbox{if $|x| = \alpha$;}\\
                            x, & \hbox{if $|x| > \alpha$,}
                         \end{array}
\right.
\end{equation}

Finally, if $\beta>\alpha$, Lemma~\ref{lemma:b2>0:case3} gives
\begin{equation}\label{eq:beta>alpha}
\mathrm{prox}_{\beta f_{\alpha}} (x) =\left\{
                           \begin{array}{ll}
                           0, & \hbox{if $|x|< \sqrt{\alpha \beta}$;}\\
                           \{0,x\}, & \hbox{if $|x| = \sqrt{\alpha \beta}$;}\\
                            x, & \hbox{if $|x| > \sqrt{\alpha \beta}$;}
                         \end{array}
\right.
\end{equation}
The proximity operator $\mathrm{prox}_{\beta f_{\alpha}}$ for different values of $\alpha$ and $\beta$ is plotted in Figure~\ref{fig:prox-|x|}.

\begin{figure}[h]
\centering
\begin{tabular}{ccc}
\includegraphics[scale=0.4]{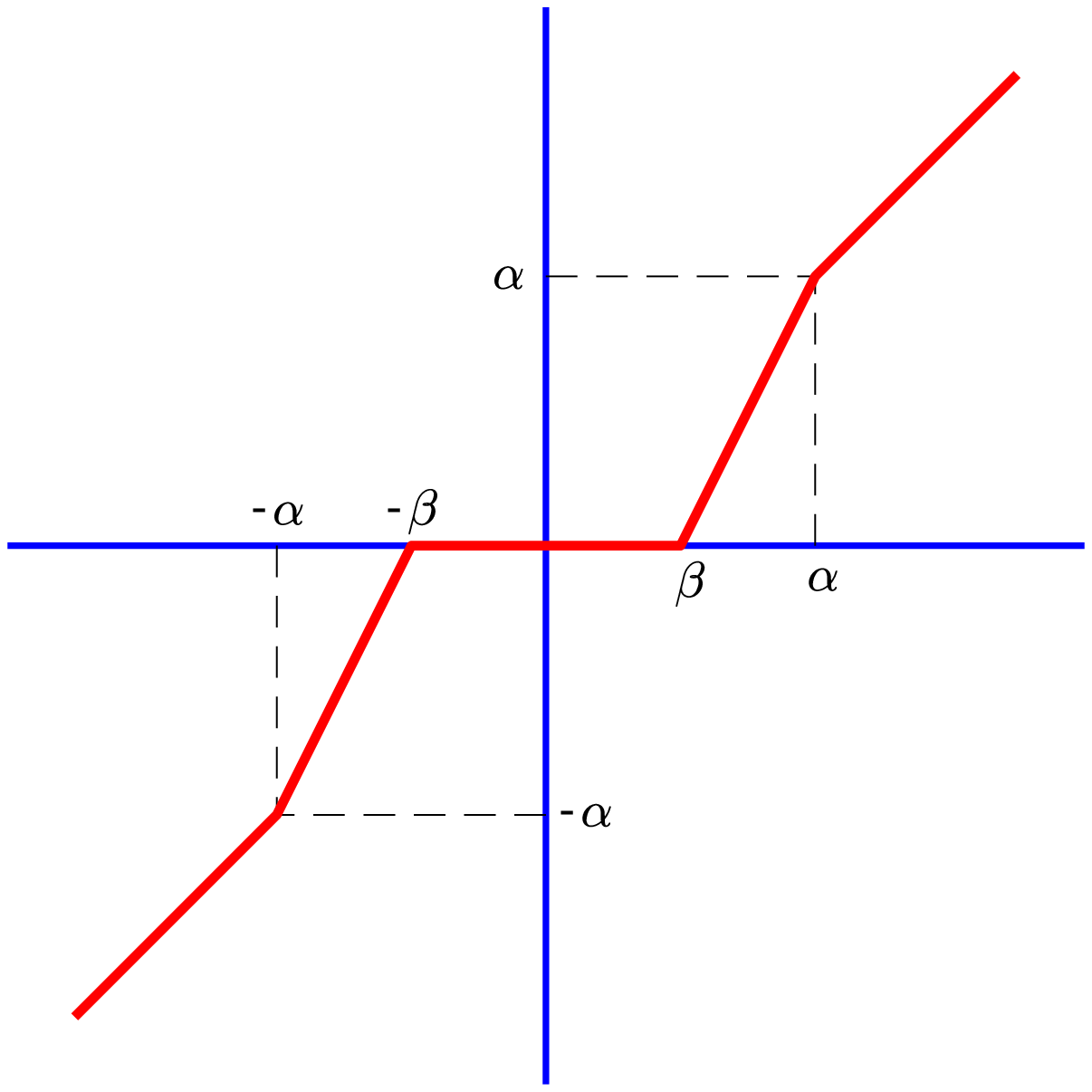} & \includegraphics[scale=0.4]{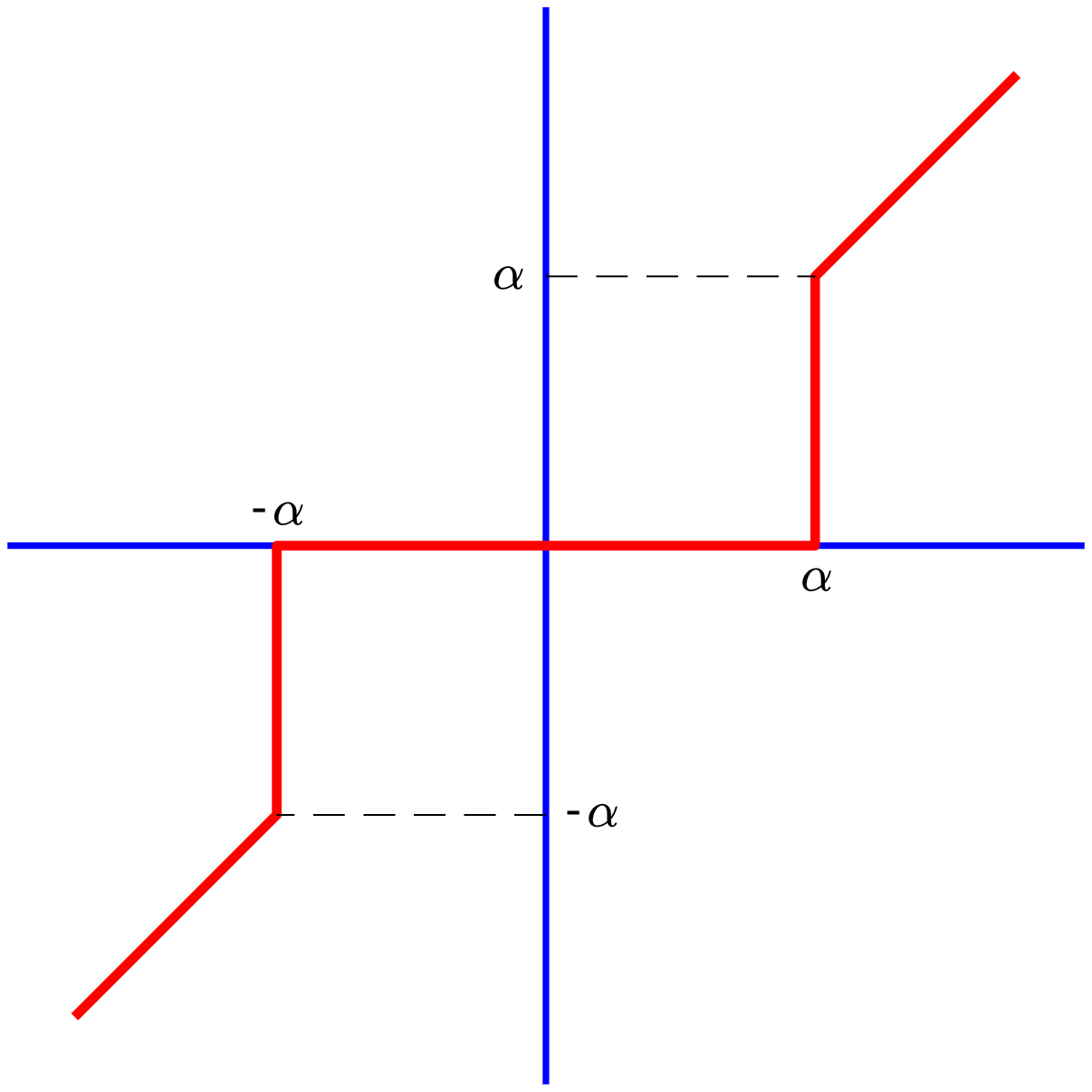} & \includegraphics[scale=0.4]{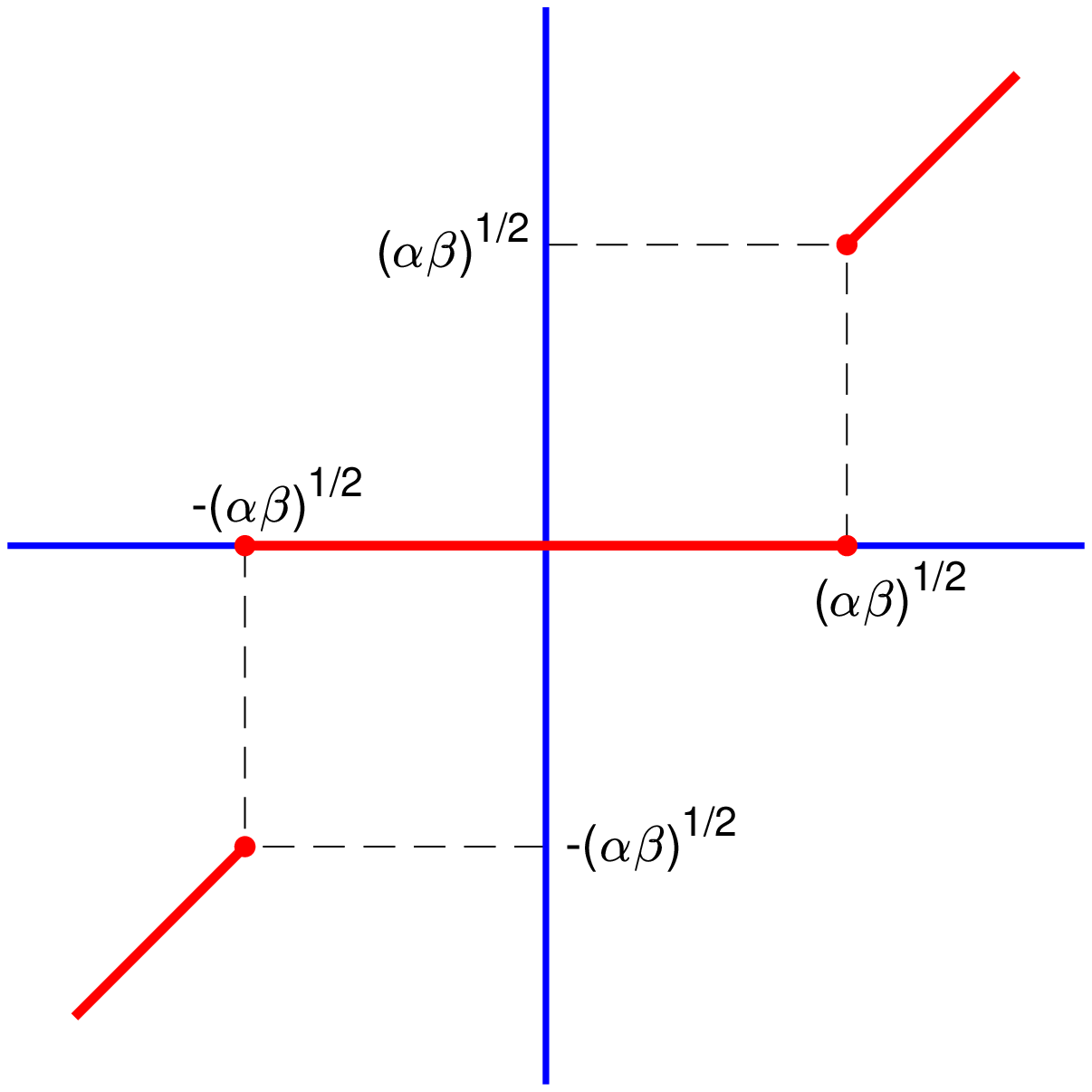}\\
(a) & (b) & (c)
\end{tabular}
\caption{Typical shapes of the proximity operator of $| \cdot |_\alpha$ for (a) $\beta < \alpha$, (b) $\beta = \alpha$, (c) $\beta > \alpha$. The sparsity threshold and the thresholding behavior depend on the relationship between $\alpha$ and $\beta$.}
\label{fig:prox-|x|}
\end{figure}

To end this example, we give several remarks on the proximity operators of $\mathrm{prox}_{\alpha f}$ and $\mathrm{prox}_{\beta f_{\alpha}}$ as follows:
\begin{itemize}
\item Note that $\partial f(0)=[-1,1]$. The results given in \eqref{eq:beta<alpha} (for $\beta<\alpha$) and \eqref{eq:beta=alpha} (for $\beta=\alpha$)  exactly match the first two statements of Theorem~\ref{thm:major}. For $\beta>\alpha$, the $\mathrm{prox}_{\beta f_\alpha}(x)=0$ for all $x\in [-\sqrt{\alpha\beta},\sqrt{\alpha\beta}]$ which includes the interval $[-\alpha, \alpha]=\alpha \partial f(0)$ as indicated in the third statement of Theorem~\ref{thm:major}.

\item The operator $\prox_{\alpha f}$ forces its variable to zero when the absolute value is less than a given threshold, and otherwise reduces the variable, in absolute value, by the amount of the threshold. Like $\prox_{\alpha f}$, $\prox_{\beta f_\alpha}$ forces its variable to zero when the absolute value is less than a given threshold, but it fixes variables whose absolute value exceeds a certain threshold.

%\item For any $x \in \mathbb{R}$, we have $\lim_{\alpha \rightarrow \infty} \mathrm{prox}_{\beta f_{\alpha}}(x)=\mathrm{prox}_{\beta f}(x)$. This indicates that the firm thresholding operator $\mathrm{prox}_{\beta f_{\alpha}}$ tends to the soft thresholding operator $\mathrm{prox}_{\beta f}$ when $\alpha$ goes to infinity.

\item For $\beta \ge \alpha$ the proximity operator $\mathrm{prox}_{\beta f_{\alpha}}$ is almost identical to the hard threshold operator. Let $|\cdot|_0$ be the $\ell_0$ ``norm'' on $\mathbb{R}$, that is, $|x|_0$ equals $1$ if $x$ is nonzero, $0$ otherwise. The proximity operator of $|\cdot|_0$ with parameter $\gamma$ at $x$ is
    $$
    \prox_{\gamma |\cdot|_0}(x)=\left\{
                           \begin{array}{ll}
                           \{0\}, & \hbox{if $|x|< \sqrt{2\gamma}$;}\\
                           \{0,x\}, & \hbox{if $|x| = \sqrt{2\gamma}$;}\\
                            \{x\}, & \hbox{if $|x| > \sqrt{2\gamma}$,}
                         \end{array}
\right.
    $$
which is called the hard thresholding operator with threshold $\sqrt{2\gamma}$.  We can see that $\prox_{\gamma |\cdot|_0}=\mathrm{prox}_{\beta f_{\alpha}}$ as long as $2\gamma=\alpha \beta$ and $\beta>\alpha$. It is interesting that although $|\cdot|_0$ is discontinuous and $f_\alpha$ is continuous, they have the same proximity operator. Moreover, by fixing $\alpha$ and varying the parameter $\beta$,  the proximity operator $\mathrm{prox}_{\beta f_{\alpha}}$ changes from the firm thresholding operator to the hard thresholding operator.

\end{itemize}

%%%%%%%%%%%%%%ReLU%%%%%%%%%%%%%%%%%%%%%%
\subsection{Example 2: ReLU function}

The ReLU (Rectified Linear Unit) function  on $\mathbb{R}$ is
$$
f(x):=\max\{0, x\},
$$
which is a special case of the piecewise quadratic function in \eqref{def:quad} with $a_1=b_1=a_2=0$ and $b_2=1$. The proximity operator and the Moreau envelope of $f$ with parameter $\alpha>0$ are
$$
\mathrm{prox}_{\alpha f} (x)=\min\{x, \max\{0, x-\alpha\}\} \quad \mbox{and} \quad
\mathrm{env}_{\alpha f} (x) = \left\{
                                \begin{array}{ll}
                                0   & \hbox{if $x \leq 0$;}\\
                                  \frac{1}{2\alpha}x^2, & \hbox{if $0 \leq x \leq \alpha$;} \\
                                  x-\frac{1}{2}\alpha, & \hbox{if $x \ge \alpha$,}
                                \end{array}
                              \right.
$$
respectively. By~\eqref{def:falpha}, $f_\alpha(x) = f(x) - \env_{\alpha f}(x)$ is
\begin{align*}
f_\alpha (x) = \begin{cases}
0, & \text{if } x < 0\\
x - \frac{1}{2\alpha}x^2, & \text{if } 0 \leq x \leq \alpha \\
\frac{\alpha}{2}, & \text{if } x>\alpha
\end{cases}
%&&
%\partial f_\alpha (x) = \begin{cases}
%0, & \text{if } x < 0 \quad \mbox{or} \quad x>\alpha;\\
%[0, 1], & \text{if } x = 0;\\
%1 - \frac{1}{\alpha} x, & \text{if } 0 < x \leq \alpha.
%\end{cases}
\end{align*}

Figure~\ref{figure:ex2-functions}(a) depicts the graphs of $f$ and $\env_{\alpha f}$ while Figure~\ref{figure:ex2-functions}(b) presents the function $f_\alpha$. The graph of  $\prox_{\alpha f}$ is given in Figure~\ref{figure:ex2-f-prox}.

\begin{figure}[h]
\centering
\begin{tabular}{cccc}
\includegraphics[scale=0.35]{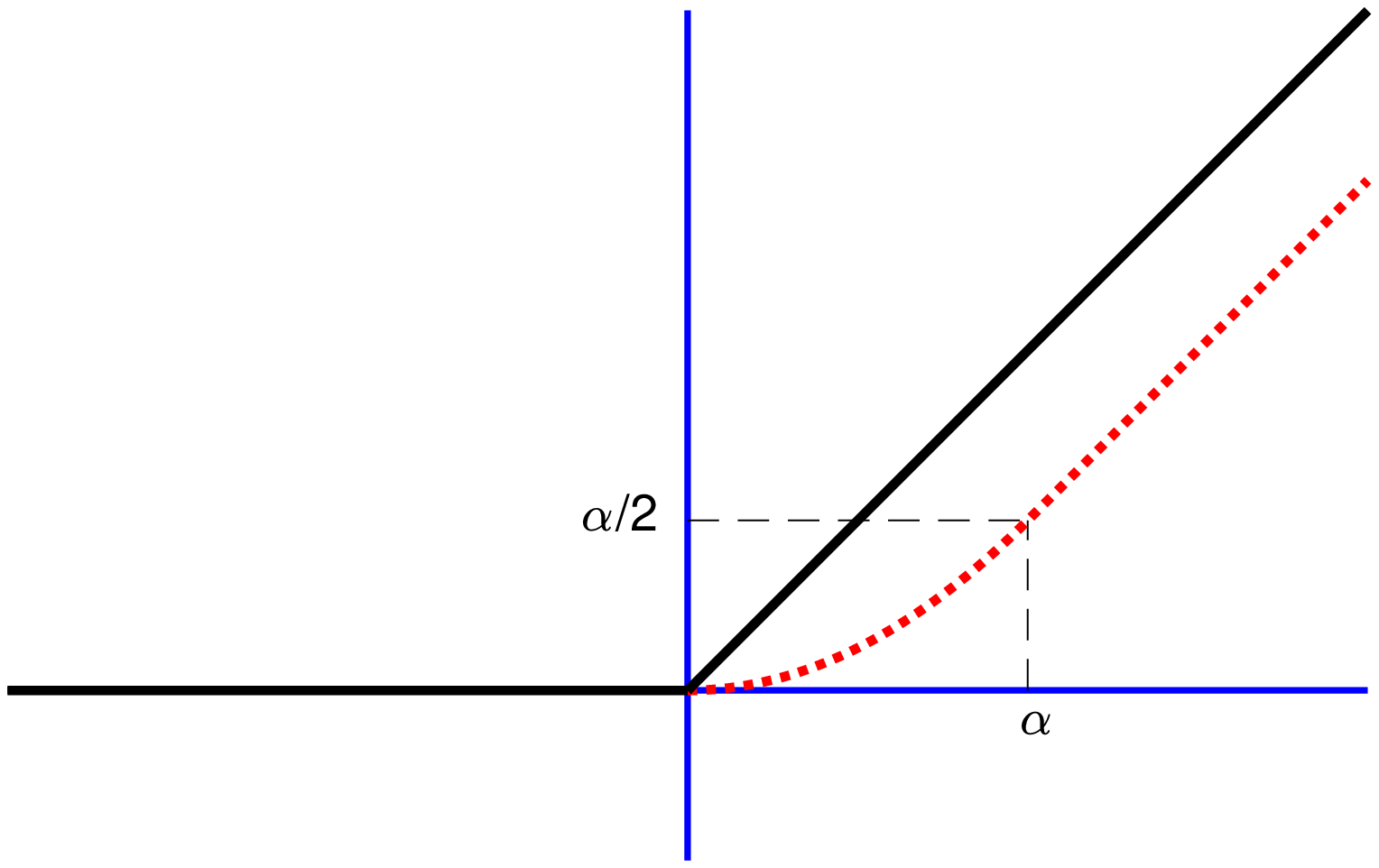} & & & \includegraphics[scale=0.35]{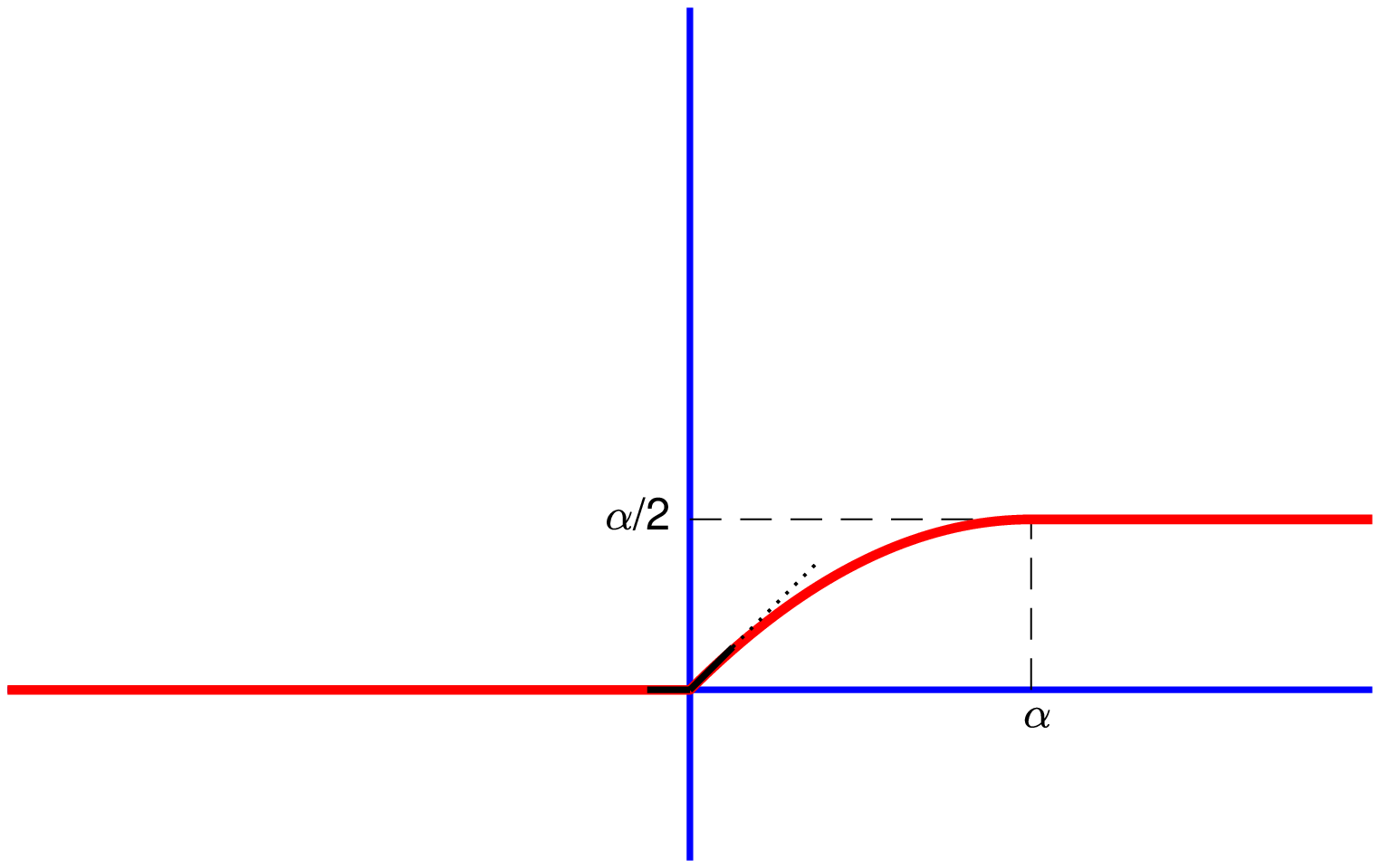} \\
(a) & & & (b)
\end{tabular}
\caption{Example 2. (a) The graphs of $f$ (solid), $\env_{\alpha f}$ (dotted), and (b) their difference $f_\alpha = f - \env_{\alpha f}$. The singularity of $f_\alpha$ at zero is emphasized in black (solid-dotted).}
\label{figure:ex2-functions}
\end{figure}

\begin{figure}[h]
\centering
\includegraphics[scale=0.35]{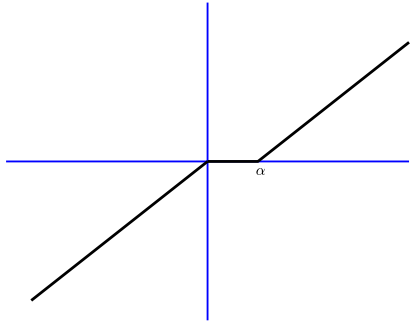}
\caption{Example 2. The typical shape of $\prox_{\alpha f}$. The parameter $\alpha$ is the sparsity threshold.}
\label{figure:ex2-f-prox}
\end{figure}

As in example 1, the expression of $\mathrm{prox}_{\beta f_{\alpha}}$ depends on  the relative values of $\alpha$ and $\beta$. If $\beta<\alpha$,
\begin{equation}\label{ex2:beta<alpha}
\prox_{\beta f_\alpha} (x) =
\begin{cases}
x, & \text{if }  x \leq 0 \text{ or } x \ge \alpha; \\
0, & \text{if } 0 \leq x \leq \beta ; \\
\frac{\alpha(x-\beta)}{\alpha - \beta};  & \text{if }  \beta \leq x \leq \alpha.
\end{cases}
\end{equation}
If $\beta = \alpha$,
\begin{equation}\label{ex2:beta=alpha}
\prox_{\beta f_\alpha} (x) = \begin{cases}
x, & \text{if }  x \leq 0 \text{ or } x>\alpha; \\
0, & \text{if } 0 \leq x < \alpha; \\
[0, \alpha] & \text{if } x = \alpha.
\end{cases}
\end{equation}
Finally, if $\beta > \alpha$,
\begin{equation}\label{ex2:beta>alpha}
\prox_{\beta f_\alpha} (x) = \begin{cases}
x,  & \text{if }  x \leq 0 \text{ or } x>\sqrt{\alpha \beta}; \\
0, & \text{if }  0 \leq x < \sqrt{\alpha \beta}; \\
\{0, \sqrt{\alpha \beta}\}, & \text{if } x = \sqrt{\alpha \beta}.
\end{cases}
\end{equation}

\begin{figure}[h]
\begin{centering}
\begin{tabular}{ccc}
\includegraphics[scale=0.35]{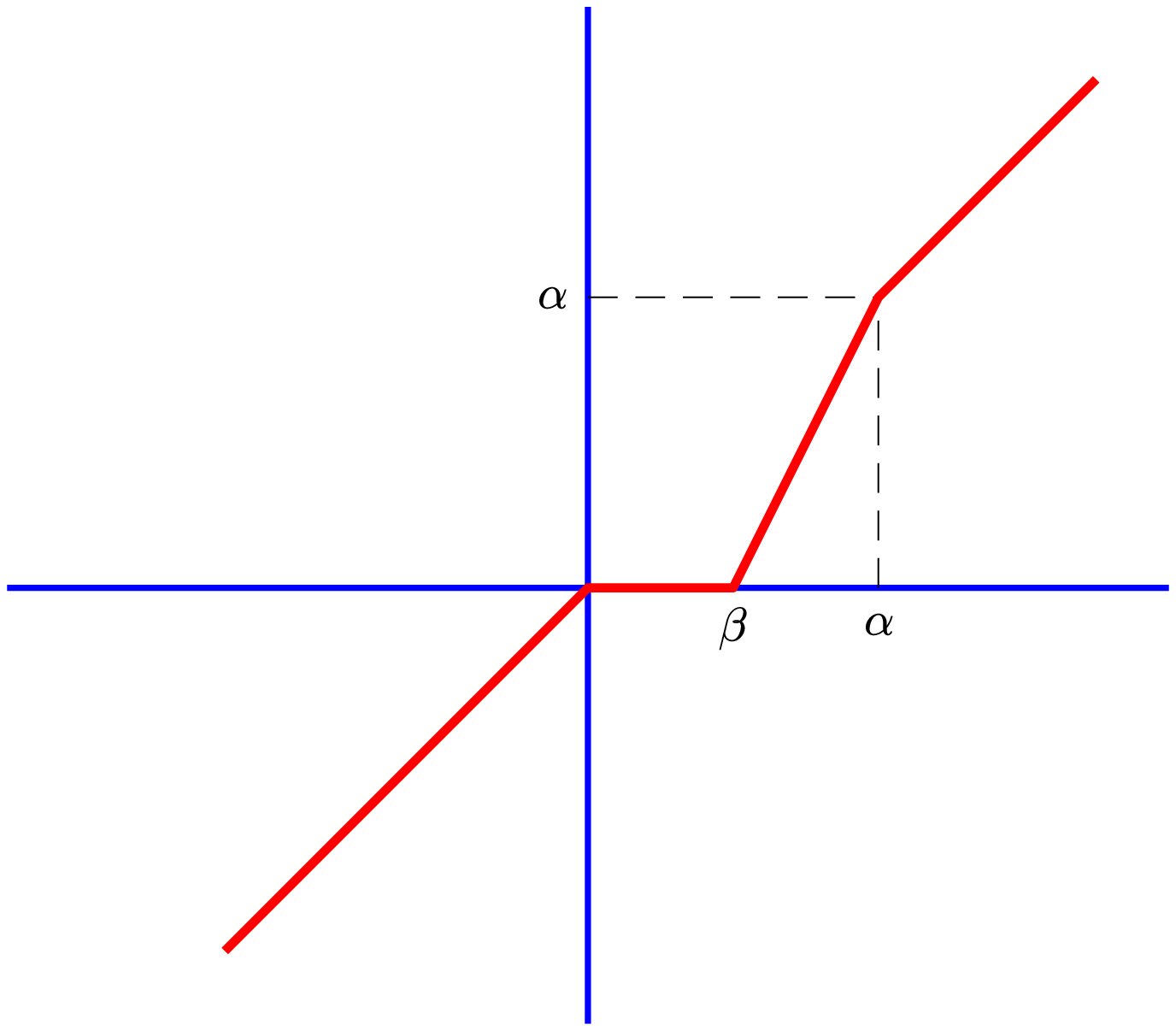} & \includegraphics[scale=0.35]{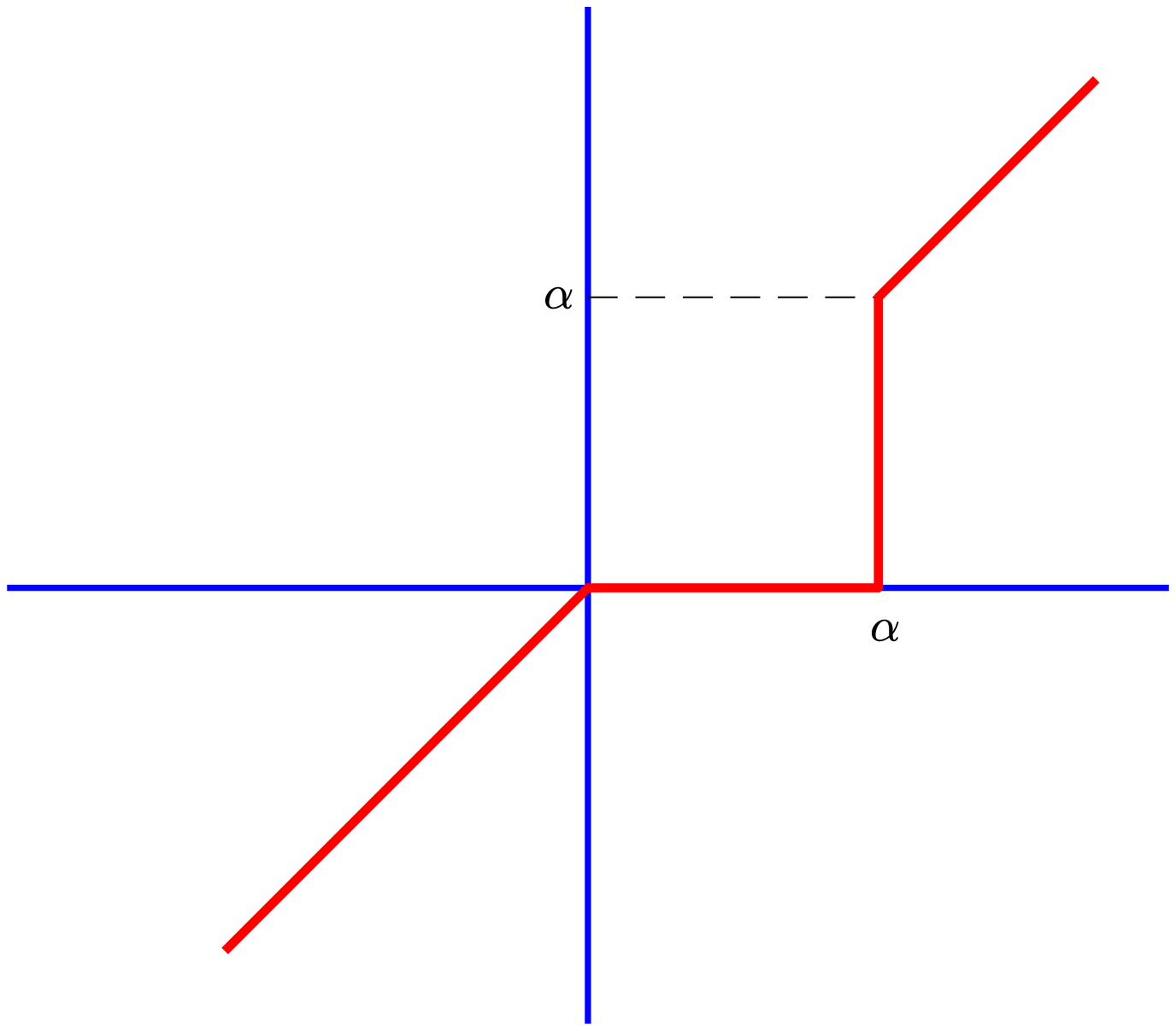} & \includegraphics[scale=0.35]{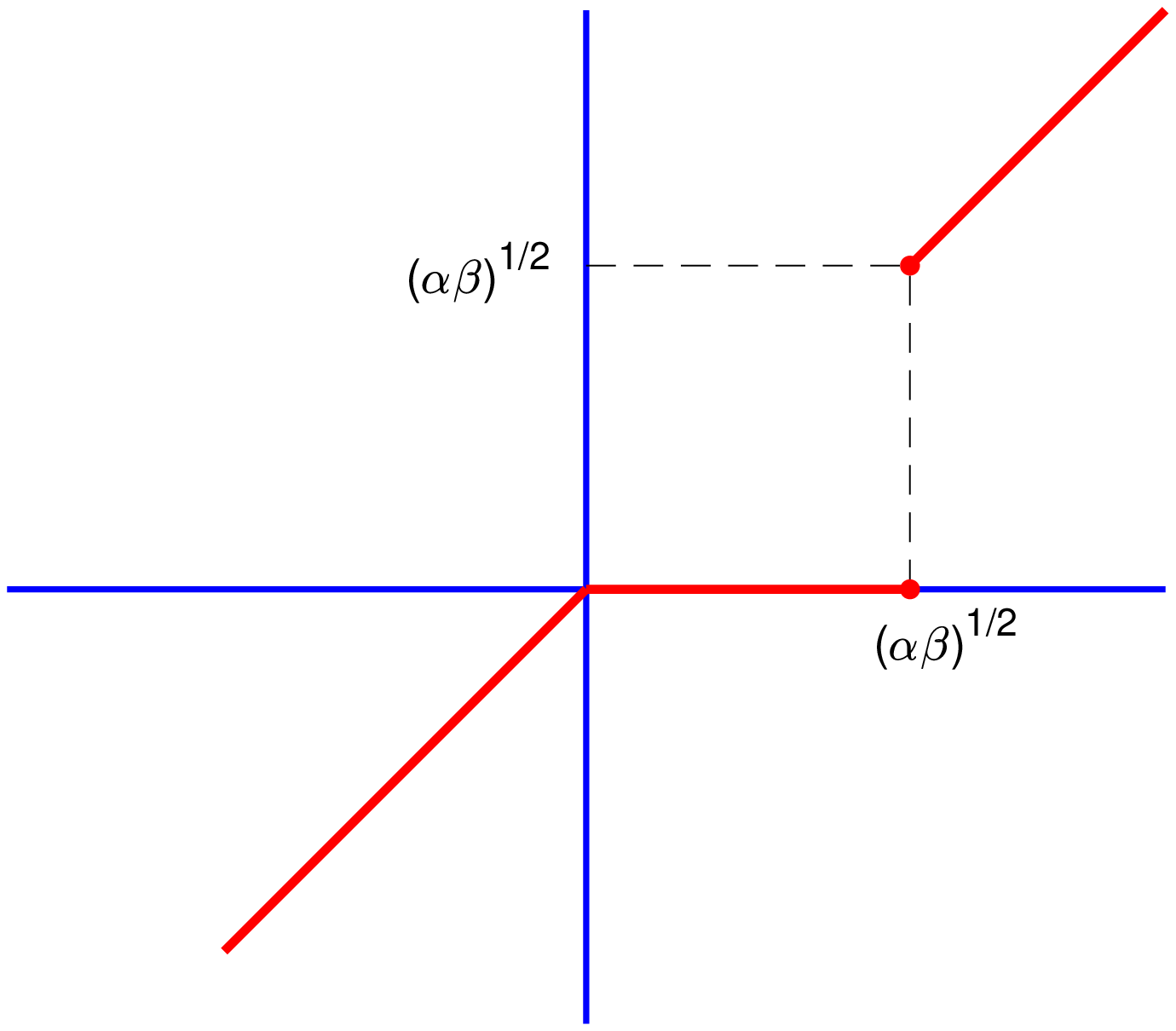} \\
(a) & (b) & (c)
\end{tabular}
\caption{Example 2. Typical shapes of the proximity operator of $f_\alpha$ for (a) $\beta < \alpha$; (b) $\beta = \alpha$; and (c) $\beta > \alpha$.}
\end{centering}
\end{figure}

Note that $\partial f(0)=[0,1]$. The results given in \eqref{ex2:beta<alpha} (for $\beta<\alpha$) and \eqref{ex2:beta=alpha} (for $\beta=\alpha$)  exactly match the first two statements of Theorem~\ref{thm:major}. For $\beta>\alpha$, equation~\eqref{ex2:beta>alpha} shows that $\mathrm{prox}_{\beta f_\alpha}(x)=0$ for all $x\in [0,\sqrt{\alpha\beta}]$, which includes the interval $[0, \alpha]=\alpha \partial f(0)$ as indicated in the third statement of Theorem~\ref{thm:major}.

%%%%%%%%%%Elastic net%%%%%%%%%%%%%%
\subsection{Example 3: Elastic Net}

The elastic net is a regularized regression method in data analysis that linearly combines the $\ell_1$ and $\ell_2$ penalties of the LASSO and ridge methods. In this example, we consider a special case of the elastic net in $\mathbb{R}$:
$$
f(x)=\frac{1}{2}x^2 + |x|.
$$
This is an instance of the piecewise quadratic function given in \eqref{def:quad} with $a_1=a_2=1$, $b_1=-1$ and $b_2=1$.
Clearly,  $f$ is nondifferentiable at the origin with $\mathrm{argmin}_{x \in \mathbb{R}} f(x) =\{0\}$. Moreover, $\partial f (0) =\partial |\cdot|(0)=[-1,1]$.

The proximity operator and the Moreau envelope of $f$ with parameter $\alpha>0$ are
$$
\prox_{\alpha f}(x) = \max \left\{0, \frac{1}{\alpha + 1}(|x|- \alpha)\right\} \sign(x) \quad \mbox{and} \quad
\env_{\alpha f} (x) = \begin{cases} \frac{1}{2\alpha}x^2, & \text{if } |x| \leq \alpha; \\
\frac{1}{\alpha + 1}(
\frac{1}{2}x^2 + |x| - \frac{\alpha}{2}), & \text{if } |x| \geq \alpha, \\
\end{cases}
$$
respectively.

The graphs of $f$ and $\env_{\alpha f}$ are plotted in Figure~\ref{figure:ex3-functions} (a). The graph of $\prox_{\alpha f}$  is plotted in Figure~\ref{figure:ex3-functions} (b). As in the case of the absolute value function, $\prox_{\alpha f}$ sends all values between $\alpha$ and $-\alpha$ to zero. Unlike the absolute value, it also contracts elements outside of this interval toward the origin.

\begin{figure}[h]
\centering
\begin{tabular}{cccc}
\includegraphics[scale=0.35]{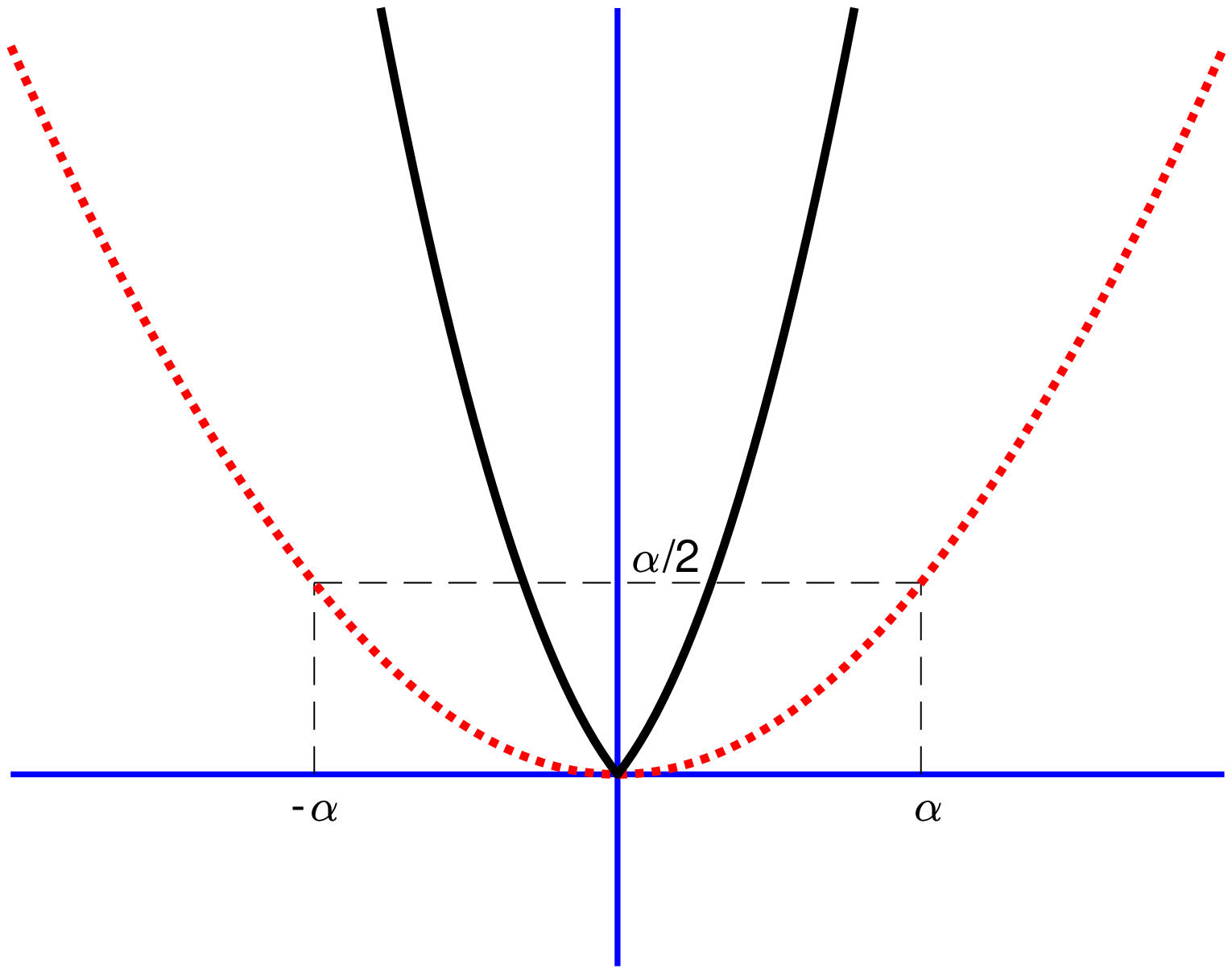}& & & \includegraphics[scale=0.35]{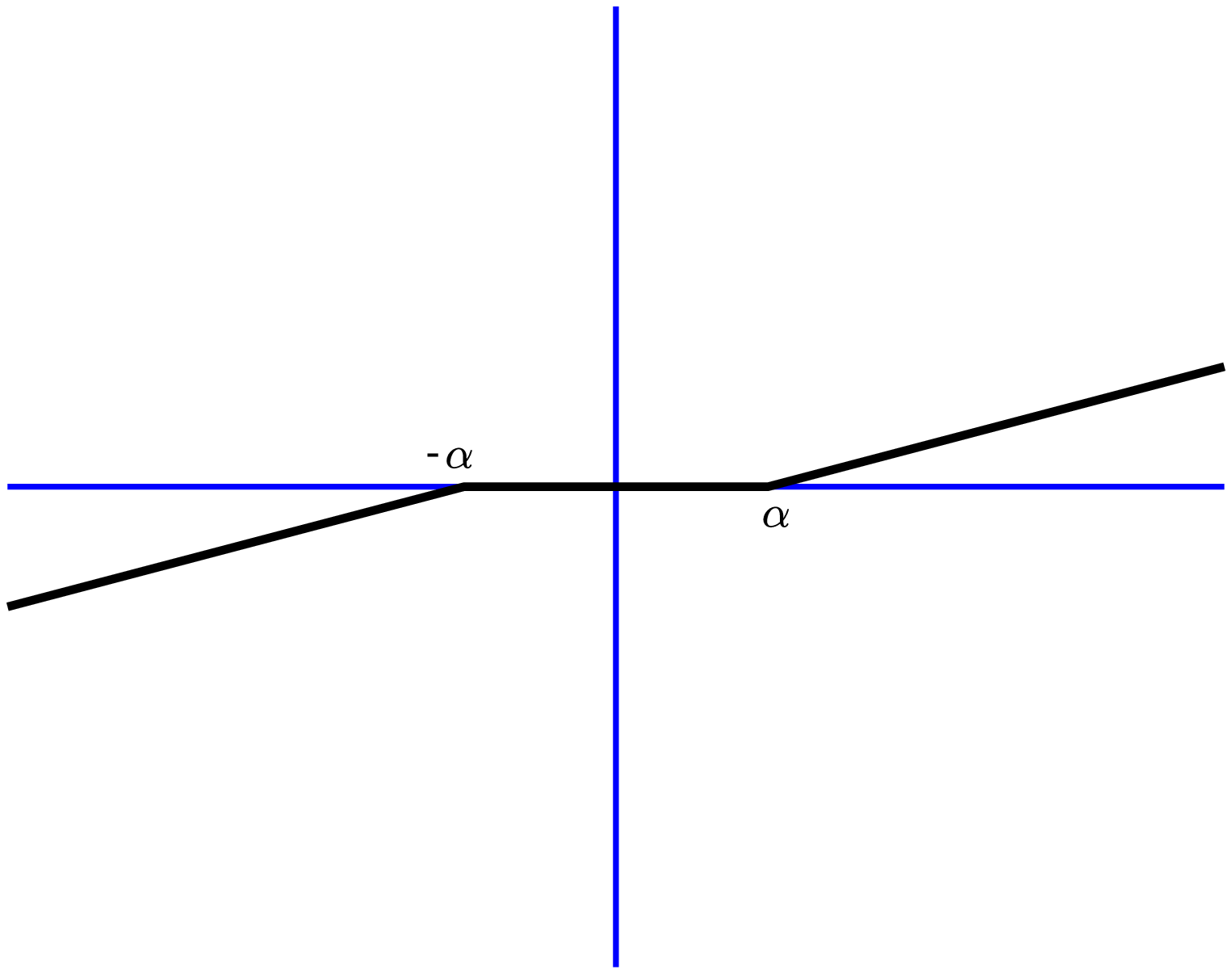} \\
(a)& & & (b)
\end{tabular}
\caption{Example 3. (a) The graphs of  $f$ (solid) and $\env_{\alpha f}$ (dotted); and (b) the graph of $\prox_{\alpha f}$.}\label{figure:ex3-functions}
\end{figure}

%\begin{figure}[h]
%\centering
% \includegraphics[scale=0.3]{figure_elasticprox.eps}
% \caption{The typical shape of $\prox_{\alpha f}$.}
%\end{figure}

Now $f_\alpha$, the difference between $f$ and its Moreau envelope $\env_{\alpha f}$, is
$$
f_\alpha(x) = \begin{cases}
\frac{\alpha-1}{2\alpha}x^2 + |x|, & \text{if } |x| \leq \alpha; \\
\frac{\alpha}{2(\alpha + 1)}x^2 +\frac{\alpha}{\alpha+1}|x|+\frac{\alpha}{2(\alpha+1)}, & \text{if } |x| \geq \alpha.
\end{cases}
$$
We remark that $f_\alpha$ is convex when $\alpha \ge 1$ and nonconvex when $\alpha < 1$. The graph of $f_\alpha$ for $\alpha \ge 1$ and $\alpha<1$ are shown in Figure~\ref{figure:ex3-falpha}(a) and (b), respectively.
\begin{figure}[h]
\centering
\begin{tabular}{cccc}
\includegraphics[scale=0.35]{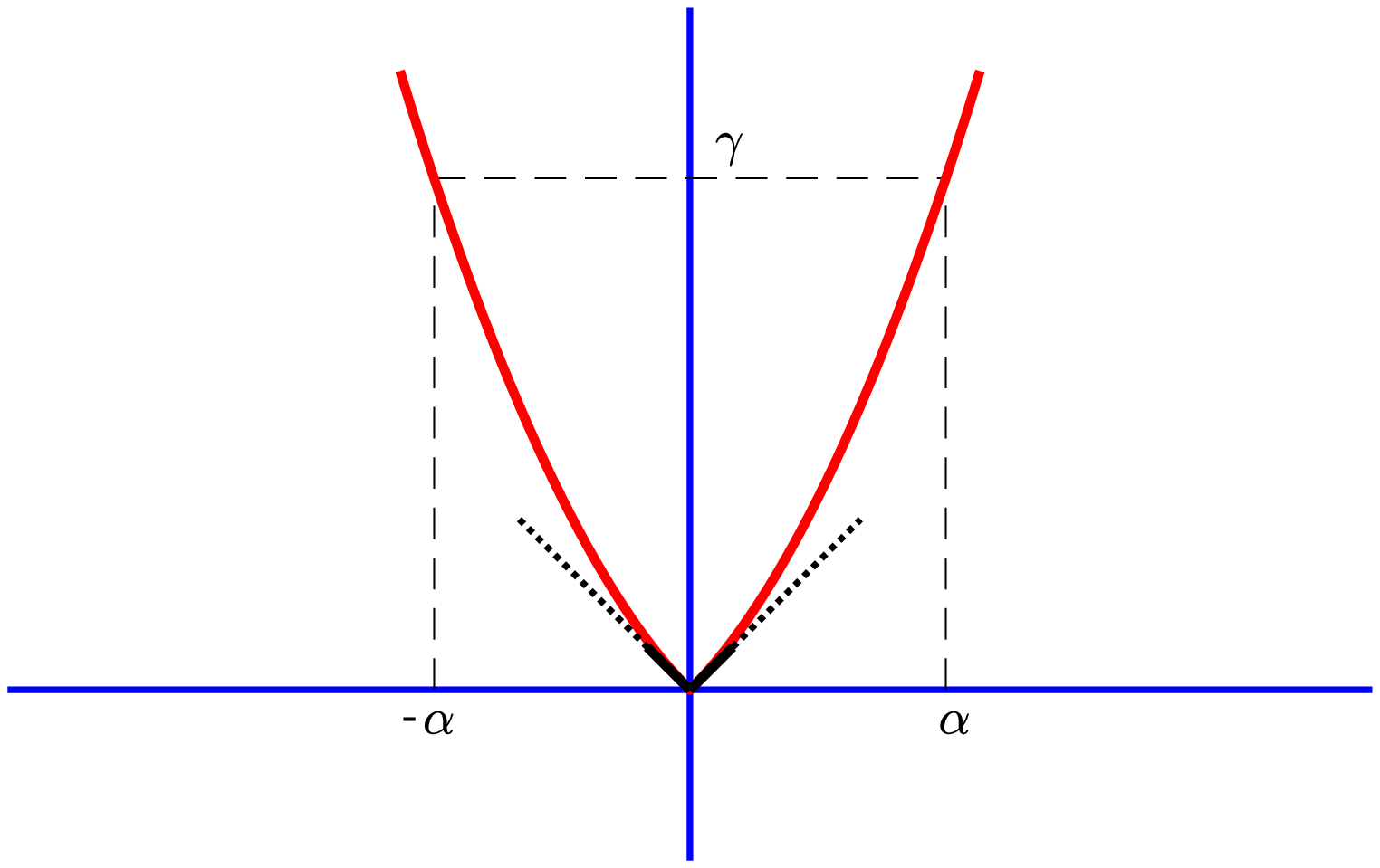} & & & \includegraphics[scale=0.35]{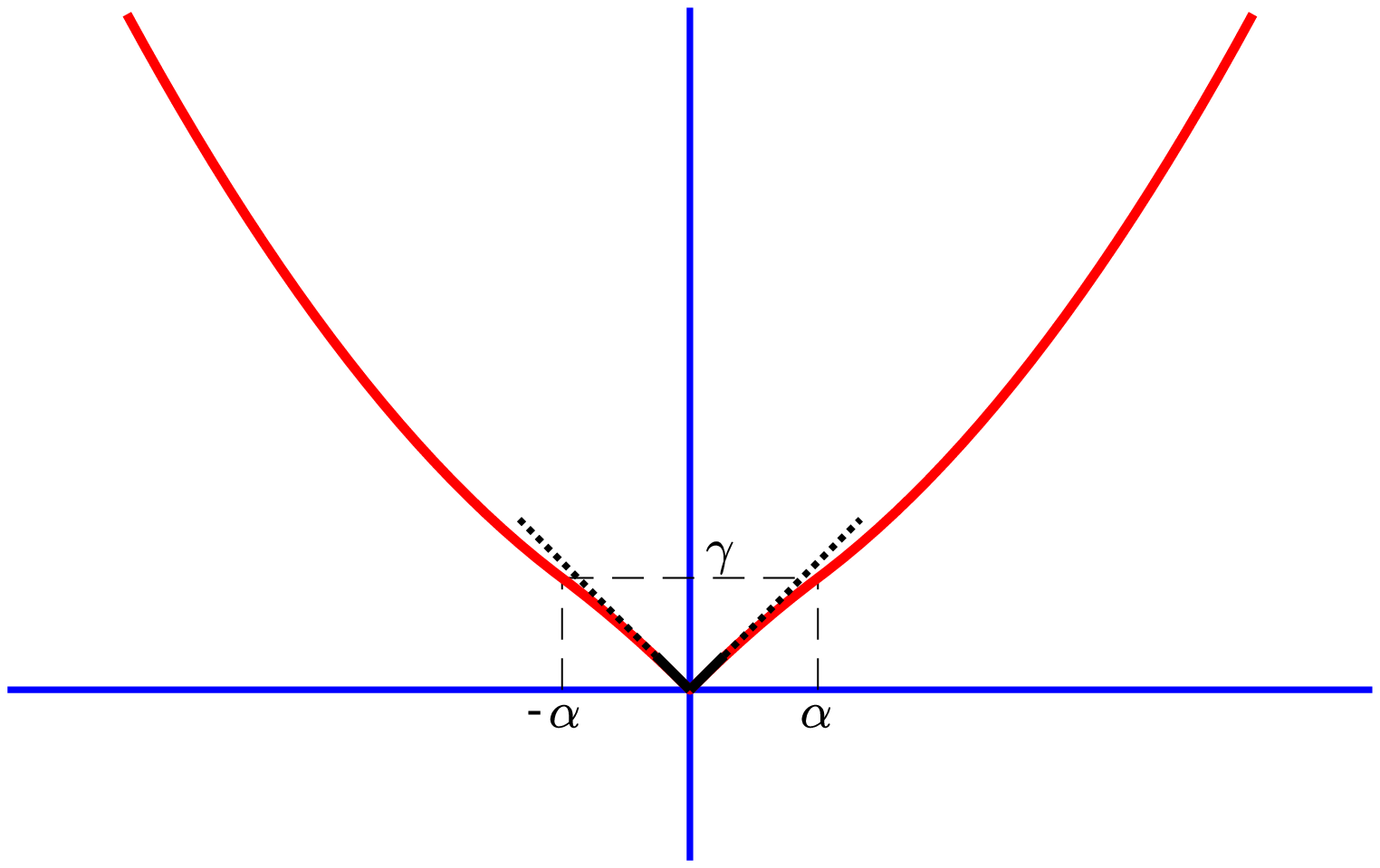}\\
(a) & & & (b)
\end{tabular}
\caption{Example 3. The graph of $f_\alpha$ when (a) $\alpha \geq 1$ and (b) $\alpha < 1$. The singularity of $f_\alpha$ at zero is emphasized in black (solid-dotted).}
\label{figure:ex3-falpha}
\end{figure}

According to the discussion given in subsection~\ref{subsec:quadratic}, we consider three cases: $\beta(\alpha - 1) + \alpha > 0$, $\beta(\alpha - 1) + \alpha = 0$, and $\beta(\alpha - 1) + \alpha < 0$. These cases are equivalent to $\alpha (\beta +1) > \beta$, $\alpha (\beta + 1) = \beta$, and $\alpha (\beta + 1) < \beta$ respectively. Recall that these cases correspond to the convexity (or lack thereof) of $f_\alpha(u) + \frac{1}{2\beta}(u-x)^2$ for $u$ close to zero.

\underline{Case 1: $\alpha(\beta+1)>\beta$.}  In this case, by Lemma~\ref{lemma:b2>0:case1} we have
\begin{equation}\label{ex4:beta<=alpha}
\prox_{\beta f_\alpha}(x) = \begin{cases} 0 & \text{if } |x| \leq \beta;\\
\frac{\alpha}{\alpha \beta - \beta + \alpha}(x-\beta\sign(x)) & \text{if } \beta \leq |x| \leq \alpha(\beta+1);\\
\frac{\alpha +1}{\alpha\beta + \alpha + 1}(x- \frac{\alpha\beta}{\alpha+1}\sign(x)) & \text{if } \alpha(\beta + 1) \leq |x|.
\end{cases}
\end{equation}

\underline{Case 2: $\alpha(\beta+1)=\beta$.} By Lemma~\ref{lemma:b2>0:case2} we have
\begin{equation}\label{ex4:beta>alpha1}
\prox_{\beta f_\alpha}(x) = \begin{cases} 0 & \text{if } |x| \leq \beta; \\
[0, \alpha]\sign(x) & \text{if } |x| = \beta; \\
\frac{\alpha + 1}{\alpha\beta + \alpha + 1}(x - \frac{\alpha \beta}{\alpha + 1} \sign(x)) & \text{if } \beta \leq |x|.
\end{cases}
\end{equation}

\underline{Case 3: $\alpha(\beta+1)<\beta$.} Define
\begin{equation}\label{eq:define-tau}
\tau = \frac{\alpha \beta}{\alpha + 1} + \frac{\sqrt{\alpha \beta(\alpha \beta + \alpha + 1)}}{\alpha + 1}.
\end{equation}
as in Lemma~\ref{lemma:b2>0:case3}. Then we have
\begin{equation}\label{ex4:beta>alpha2}
\prox_{\beta f_\alpha}(x) = \begin{cases} 0 & \text{if } |x| \leq \tau; \\
\{0,\omega\}& \text{if } |x|=\tau;\\
\frac{(\alpha + 1)x -\alpha \beta\sign(x)}{\alpha \beta + \alpha + 1}, & \text{if } |x|>\tau,
\end{cases}
\end{equation}
where $\omega = \frac{(\alpha+1)\tau-\alpha \beta}{\alpha\beta+\alpha+1}$.
The graphs of $\prox_{\beta f_\alpha}$ in the above three cases are plotted in Figure~\ref{figure:ex4-f-prox}.
\begin{figure}[h] \label{figure:ex4-f-prox}
\centering
\begin{tabular}{ccc}
\includegraphics[scale=0.35]{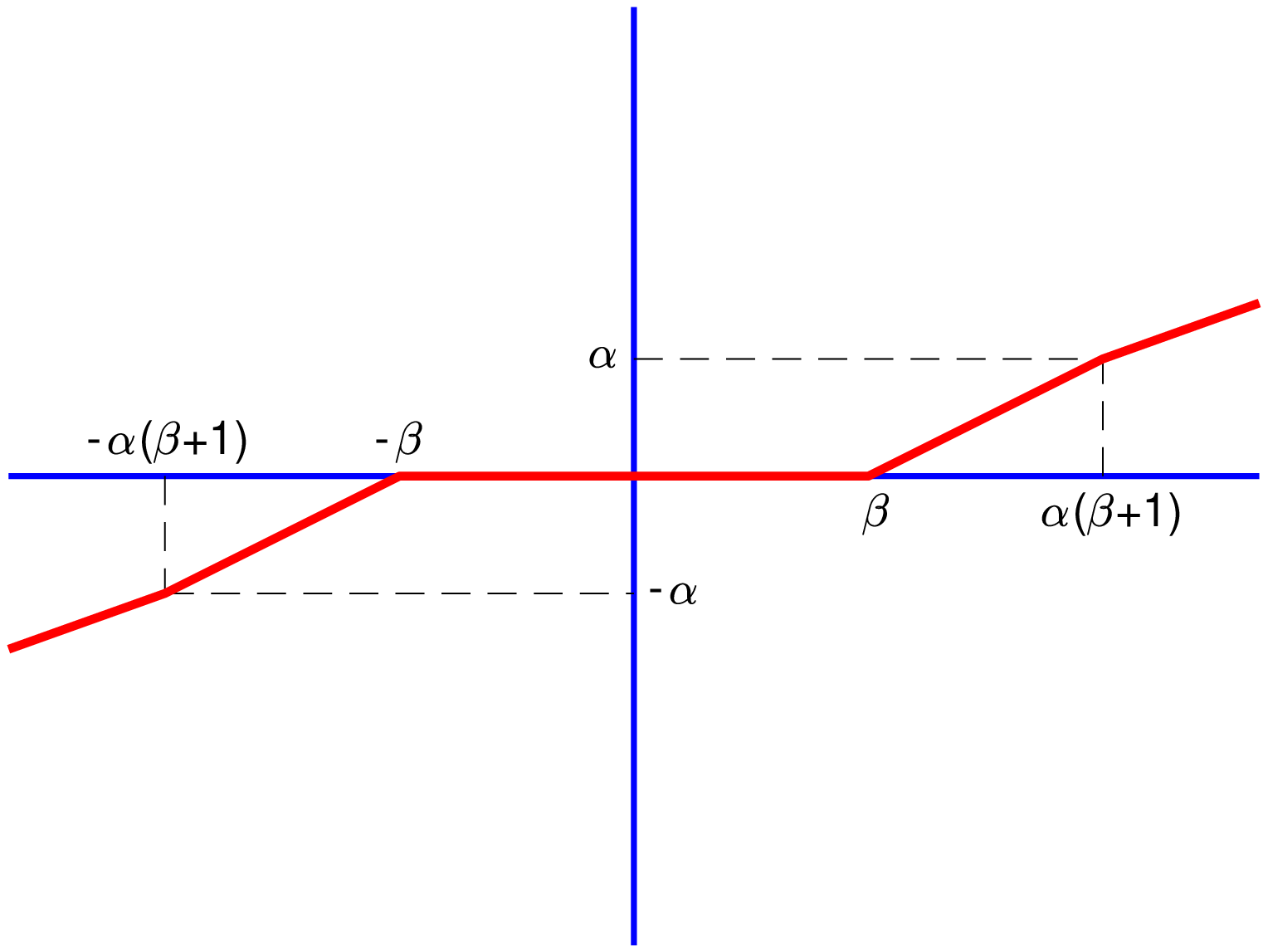} & \includegraphics[scale=0.35]{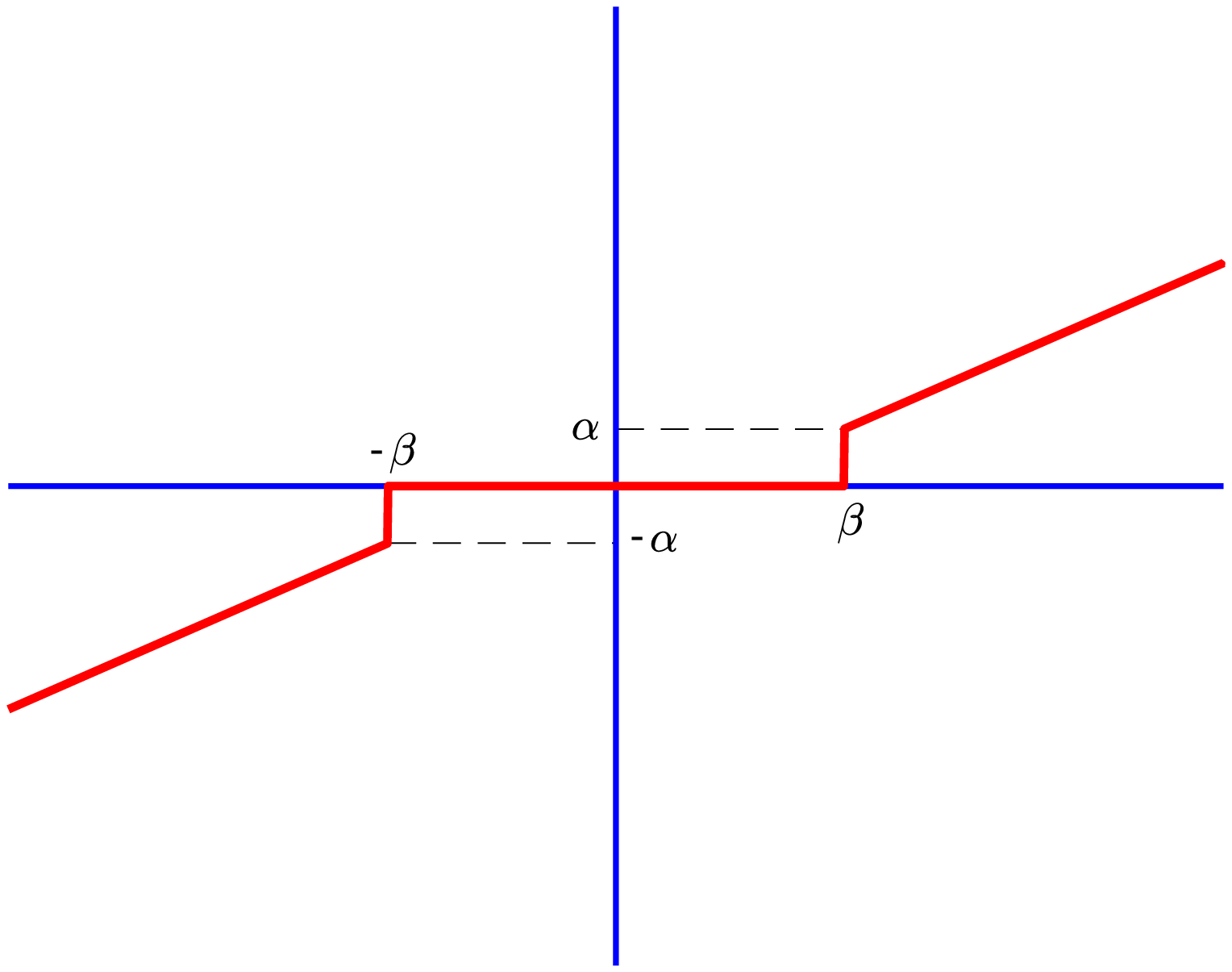}&\includegraphics[scale=0.35]{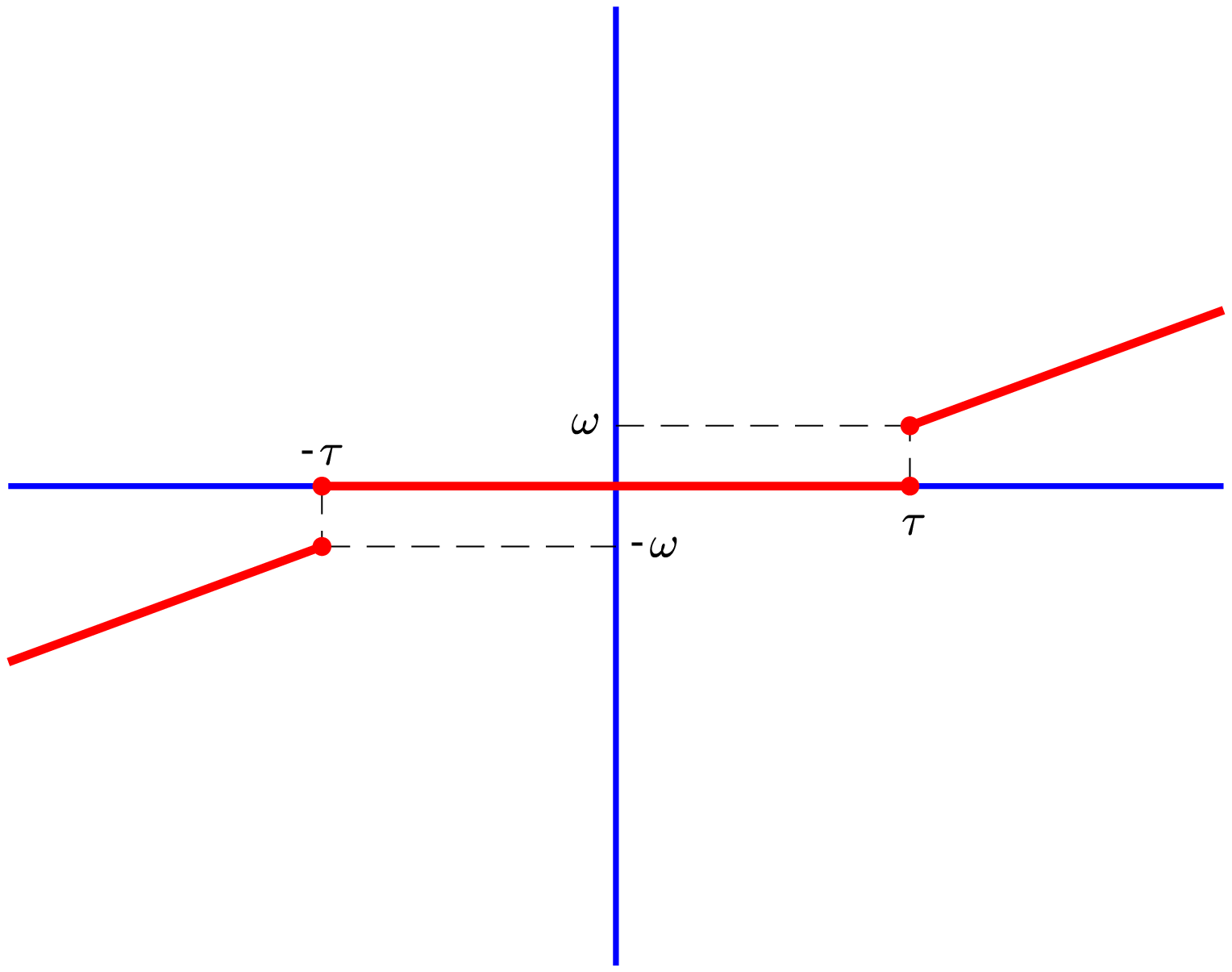}\\
(a) & (b) & (c)
\end{tabular}
\caption{Example 3. Typical shapes of $\prox_{\beta f_\alpha}$ when (a) $\alpha(\beta + 1) > \beta$, (b) $\alpha(\beta+1) = \beta$, and (c) $\alpha(\beta+1) < \beta$.}

\end{figure}

Below are some comments on this example.
\begin{itemize}
\item The function $f_\alpha$ in the first two examples is nonconvex for any $\alpha>0$, however, by Proposition~\ref{prop:semiconvex} it is convex if $\alpha \ge 1$ due to  our elastic net function $f$ being $1$-strongly convex.

\item The computation of the proximity operator $\prox_{\beta f_\alpha}$ is discussed under three different situations, namely, $\alpha(\beta + 1) > \beta$, $\alpha(\beta + 1) = \beta$, and $\alpha(\beta + 1) < \beta$. These situations are quite nature from Proposition~\ref{prop:semiconvex}. Since $f$ is $1$-strongly convex, hence, the function $f_\alpha+\frac{1}{2\beta}(\cdot-x)^2$ is $(1+\beta^{-1}-\alpha^{-1})$-strongly convex if $\alpha(\beta + 1) > \beta$, convex if $\alpha(\beta + 1) = \beta$, and $(\alpha^{-1}-1-\beta^{-1})$-semiconvex if $\alpha(\beta + 1) <\beta$.

\item For the case of $\beta\le \alpha$, we know that $\alpha(1+\beta)>\beta$, so the proximity operator given \eqref{ex4:beta<=alpha} covers both statements 1 and 2 in Theorem~\ref{thm:major}.

\item For the case of $\beta > \alpha$, there are three possible related cases. If $\alpha<\beta <\alpha(\beta+1)$ (resp. $\alpha<\beta = \alpha(\beta+1)$), the proximity operator given \eqref{ex4:beta<=alpha} (resp. \eqref{ex4:beta>alpha1}) shows that this operator vanishes all elements in $\beta \partial f(0)=[-\beta, \beta] \supset \alpha \partial f(0)$, fulfilling the third statement of Theorem~\ref{thm:major}. If $\beta>\alpha(\beta+1)$, we know that $\alpha<1$, $\beta>\frac{\alpha}{1-\alpha}$, and $\tau$ defined in \eqref{eq:define-tau} satisfying
    $$
    \tau=\frac{\alpha \beta}{\alpha + 1} + \frac{\sqrt{\alpha \beta(\alpha \beta + \alpha + 1)}}{\alpha + 1}>\frac{\alpha^2}{1-\alpha^2}+\frac{\alpha}{1-\alpha^2}>\alpha.
    $$
    Hence, the proximity operator given \eqref{ex4:beta>alpha2} annihilates all elements in $\tau \partial f(0) \supset \alpha \partial f(0)$, once again fulfilling the third statement of Theorem~\ref{thm:major}.
\end{itemize}

%%%%%%%%%Indicator%%%%%%%%%%%%
\subsection{Example 4: Absolute value on an interval centered at the origin}

Let $\lambda$ be a positive parameter. The absolute function on the interval $[-\lambda, \lambda]$ centered at the origin is
$$
f(x):=|x| + \iota_{[-\lambda, \lambda]}(x),
$$
which is a special case given in \eqref{def:ftidle} with $a_1=a_2=0$, $b_1=-1$, $b_2=1$, and $C=[-\lambda, \lambda]$.
Its proximity operator and Moreau envelope with parameter $\alpha$ at point $x$, respectively, are
\begin{equation*}
\prox_{\alpha f}(x) = \begin{cases}
0, & \text{ if $|x| \leq \alpha$;} \\
\sign(x)(|x| - \alpha), & \text{ if $\alpha < |x| \leq \alpha + \lambda$;} \\
\lambda \sign(x), &\text{ if $\alpha + \lambda < |x|$;}
\end{cases}
\end{equation*}
and
\begin{equation*}
\env_{\alpha f}(x) = \begin{cases}
|x| - \frac{\alpha}{2} + \frac{1}{2\alpha}(|x| - \alpha)^2, & \text{ if $|x| \leq \alpha$;} \\
|x| - \frac{\alpha}{2}, & \text{ if $\alpha < |x| \leq \alpha + \lambda$;} \\
|x| - \frac{\alpha}{2} + \frac{1}{2\alpha}(|x| - (\lambda + \alpha))^2, & \text{ if $\alpha + \lambda < |x|$.}
\end{cases}
\end{equation*}

\begin{figure}[h]\label{figure:ex3-f-prox}
\centering
\begin{tabular}{ccc}
\includegraphics[scale=0.35]{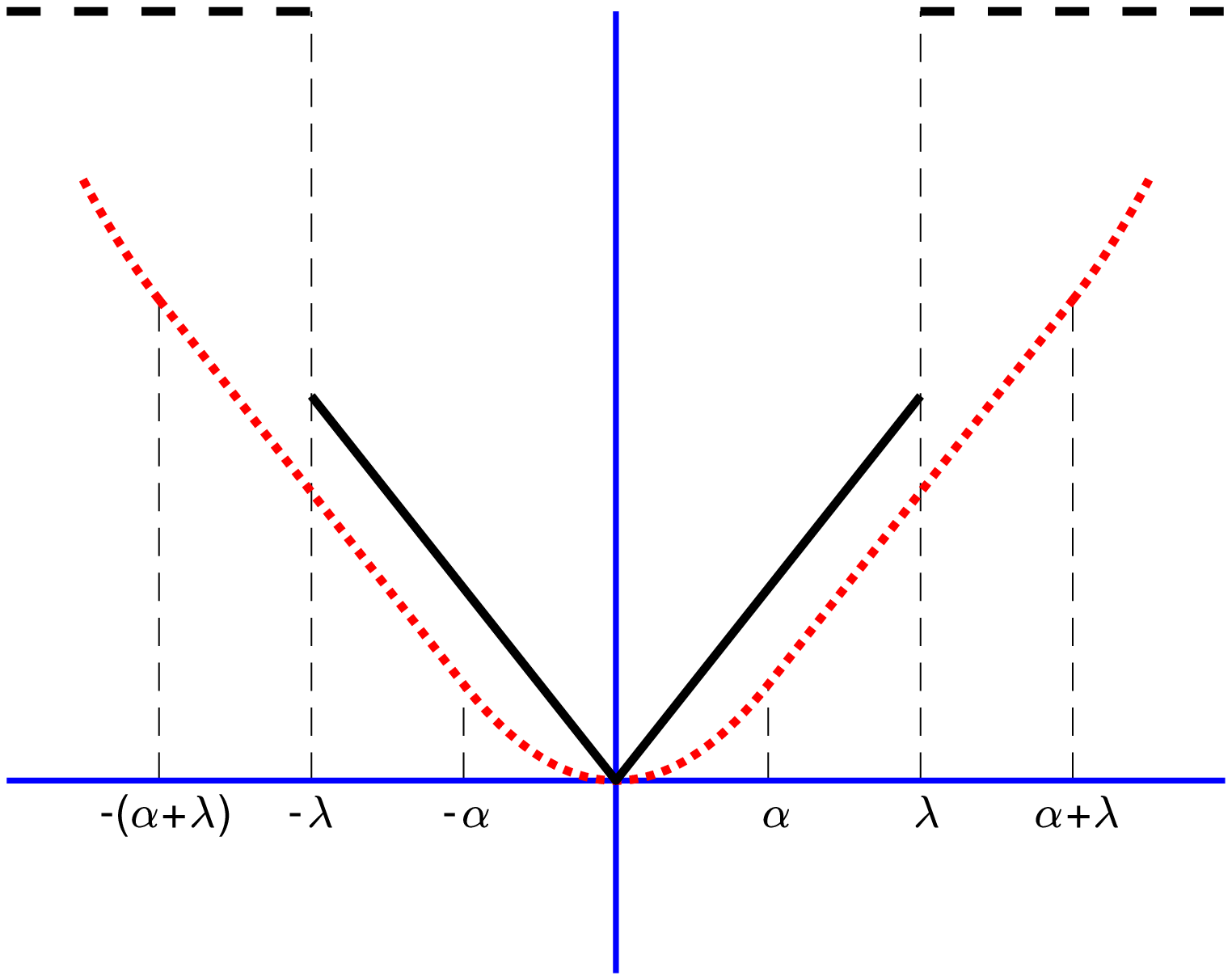} & \includegraphics[scale=0.35]{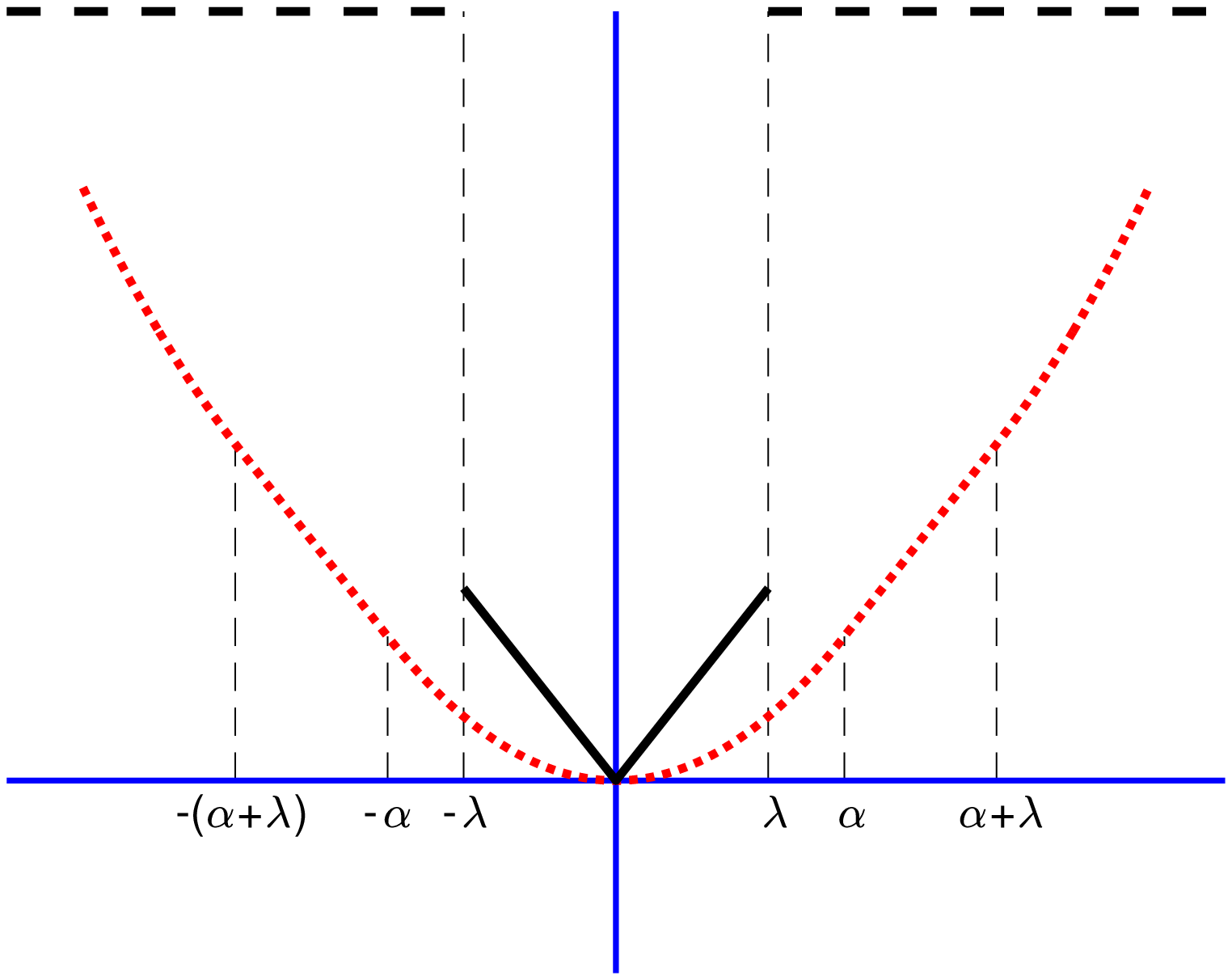}&\includegraphics[scale=0.35]{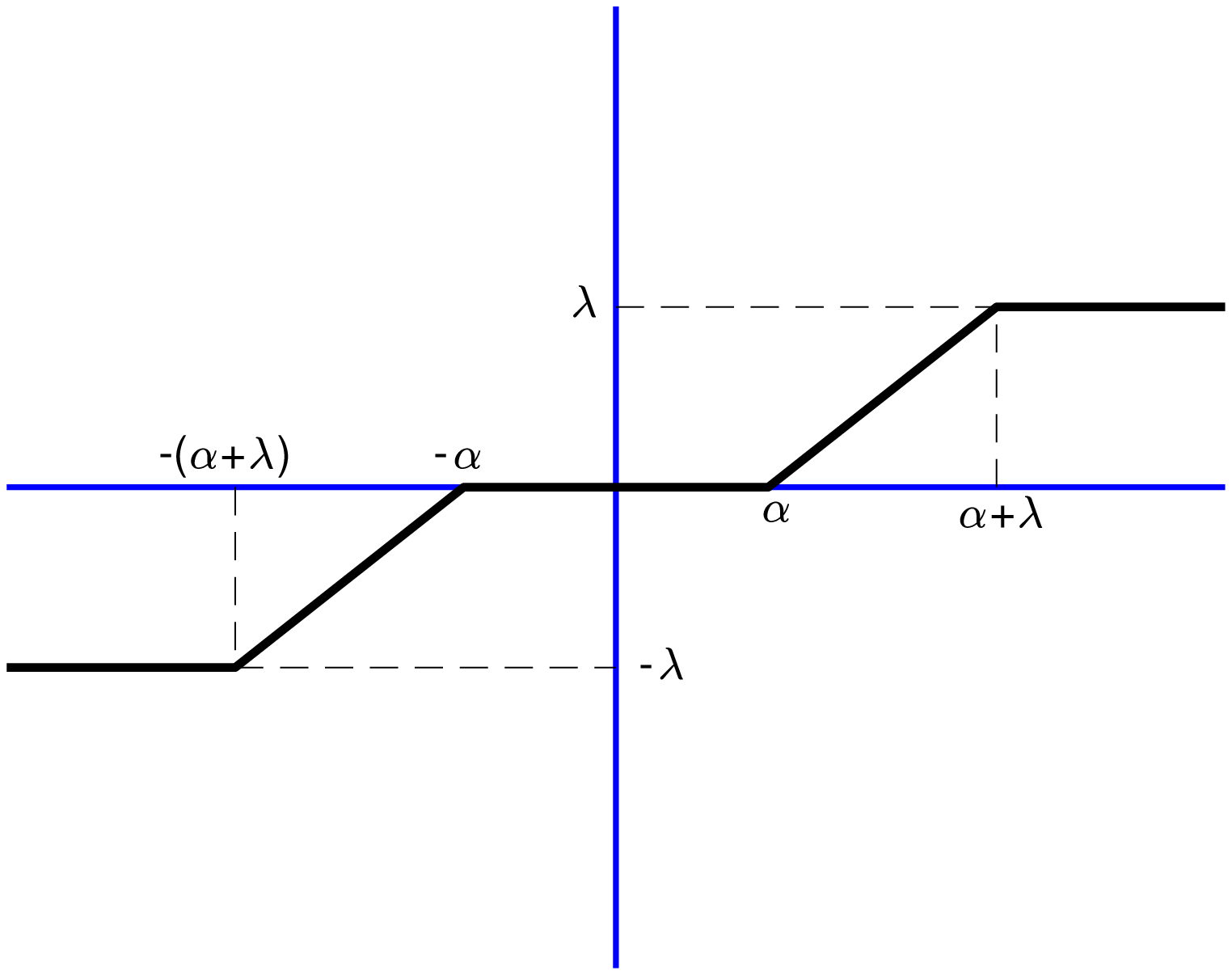}\\
(a) & (b) & (c)
\end{tabular}
\caption{Example 4. The graphs of $f$ (solid, dashed) and $\env_{\alpha f}$ (dotted) when (a) $\alpha < \lambda$ and (b) $\alpha > \lambda$. The graph of $\prox_{\alpha f}$ is shown in (c). Between $-(\alpha + \lambda)$ and $\alpha + \lambda$, $\prox_{\alpha f}$ is the soft thresholding operator with sparsity parameter $\alpha$; otherwise it projects onto this interval.}
\end{figure}
Figure~\ref{figure:ex3-f-prox} depicts the graphs of $f$, $\env_{\alpha f}$, and $\prox_{\alpha f}$. We observe that on the interval $[-\lambda, \lambda]$ (the domain of $f_\alpha$) the envelope $\env_{\alpha f}$ is piecewise quadratic polynomial (Figure~\ref{figure:ex3-f-prox}(a)) if $\alpha <\lambda$ and is simply quadratic polynomial (Figure~\ref{figure:ex3-f-prox}(b)) if $\alpha \ge \lambda$. It turns out that the expression of $\prox_{\beta f_\alpha}$ for  $\alpha <\lambda$ is much more complicated than that for $\alpha \ge \lambda$ as we will see below.

As both $f$ and $\env_{\alpha f}$ depend on $\alpha$ and $\lambda$, the explicit expression for $f_{\alpha}$ will depend on the values of these parameters. To compute the proximity operator $\prox_{\beta f_\alpha}$, we consider separately two main cases: $\alpha < \lambda$ and $\alpha \geq \lambda$.

\underline{Case 1: $\alpha < \lambda$.} In this case, we get (see Figure~\ref{figure:ex3:env-alpha<lambda})
\begin{equation}\label{tmp:1-falpha}
f_\alpha (x) = f(x) - \env_{\alpha f}(x) = \begin{cases} \frac{\alpha}{2} - \frac{1}{2\alpha}(|x| - \alpha)^2, & \text{ if $|x| \leq \alpha$;}\\
\frac{\alpha}{2}, & \text{ if $a \leq |x| \leq \lambda$;}\\
+ \infty, & \text{ if $\lambda < |x|$.}
\end{cases}
\end{equation}

\begin{figure}[h]\label{figure:ex3:env-alpha<lambda}
\centering
\includegraphics[scale=0.35]{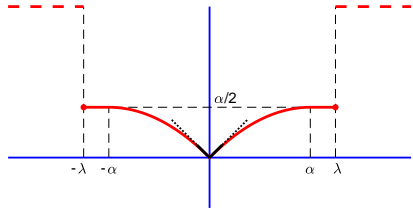}
\caption{Example 4. The graph of $f_\alpha$ when $\alpha < \lambda$ with the singularity of $f_\alpha$ at zero emphasized in black (solid-dotted). Further, we see that $f_\alpha$ agrees with Example 1 on $[- \lambda, \lambda ]$.}
\end{figure}

Depending on the values of $\alpha, \beta,$ and $\lambda$, we consider four possible cases: $\beta < \alpha < \lambda$, $\beta = \alpha < \lambda$, $\alpha < \beta \leq \lambda$, and $\lambda < \beta$.

\underline{Case 1.1: $\beta < \alpha < \lambda$.} In this case, we have
%$E_1(x,u)$ is strongly convex in $u$ with its vertex at $u = \frac{\alpha(x-\beta)}{\alpha - \beta}$. Therefore,
%\begin{equation}\label{1.1E1}
%\min_{u \in [0, \alpha]} E_1(x,u) = \begin{cases} E_1(x,0), & \text{ if $0 \leq x \leq \beta$;} \\
%E_1(x,\frac{\alpha(x-\beta)}{\alpha - \beta}), & \text{ if $\beta \leq x \leq \alpha$;} \\
%E_1(x,\alpha), & \text{ if $\alpha < x$.}
%\end{cases}
%\end{equation}
%From \eqref{1.1E1} and \eqref{1E2min}, it follows
\begin{equation}\label{1.1prox}
\prox_{\beta f_{\alpha}} (x) = \begin{cases}
\max \{0, \frac{\alpha(|x| - \beta)}{\alpha - \beta}\}\sign(x), & \text{ if $|x| \leq \alpha$;}\\
\min\{ |x|, \lambda\} \sign(x), & \text{ if $|x|>\alpha$.}
\end{cases}
\end{equation}

\underline{Case 1.2: $\beta = \alpha < \lambda$.} In this case, we have
%$E_1(x,u)$ is convex (but not strongly convex) in $u$, and $E_1(\alpha,u) = \frac{\alpha}{2}$. So we see that
%\begin{equation}\label{1.2E1}
%\min_{u \in [0, \alpha]} E_1(x,u) = \begin{cases}
%E_1(x,0), & \text{ if $0 \leq x \leq \alpha$;} \\
%E_1(\alpha,v), & \text{ for all $0 \leq v \leq \alpha$;}\\
%E_1(x,\alpha), & \text{ if $\alpha < x$}.
%\end{cases}
%\end{equation}
%Comparing the values in \eqref{1.2E1} and\eqref{1E2min} and keeping in mind that the minimizer of \eqref{1E} is in $\prox_{\beta f_\alpha}(x)$, we see that
\begin{equation}\label{1.2prox}
\prox_{\beta f_\alpha}(x) = \begin{cases}
0, & \text{ if $|x| < \alpha$;}\\
\sign(x)[0, \alpha], & \text{ if $|x| = \alpha$;}\\
\sign(x)\min\{|x|, \lambda\}, & \text{ if $\alpha < |x|$,}
\end{cases}
\end{equation}

\underline{Case 1.3: $\alpha < \beta \leq \lambda$.}  In this case, we have
%Regardless of $\lambda$, $E_1(x,u)$ is concave in $u$ whenever $\beta < \alpha$. Therefore the vertex $\frac{\alpha(x-\beta)}{\alpha - \beta}$ is  now its maximizer, and the minimum values will occur at the endpoints. If the vertex is greater than $\frac{\alpha}{2}$, the minimum will be at $0$. If the vertex is less than $\frac{\alpha}{2}$, the minimum will be at $\alpha$. It is easy to see that if the vertex is at $\frac{\alpha}{2}$, $E_1(x,0) = E_1(x,\alpha)$. Some simple computations reveal that $\frac{\alpha(x-\beta)}{\alpha - \beta} = \frac{\alpha}{2}$ when $x = \frac{1}{2}(\alpha+\beta)$. Therefore,
%\begin{equation}\label{1.3E1min}
%\min_{u \in [0, \alpha]} E_1(x,u) = \begin{cases}
%E_1(x,0), & \text{ if $0 \leq x < \frac{1}{2}(\alpha + \beta)$;}\\
%E_1(\frac{1}{2}(\alpha + \beta),x), & \text{ if $v \in \{0, \alpha\}$;}\\
%E_1(x,\alpha), & \text{ if $\frac{1}{2}(\alpha + \beta) < x$.}
%\end{cases}
%\end{equation}
%
%Now since $\frac{1}{2}(\alpha + \beta) > \alpha$, we must compare $E_1(x,0) = \frac{1}{2\beta}x^2$ with $E_2(x,\alpha) = \frac{\alpha}{2} + \frac{\alpha^2}{2\beta} - \frac{\alpha}{\beta}x + \frac{1}{2\beta}x^2$ and $E_2(x, x) = \frac{\alpha}{2}$. For $x \in [0, \alpha]$, $E_1(x,0) < E_2(x,\alpha)$. We see that $\frac{\alpha}{2} \leq \frac{1}{2\beta} x^2$ when $x \geq \sqrt{\alpha \beta}$, so we have
\begin{equation}\label{1.3prox}
\prox_{\beta f_\alpha}(x) = \begin{cases}
0, & \text{ if $|x| < \sqrt{\alpha \beta}$;}\\
\{ 0, \sign(x) \sqrt{\alpha \beta}\}, & \text{ if $|x| = \sqrt{\alpha \beta}$;}\\
\min \{ |x|, \lambda \} \sign(x), & \text{ if $\sqrt{\alpha \beta} < |x|$,}
\end{cases}
\end{equation}

\underline{Case 1.4: $\alpha < \lambda < \beta$.} We have
%As in the previous case $E_1$ is concave, and we recover \eqref{1.3E1min}. However, we must now take care around the switching point $\sqrt{\alpha \beta}$.
%
%When $\sqrt{\alpha \beta} \leq \lambda$, the objective function is exactly as before, and we have \eqref{1.3prox}.
%
%When $\lambda < \sqrt{\alpha \beta}$, we must compare $E_1(x,0)$ with $E_2(x,\lambda)$. In this case, we see that $E_1(x,0) = \frac{1}{2\beta}x^2 \leq E_2(x,\lambda) = \frac{\alpha}{2} + \frac{1}{2\beta}(\lambda - x)^2$ when $x \leq \frac{\alpha \beta + \lambda^2}{2\lambda}$. So we see that if $\lambda < \sqrt{\alpha \beta}$,
\begin{equation}\label{1.4prox}
\prox_{\beta f_\alpha}(x) = \begin{cases}
\{0 \}, & \text{ if $|x| < \frac{\alpha \beta + \lambda^2}{2\lambda}$;}\\
\{0, \lambda\sign(x)\}, & \text{ if $|x| = \frac{\alpha \beta + \lambda^2}{2\lambda}$;} \\
\{\lambda\sign(x)\}, & \text{ if $\frac{\alpha \beta + \lambda^2}{2\lambda} < |x|$,}
\end{cases}
\end{equation}

We now move on to the second main case.

\underline{Case 2: $\lambda \leq \alpha$.} In this case, we get (see Figure~\ref{figure:ex3:env-alpha>=lambda})
$$f_\alpha(x) = \begin{cases}
\frac{\alpha}{2} - \frac{1}{2\alpha}(|x| - \alpha)^2, & \text{ if $|x| \leq \lambda$;}\\
+ \infty, & \text{ otherwise.}
\end{cases}$$

\begin{figure}[h] \label{figure:ex3:env-alpha>=lambda}
\centering
\includegraphics[scale=0.35]{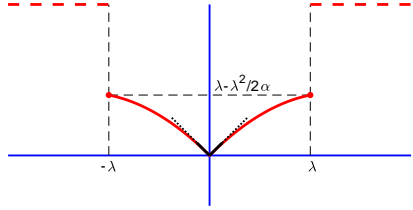}
\caption{Example 4. The graph of $f_\alpha$ when $\lambda \leq \alpha$ with the singularity of $f_\alpha$ at zero emphasized in black (solid-dotted). As before, $f_\alpha$ agrees with Example 1 on $[-\lambda, \lambda]$, but is cut off before it plateaus.}
\end{figure}
%As before, $f_\alpha$ is even and nondecreasing on $[0, \infty)$. Therefore without loss of generality, we consider only $x \geq 0$. For a given $x$, we seek to minimize the objective function
%$$E(u, x) \coloneqq \frac{\alpha}{2} - \frac{1}{2\alpha}(u-\alpha)^2 + \frac{1}{2\beta}(u- x)^2$$
%over $u \in [0, \lambda]$.

To compute $\prox_{\beta f_\alpha}$, we consider three situations: $\beta < \alpha$, $\beta = \alpha$, and $\beta > \alpha$.

\underline{Case 2.1: $\beta < \alpha$.} In this case, we have that
\begin{equation}\label{2.1prox}
\prox_{\beta f_\alpha} (x) = \begin{cases}
0, & \text{ if $|x| \leq \beta$;}\\
\frac{\alpha(|x| - \beta)}{\alpha-\beta}\sign(x), & \text{ if $\beta \leq |x| \leq \beta + \frac{\alpha-\beta}{\alpha}\lambda$;}\\
\lambda \sign(x), & \text{ if $\beta + \frac{\alpha - \beta}{\alpha}\lambda \leq |x|,$}
\end{cases}
\end{equation}

\underline{Case 2.2: $\beta = \alpha$.} In this case, we have
\begin{equation}\label{2.2prox}
\prox_{\beta f_\alpha}(x) = \begin{cases} 0, & \text{ if $|x| < \alpha$;}\\
\sign(x)[0, \lambda], & \text{ if $|x| = \alpha$;}\\
\lambda\sign(x), & \text{ if $\alpha < |x|$,}
\end{cases}
\end{equation}

\underline{Case 2.3: $\beta > \alpha$.} Similar to Case 1.4, we get

\begin{equation}\label{2.3prox}
\prox_{\beta f_\alpha}(x) = \begin{cases}
0, & \text{ if $|x| \leq \beta - \frac{\beta - \alpha}{2\alpha}\lambda$;}\\
\sign(x)\{0, \lambda\}, & \text{ if $|x| = \beta - \frac{\beta -\alpha}{2\alpha}\lambda$;}\\
\lambda \sign(x), & \text{ if $\beta - \frac{\beta - \alpha}{2\alpha}\lambda < |x|$,}
\end{cases}
\end{equation}

\begin{figure}[h]
\centering
\begin{tabular}{ccc}
\includegraphics[scale=0.35]{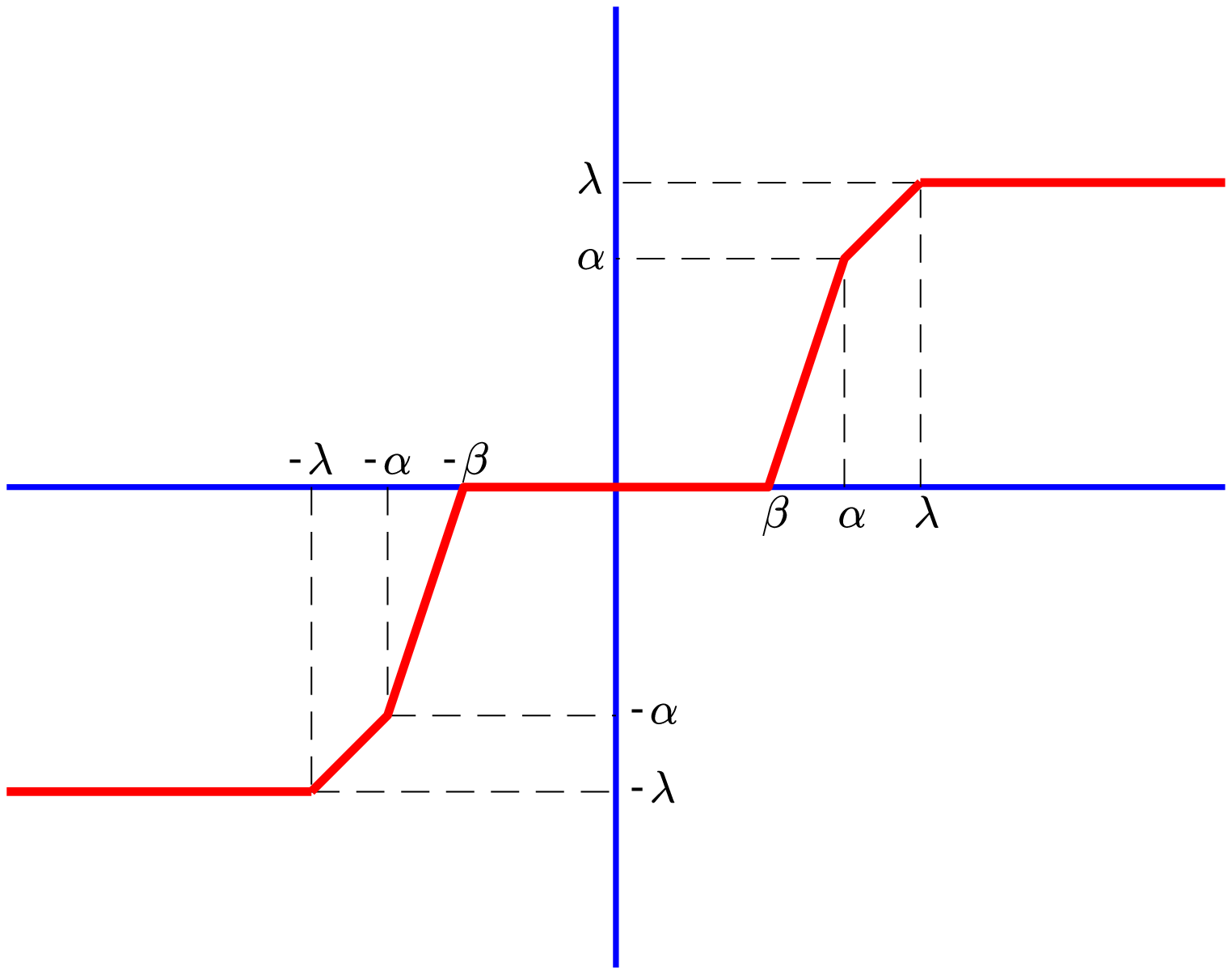} & \includegraphics[scale=0.35]{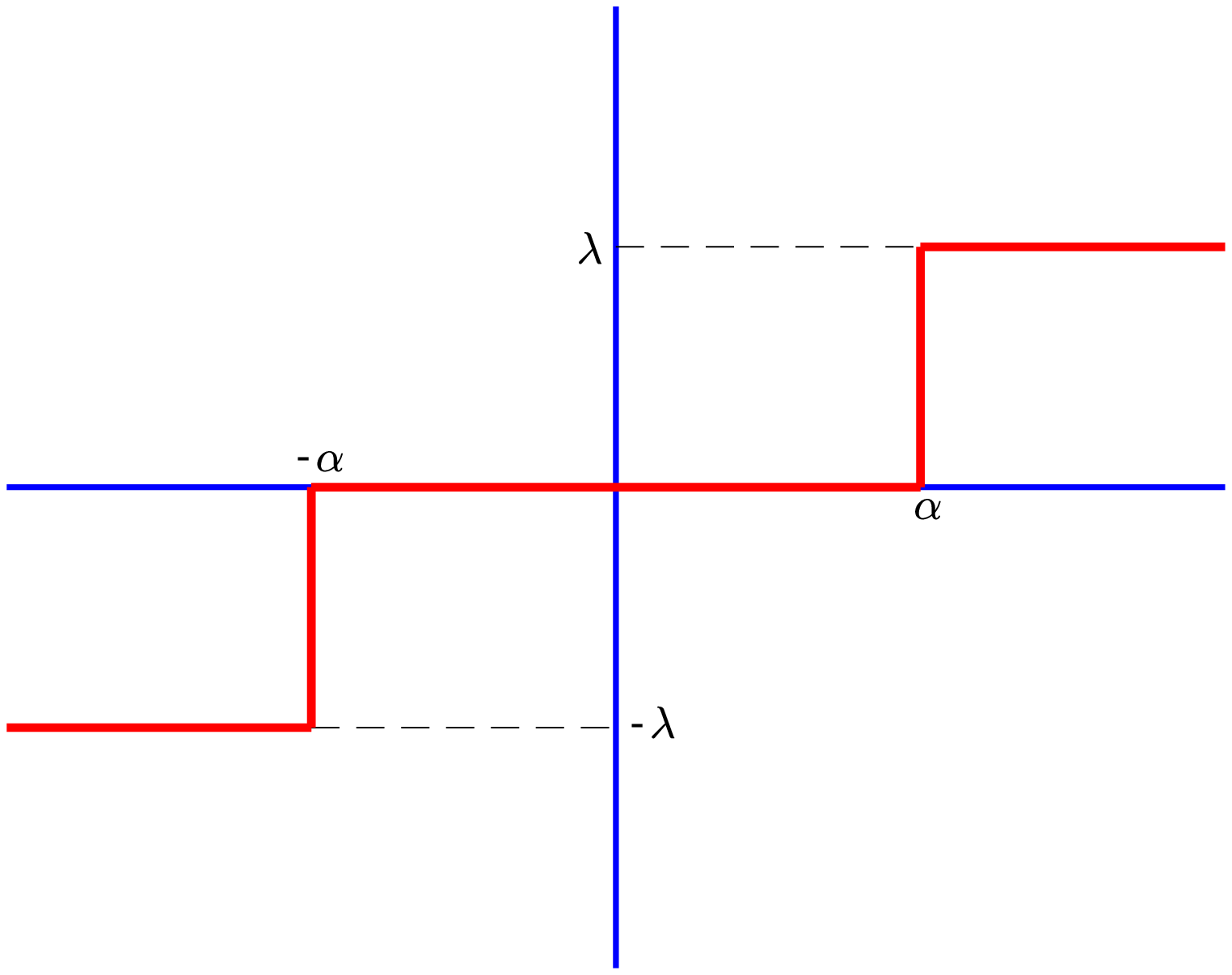}&\includegraphics[scale=0.35]{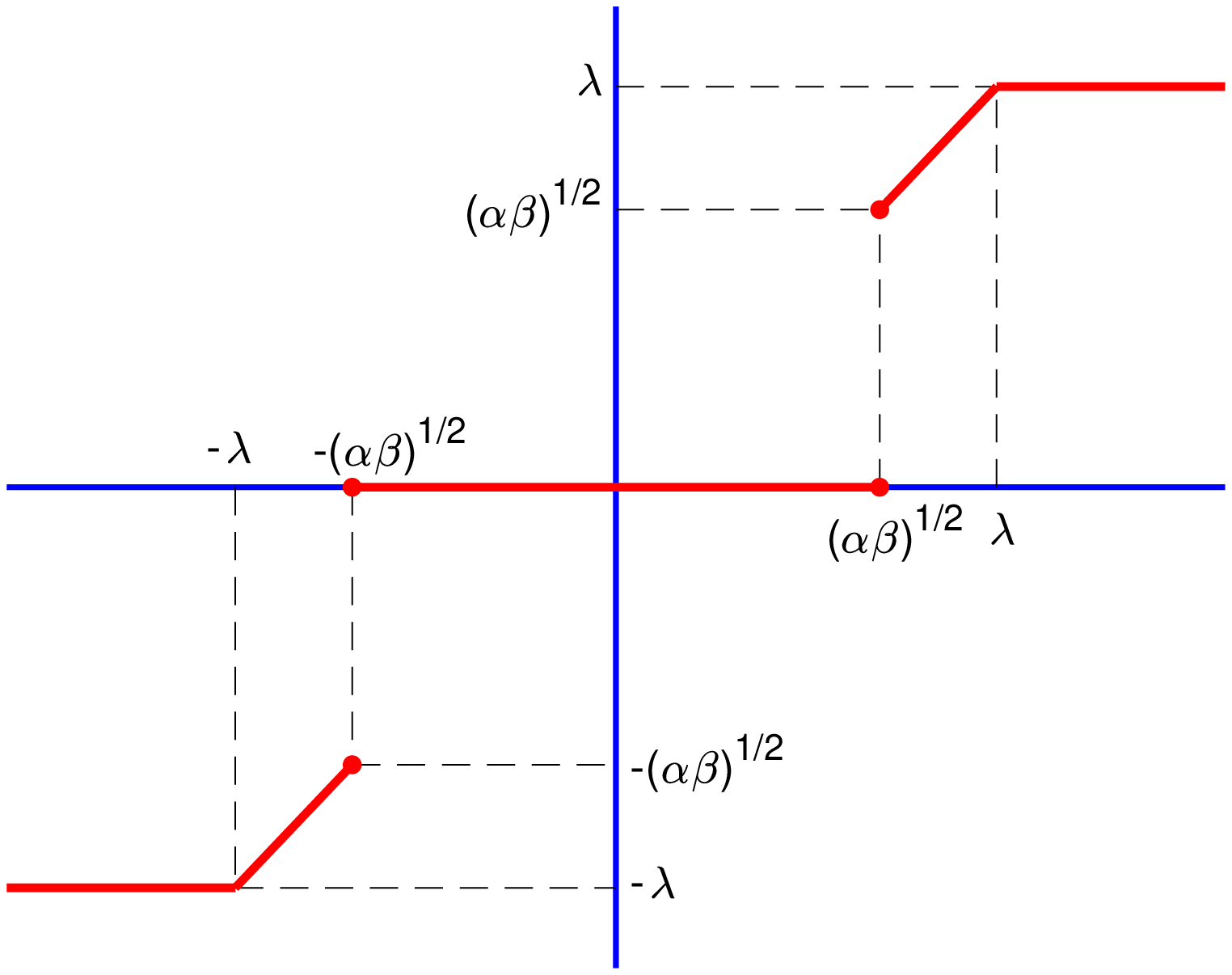}\\
(a) & (b) & (c)
\end{tabular}
\caption{Example 4. Typical shapes of $\prox_{\beta f_\alpha}$ in (a) Case 1.1: $\beta < \alpha < \lambda$, (b) Case 2.2: $\beta = \alpha \geq \lambda$, and (c) Case 1.3: $\alpha < \beta \leq \lambda$. In each case, we see that the absolute threshold is $\lambda$, while the sparsity threshold and thresholding behavior depend on $\alpha$ and $\beta$. }
\end{figure}

To end up this example, we comment on this example in a comparison with Theorem~\ref{thm:major}.
\begin{itemize}
\item Note that $\partial f(0)=[-1,1]$. For $\beta<\alpha$, both equations \eqref{1.1prox} and \eqref{2.1prox} show that the operator $\prox_{\beta f_\alpha}$ vanishes all elements in $\beta \partial f(0)=[-\beta, \beta]$ as required by the first statement of Theorem~\ref{thm:major}.

\item For $\beta=\alpha$, both equations \eqref{1.2prox} and \eqref{2.2prox} show that the operator $\prox_{\beta f_\alpha}$ vanishes all elements in $\mathrm{ri}(\alpha \partial f(0))=(-\alpha, \alpha)$ as described in the second statement of Theorem~\ref{thm:major}.

\item For $\beta>\alpha$, since $\sqrt{\alpha \beta} > \alpha$, $\frac{\alpha\beta+\lambda^2}{2\lambda}>\alpha$ when $\alpha<\lambda <\beta$, and $\beta-\frac{\beta-\alpha}{2\alpha}\lambda \ge \alpha$ when $\alpha \ge \lambda$, then equations \eqref{1.3prox}, \eqref{1.3prox},  and \eqref{2.3prox} shows that the operator $\prox_{\beta f_\alpha}$ vanishes all elements in $\mathrm{ri}(\alpha \partial f(0))=(-\alpha, \alpha)$ as described in the third statement of Theorem~\ref{thm:major}.
\end{itemize}

To close this section, Table~\ref{table:prox} lists the proximity operators $\prox_{\beta f_\alpha}$ of all examples.
\begin{table}[htb]\normalsize
\caption{Proximity operators for all examples}\label{table:prox}
\begin{center}
\begin{tabular}{|c|cc|cc|ccc|}
\hline
\rule{0pt}{4ex}    Function& \multicolumn{2}{c|}{$\beta < \alpha$}&  \multicolumn{2}{c|}{$\beta = \alpha$} & \multicolumn{3}{c|}{$\beta > \alpha$} \B \\
%\cline{2-3} \cline{4-5} \cline{6-8}
\hline

\rule{0pt}{3ex}     $f(x)=|x|$ & \multicolumn{2}{c|}{\eqref{eq:beta<alpha}}  & \multicolumn{2}{c|}{\eqref{eq:beta=alpha}}  & \multicolumn{3}{c|}{\eqref{eq:beta>alpha}} \T\B \\
\hline

\rule{0pt}{3ex}    $f(x)=\max\{0,x\}$ & \multicolumn{2}{c|}{\eqref{ex2:beta<alpha}}  & \multicolumn{2}{c|}{\eqref{ex2:beta=alpha}}  & \multicolumn{3}{c|}{\eqref{ex2:beta>alpha}} \T\B  \\
\hline

\rule{0pt}{3ex}    & \multicolumn{2}{c|}{}  & \multicolumn{2}{c|}{}&$\beta<\alpha(\beta+1)$ &$\beta=\alpha(\beta+1)$&$\beta>\alpha(\beta+1)$ \T\B\\

\rule{0pt}{3ex}    $f(x)=\frac{1}{2}x^2+|x|$  & \multicolumn{2}{c|}{\eqref{ex4:beta<=alpha}}  & \multicolumn{2}{c|}{\eqref{ex4:beta<=alpha}}&\eqref{ex4:beta<=alpha} &\eqref{ex4:beta>alpha1} &\eqref{ex4:beta>alpha2} \T\B \\
\hline

\rule{0pt}{3ex}    &$\alpha<\lambda$& $\alpha \ge\lambda$&  $\alpha<\lambda$& $\alpha \ge\lambda$ & $\beta\le\lambda$&$\beta>\lambda$& $\alpha \ge\lambda$ \T\B \\
$f(x)=|x|+\iota_{[-\lambda,\lambda]}$ &\eqref{1.1prox}&\eqref{2.1prox}&\eqref{1.2prox}&\eqref{2.2prox}&\eqref{1.3prox}&\eqref{1.4prox}&\eqref{2.3prox} \T\B \\
\hline

\end{tabular}
\end{center}    `
\end{table}

%%%%%%%%%%%%%%%%%%%%%%%%%%%%%%%%%%%%%%%%%%%%%%%%%%%%%%%%%%%%%%%%
\section{Conclusions}\label{sec:concluions}
We presented a simple scheme to construct a family of semiconvex structured sparsity promoting functions from any convex sparsity promoting function. Theoretical guarantees of sparsity promotion were proved in Section~\ref{sec:NSPF}, among other properties related to the structure of these functions. In Section~\ref{sec:Special}, we expanded upon these results in the case of indicator and piecewise quadratic functions. We demonstrated that the classical MCP can be derived under this framework, while also providing several other examples motivated by a variety of applications.

Because of the structure of the proposed functions, we can use convex, nonconvex, and difference of convex algorithms in practice. We plan on testing these examples on problems such as signal denoising and variable selection. Furthermore, we hope to use the unique properties of these functions to develop new algorithms. Other future work will also expand upon the theoretical properties of these functions.
%\section*{Acknowledgment of Support and Disclaimer}
\section*{Disclaimer and Acknowledgment of Support}
Any opinions, findings and conclusions or recommendations expressed in this material are those of the authors
and do not necessarily reflect the views of AFRL (Air Force Research Laboratory).

Lixin Shen is partially supported by the US National Science Foundation under grant DMS-1522332.

\bibliographystyle{siam}
%\bibliography{C:/Users/lshen03/Dropbox/shen}

\end{document}